%% file: MRB2.tex
\documentclass[mathpazo]{cicp}

\usepackage[T1]{fontenc}
\usepackage{newtxtext}
\usepackage{color}
\usepackage{mathtools}
\usepackage {cleveref}

\usepackage[numbers]{natbib}
\usepackage{subfigure}
\graphicspath{{Fig/}}

\numberwithin{equation}{section}

\newcounter{assumption}
\newtheorem{assumption}[assumption]{Assumption}

\newcommand{\nn}{\nonumber} 
\def\refe#1{(\ref{#1})}

\newcommand\bs[1]{\boldsymbol{#1}}

\begin{document}

\title[Second-order Decoupled FEM for Micropolar RBC System
]{An efficient fully decoupled finite element method with second-order accuracy for the micropolar Rayleigh-B{\'e}nard convection system}

\author[M. Cui, A. Hou, X. Dong]{Ming Cui, Akang Hou, and Xiaoyu Dong\corrauth}
\address{Department of Mathematics, 
School of Mathematics, Statistics and Mechanics, \\
Beijing University of Technology,
          Beijing 100124, P.R. China}
\emails{{\tt mingcui@bjut.edu.cn} (M.~Cui), {\tt houakang@emails.bjut.edu.cn} (A.~Hou), {\tt dongxy@bjut.edu.cn} (X.~Dong)}

\begin{abstract}
The micropolar Rayleigh-B{\'e}nard convection system, which consists of Navier-Stokes equations, the angular momentum equation, and the heat equation, is a strongly nonlinear, coupled, and saddle point structural multiphysics system. A second-order pressure projection finite element method, which is linear, fully decoupled, and second-order accurate in time, is proposed to simulate the system. Only a few decoupled linear elliptic problems with constant coefficients are solved at each time step, simplifying calculations significantly. The stability analysis of the method is established and the optimal error estimates are derived rigorously with the negative norm technique. Extensive numerical simulations, including 2D and 3D accuracy tests, the lid-driven cavity flow, and the passive-scalar mixing experiment, are carried out to illustrate the effectiveness of the method.
\end{abstract}

\ams{65M12, 65M60, 76M10} 
\keywords{Micropolar Rayleigh-B{\'e}nard convection system, projection method, fully decoupling, unconditional stability, optimal error estimates.}
\maketitle

\input{Introduction}
\input{Model}

\input{NumericalScheme}

\input{Error}
\input{Examples}

\input{Conclusion}

\section*{Acknowledgments}
Ming Cui's work was supported by National Natural Science Foundation of China No. 12371368. Xiaoyu Dong's work was supported by the Postdoctoral Fellowship Program of CPSF under Grant No. GZC20251999 and the Beijing Postdoctoral Research Foundation under Grant No. 2025-ZZ-34.

\bibliographystyle{unsrtnat}
\bibliography{MRB2}

\end{document}

%% file: Introduction.tex
\section{Introduction}

The micropolar fluid model, originally introduced by Eringen in 1966 \cite{Eringen1966}, belongs to a class of non-Newtonian fluid theories accounting for microstructural effects and local particle rotations, which is a significant generalization of the incompressible Navier-Stokes model of classical hydrodynamics. A key feature of this model is the interaction between the rotational motion of microscopic constituents and the macroscopic velocity field. In narrow channels, the differences between the micropolar fluid model and the Navier-Stokes model become more pronounced \cite{Bahouri2011, Beale1984}. Micropolar theory effectively describes the dynamics of complex fluids with microstructural characteristics, substantially extending the range of physical phenomena captured by conventional continuum theories, such as the flow behavior of biological fluids, liquid crystals and dilute aqueous polymer solutions \cite{Popel1974}. In recent years, many scholars have studied the dynamics of fluid layers between two rigid boundaries heated from below, which is a basic problem in thermal convection. To capture thermally driven flow accurately,  in this work, we focus on a mathematical model based on the Boussinesq approximation which describes the heat convection in a micropolar fluid, which is called a micropolar Rayleigh-B{\'e}nard convection system \cite{Foias1987, Kalita2019, Tarasinska2006, Xu2020}. 

The micropolar Rayleigh-B{\'e}nard convection system provides a more accurate description of microscale flow and angular momentum transport driven by temperature gradients, enhancing the thermal control efficiency and system design. Therefore, it has significant applications in engineering fields such as microelectronic cooling, biomedical hyperthermia, ferrofluid-based thermal management, and polymer processing  \cite{nanofluid2020review, Turkyilmazoglu2014, Tarasinska2006, Rana2012}. Many problems in engineering, especially those involving coupled heat transfer, often exhibit strongly nonlinear characteristics. Although certain cases can be solved analytically, most of these issues require numerical methods such as the perturbation method, the homotopy perturbation method, the variational iteration approach \cite{He1999}, and the finite difference method \cite{AydnaPopb2007}. At present, few numerical techniques are available for the micropolar Rayleigh-B{\'e}nard convection system. The finite element method is employed to discrete the highly linear  micropolar Rayleigh-B{\'e}nard convection system in this work. To overcome the difficulty due to the  saddle point structure formed by velocity and pressure, we adpot a  pressure projection technique, which was first introduced by Chorin \cite{Chorin1968} and was confirmed to be an effective approach to overcome such challenges in the late 1960s \cite{Temam1969}. This method is based on a rather peculiar time-discretization of the equations governing viscous incompressible flows, in which the viscosity and the incompressibility of the fluid are treated within two separate steps.  As shown in \cite{Rannacher1992}, the projection algorithm can also be interpreted as a pressure stabilization technique. In practice, this technique is compatible with various spatial discretizations, including spectral approximations \cite{Ku1987}  and finite differences \cite{Bell1989}. Currently, the hybrid scheme of this approach with the finite element method has been widely used by researchers to address various types of problems. Guermond \cite{Guermond1996} investigated a div-grad problem, employing velocity test functions that satisfy various boundary conditions, and supplemented the formulation for a Poisson equation subject to a Neumann boundary condition. Furthermore, a fractional-step projection method for incompressible viscous flows using finite element approximations was analyzed in \cite{GuermondQuartapelle1998} which established finite-time error estimates and illustrated the flexibility of the method through examples involving nonstandard boundary conditions. Then a second-order pressure projection finite element scheme was proposed for the Navier-Stokes equations in \cite{Guermond1999}, and the optimal error estimate was provided. A series of studies in \cite{ZhangHeYang2021, ZhangHeYang2023, ZhangHeYang2024} applied the projection method with the finite element method to two-phase FHD Shliomis models.  For more information about the application of the projection method, please refer to \cite{Pyo2009, Pyo2013, SiWangWang2022, GuermondShen2003, SiLuWang2022, Xu2020}. For the existence of solutions, the global regularity for the 2D micropolar Rayleigh-B{\'e}nard convection system with zero diffusivity was given in \cite{XuQiaoZhang2021} and the well-posedness of the 2D micropolar Rayleigh-B{\'e}nard convection problem was established in \cite{YuanLi2024}. In \cite{Galdi2011}, the global regularity for the 2D micropolar Rayleigh-B{\'e}nard convection system is provided with zero velocity dissipation and critical temperature diffusion. These theoretical researches become excellent foundations for numerical schemes to solve the micropolar Rayleigh-B{\'e}nard convection system.

The rest of this paper is organized as follows. We first introduce the micropolar Rayleigh-B{\'e}nard convection system and preliminaries in Section 2, and then develop a second-order pressure projection method to solve the micropolar Rayleigh-B{\'e}nard convection system and prove its unconditional stability in Section 3. In Section 4, the optimal error estimates of the numerical scheme are presented. Numerical tests are carried out in Section 5 to verify the effectiveness of the method.

%% file: Model.tex
\section{The micropolar Rayleigh-B{\'e}nard convection system}

The micropolar Rayleigh-B{\'e}nard convection system is formulated as
\begin{eqnarray}\label{1.1}
 	\left\{\begin{array}{lll}
 	\bs{u}_t+(\bs{u}\cdot\nabla)\bs{u}-(\chi+\mu)\Delta \bs{u}+\nabla p=2\chi\nabla \times \bs{\omega}+\theta \bs{e},\\
 	\nabla \cdot \bs{u}=0,\\
 \zeta\bs{\omega}_t+\zeta(\bs{u}\cdot\nabla)\bs{\omega}-\upsilon\Delta\bs{\bs{\omega}}-\eta\nabla(\nabla\cdot \bs{\omega})+4\chi\bs{\omega}=2\chi\nabla \times \bs{u},\\
 	\theta_t+(\bs{u}\cdot\nabla)\theta-\kappa\Delta\theta=\bs{u}\cdot \bs{e}.\\
 	\end{array}\right.
\end{eqnarray}
The unknown physical variables are the velocity field $\bs{u}$, pressure $p$, the angular velocity field $\bs{\omega}$, and temperature $\theta$. The parameters $\chi, \mu, \upsilon, \eta, \zeta$ and $\kappa$ are the Newtonian kinematic viscosity, the micro-rotation viscosity, the angular viscosity, the inertia density, and the thermal diffusivity, respectively. $\bs{e}$\ is the unit vector where $\bs{e}=(0, 1)$ in 2D and $\bs{e}=(0, 0, 1)$ in 3D, $\theta \bs{e}$ describes the action of the buoyancy force on fluid motion and $\bs{u}\cdot \bs{e}$  denotes the Rayleigh-B{\'e}nard  convection in a heated inviscid fluid. 

For convenience, the symbols and function spaces used in the following are defined. The inner product of $L^2$ is defined as
$(\bs{a}, \bs{b}) = \int_{\Omega} \bs{a} \cdot \bs{b} \, \mathrm{d} \bs{x}$,
and the corresponding $L^2$-norm,   $H^1$-norm and $H^2$-norm are denoted by $\|\bs{a}\|_{0}$, $\|\bs{a}\|_{1}$ and $\|\bs{a}\|_{2}$, respectively. The Sobolev spaces are introduced by \cite{BrezziFortin1991, BrennerScott1994, GiraultRaviart1986}
\begin{align*}
    &\bs{X} :=  H_0^1(\Omega)^d, \quad d=2,3, \\
    &\bs{V} := \left\{ \bs{v} \in \bs{X} : \nabla \cdot \bs{v} = 0,\ \bs{v}|_{\partial\Omega} = \bs{0} \right\}, \\
    &\bs{V}_1 := \left\{ \bs{v} \in L^2(\Omega)^d : \nabla \cdot\bs{v} = 0,\ \bs{v} \cdot \bs{n}|_{\partial\Omega} = 0 \right\}, \quad d=2,3\\
    &M := L_0^2(\Omega) = \left\{ \varphi \in L^2(\Omega) : \int_{\Omega} \varphi \, \mathrm{d}\bs{x} = 0 \right\}.
\end{align*}
   \begin{lemma}  [div-grad relation inequality \cite{Pyo2009, Pyo2013}] \label{2.0101'}
   Let 
$\Omega$ be a bounded Lipschitz domain such that for all $\bs{v}\in X$,  $$\|\nabla\cdot \bs{v}\|_{0}\leqslant \|\nabla \bs{v}\|_{0}.$$ 
   \end{lemma}
   
The Stokes operator is defined as $\mathcal{A}: \bs{D}(\mathcal{A}) \to \bs{V}_{1}$, such that $\mathcal{A}=-P\Delta$, where $\bs{D}(\mathcal{A})=H^{2}(\Omega)^{d}\cap \bs{V}$ and the $L^{2}$-projection $P: L^{2}(\Omega)^{d}\to \bs{V}_{1}.$ The following inequalities are presented in \cite{HeLi2009, He2015, SiWangWang2022}, which will be used frequently:
\begin{alignat*}{2}
 &\|\bs{v}\|_{L^{r}}\leqslant c(r)\|\nabla \bs{v}\|_{0},
    \qquad r\in[1,6], \qquad \forall \bs{v}\in \bs{X},
 \\[8pt]
 &\|\bs{v}\|_{L^{\infty}}+\|\nabla \bs{v}\|_{L^{3}}
    \leqslant C\|\bs{v}\|_{1}^{\frac{1}{2}} \|\mathcal{A}\bs{v}\|_{0}^{\frac{1}{2}},
    \qquad
    \|\bs{v}\|_{2}\leqslant C\|\mathcal{A}\bs{v}\|_{0},
    \qquad \forall \bs{v}\in \bs{D}(\mathcal{A}).
\end{alignat*}
Let the bilinear forms be
\begin{equation*}
a(\bs{u}, \bs{v}) = (\nabla \bs{u}, \nabla \bs{v}), \quad
d(\bs{v}, p) = (\nabla \cdot \bs{v}, p), \quad
\bs{B}(\bs{u}, \bs{v}) = (\bs{u} \cdot \nabla)\bs{v} + \frac{1}{2}(\nabla \cdot \bs{u})\bs{v},
\quad  \forall\, \bs{u}, \bs{v} \in \bs{X}, \ p \in M,
\end{equation*}
where $\bs{B}(\bs{u}, \bs{v})$ is the skew-symmetric form of the nonlinear convective terms. Assuming that the velocity field $\bs{u}$ is divergence-free, the nonlinear terms reduce to $\bs{B}(\bs{u}, \bs{v})=(\bs{u}\cdot\nabla)\bs{v}$. We define a trilinear form as
\begin{equation*}
   b(\bs{u}, \bs{v}, \bs{w})= (\bs{B}( \bs{u}, \bs{v}), \bs{w})=((\bs{u}\cdot\nabla)\bs{v},\bs{w})+\frac{1}{2}((\nabla\cdot\bs{u})\bs{v},\bs{w})
   = \frac{1}{2} ((\bs{u}\cdot\nabla)\bs{v},\bs{w}) -  \frac{1}{2} ((\bs{u}\cdot\nabla)\bs{w},\bs{v}).
\end{equation*}
 
To ensure the well-posedness and stability of the mixed formulation, 
the velocity–pressure pair $(\bs{X}, M)$ must satisfy the classical 
\textit{inf--sup} condition \cite{Temam1984}. Specifically, it requires that there exists a positive constant $\beta_0$, 
independent of the mesh size, such that 
\begin{equation*}
\exists\, \beta_{0}>0, \ \forall q\in M, \quad  
\beta_{0}\|q\|_{0}\leqslant 
\sup_{\bs{v}\in \bs{X},\ \bs{v}\neq \bs{0}}
\frac{d(\bs{v}, q)}{\|\nabla\bs{v}\|_{0}}.
\end{equation*}

 To derive the theoretical stability and error estimates of a numerical scheme, it is common to assume certain regularity for the exact solution and the initial values. These assumptions are reasonable since the initial values used in the numerical scheme are equal to those of the exact solution.

\begin{assumption}
[\cite{He2015, SiWangWang2022, LeiYangSi2018, SiLuWang2022}]
The initial values $\bs{u}_{0}, \bs{\omega}_{0}$ and $\theta_{0}$ satisfy$$\|\mathcal{A}\bs{u}_{0}\|_{0}+\|\bs{\omega}_{0}\|_{0}+\|\theta_{0}\|_{2}\leqslant C.$$
\end{assumption}

\begin{assumption}
[\cite{He2015, SiWangWang2022, LeiYangSi2018, SiLuWang2022}] 
Assuming that $(\bs{u},p,\bs{\omega}, \theta)$ is a unique local strong solution of the micropolar Rayleigh-B{\'e}nard convection system \refe{1.1} on $[0,T]$, there holds that
   \begin{align*}
       & \sup_{0 \leqslant t \leqslant T} (\|\mathcal{A}\bs{u}_{t}\|_{0}+\|\bs{\omega}_{t}\|_{0}+\|\theta\|_{2}
       +\|\bs{u}_{t}\|_{0}+\|\bs{\omega}_{t}\|_{0}+\|p_{t}\|_{1}+\|\mathcal{A}\bs{u}\|_{0}+\|\bs{\omega}\|_{0}+\|\nabla\bs{u}\|_{L^{\infty}}\nn\\
       & \quad +\|\nabla\bs{\omega}\|_{L^{\infty}}+\|\nabla\theta\|_{L^{\infty}})+\int_{0}^{T}(\|\bs{u}_{t}\|_{1}^{2}+\|\bs{\omega}_{t}\|_{1}^{2}+\|\theta_{t}\|_{1}^{2}
       +\|\bs{u}_{tt}\|_{0}^{2}+\|\bs{\omega}_{tt}\|_{0}^{2}+\|\theta_{tt}\|_{0}^{2}\nn\\
       & \quad +\|\bs{u}_{ttt}\|_{0}^{2} +\|\bs{\omega}_{ttt}\|_{0}^{2}+\|\theta_{ttt}\|_{0}^{2}
   +\|p_{t}\|_{1}^{2}) \,\mathrm{d}t\leqslant C.
   \end{align*}
   \end{assumption}
   
 \begin{lemma}[\cite{He2015}] \label{2.2'}
      Let $ a_n,b_n,c_n$ and $\gamma_n$, for integer $n\geqslant0$, be the
      nonnegative numbers such that
      \begin{align*}
      a_m+\delta_t\sum_{n=0}^m b_n\leqslant
      \delta_t\sum_{n=0}^m\gamma_n a_n+\delta_t\sum_{n=0}^m c_n+B, \quad \forall
      m\geqslant0.
      \end{align*}
      Suppose that $\delta_t\gamma_n<1$, for all $n$, and set
      $\sigma_n=(1-\delta_t\gamma_n)^{-1}$. Then
      \begin{align*}
      a_m+\delta_t\sum_{n=0}^mb_n\leqslant
      \exp(\delta_t\sum_{n=0}^m\gamma_n\sigma_n)(\delta_t\sum_{n=0}^mc_n+B),
      \quad\forall m\geqslant0.
      \end{align*}
      \end{lemma}

%% file: NumericalScheme.tex
\section{Numerical scheme and Stability analysis}
A pressure-projection finite element scheme, which is unconditionally stable, fully decoupled, linear, and second-order accurate in time, is constructed for \refe{1.1}. In this section, the numerical scheme and stability in fully discrete form are presented.

\subsection{Second-order pressure projection scheme}

Let $T_{h}$ be a family of regular and quasi-uniform triangulation partition of $\Omega$, consisting of tetrahedral elements $K$, where $h=\max\limits_{K\in T_{h}}h_{K}$ and $h_{K}$ is the diameter of the element $K$ \cite{BrezziFortin1991, BrennerScott1994, GiraultRaviart1986}. The scalar and vector finite element space are
    \begin{align*}
       &\bs{X}_{h}:=\{\bs{v}_{h}\in C^{0}(\Omega)^{d} : \bs{v}_{h}|_{K}\in P_{k}(K)^{d}, \forall K\in T_{h}\}\cap H_{0}^{1}(\Omega)^{d}, \\
       &\bs{V}_{h}:=\{\bs{v}_{h}\in C^{0}(\Omega)^{d} : \bs{v}_{h}|_{K}\in P_{k}(K)^{d}, \forall K\in T_{h}\}\cap \bs{V},\\
       &M_{h}:=\{q_{h}\in C^{0}(\Omega)\ :q_{h}|_{K}\in P_{l}(K),\ \forall K\in T_h\}\cap L_{0}^{2}(\Omega),\\
       &\bs{W}_{h}:=\{\bs{\Lambda}_{h}\in C^{0}(\Omega)^{d} : \bs{\Lambda}_{h}|_{K}\in P_{r}(K)^{d},\  \forall K\in T_{h}\}\cap H_{0}^{1}(\Omega)^{d},\\
       &\mathcal{T}_{h}:=\{\psi_{h}\in C^{0}(\Omega): \psi_{h}|_{K}\in P_{s}(K),\forall K\in T_{h}\}\cap H_{0}^{1}(\Omega).
    \end{align*}
   where $P_{k}(K), P_{l}(K),P_{r}(K)$ and $P_{s}(K)$ represent spaces of polynomials with degree bounded uniformly with respect to $K\in T_h$ and $d=2, 3$. Let $0=t^{0}<t^{1}<\ldots<t^{N+1}=T$ be a uniform partition of the time interval $[0,T]$ with $t^{n}=n\delta_t$.

The core idea of the second-order pressure projection method is to split the pressure term from the momentum equation.
Firstly, a prediction for the velocity field is computed. Then pressure and velocity corrections are performed to enforce the divergence-free constraint. Finally, the angular velocity and temperature fields are updated independently. Details are shown in the following.
      
Starting with the initial solution $(\bs{u}_{h}^{0},\bs{\omega}_{h}^{0}, \theta_{h}^{0}, p_{h}^{0})$, the discrete initial conditions are defined by
\begin{align}  \label{chuzhi}
   & (\tilde{\bs{u}}_{h}^{0},\bs{v}_{h}) = (\bs{u}_{h}^{0},\bs{v}_{h}) = (\bs{u}_{0},\bs{v}_{h}), \quad
   (\bs{\omega}_{h}^{0},\bs{\Lambda}_{h}) = (\bs{\omega}_{0},\bs{\Lambda}_{h}), \quad
   (\theta_{h}^{0},\psi_{h}) = (\theta_{0},\psi_{h}), \nn\\
   & \nabla p_{h}^{0} = (\chi+\mu)\Delta\bs{u}_{h}^{0} 
   - \bs{u}_{h}^{0}\cdot\nabla\bs{u}_{h}^{0}
   + 2\chi\nabla\times\bs{\omega}_{h}^{0}
   +\theta_{h}^{0}\bs{e}.
\end{align}
For $(\bs{u}_{h}^{1},\bs{\omega}_{h}^{1}, \theta_{h}^{1})$, a first-order pressure projection scheme \cite{ZhangHeYang2021} can be derived based on \eqref{chuzhi}.

\textit{\textbf{Step 1.}} Find $\tilde{\bs{u}}_{h}^{n+1} \in \bs{X}_{h}$ such that for all $\bs{v}_{h} \in \bs{X}_{h}$,
\begin{align}\label{step1}
\left(\frac{3\tilde{\bs{u}}_{h}^{n+1}-4\bs{u}_{h}^{n}+\bs{u}_{h}^{n-1}}{2\delta_t},\bs{v}_{h}\right)
&+(\chi+\mu)(\nabla\tilde{\bs{u}}_{h}^{n+1},\nabla\bs{v}_{h})
+((\bar{\bs{u}}_{h}^{n}\cdot\nabla)\tilde{\bs{u}}_{h}^{n+1},\bs{v}_{h}) \notag \\
&+(\nabla p_{h}^{n}, \bs{v}_{h})
= 2\chi(\nabla \times \bar{\bs{\omega}}_{h}^{n},\bs{v}_{h}) + (\bar{\theta}_{h}^{n}\bs{e}, \bs{v}_{h}). 
\end{align}
where $\bar{\bs{\omega}}_{h}^{n}=2\bs{\omega}_{h}^{n}-\bs{\omega}_{h}^{n-1}$ and $\bar{\theta}_{h}^{n}=2\theta_{h}^{n}-\theta_{h}^{n-1}$.

\textit{\textbf{Step 2.}} Find $p_{h}^{n+1} \in M_{h}$ such that for all $q_h \in M_{h}$,
\begin{align}\label{step2}
(\nabla p_{h}^{n+1},\nabla q_h)
= -\frac{3}{2\delta_t}(\nabla\cdot\tilde{\bs{u}}_{h}^{n+1}, q_h)
+ (\nabla p_{h}^{n},\nabla q_h). 
\end{align}

\textit{\textbf{Step 3.}} Update $\bs{u}_{h}^{n+1} \in \bs{X}_{h} \oplus \nabla M_{h}$ by
\begin{align}\label{step3}
\bs{u}_{h}^{n+1}
= \tilde{\bs{u}}_{h}^{n+1}
- \frac{2\delta_t}{3} \nabla p_{h}^{n+1}
+ \frac{2\delta_t}{3} \nabla p_{h}^{n}. 
\end{align}

\textit{\textbf{Step 4.}} Find $\bs{\omega}_{h}^{n+1} \in \bs{W}_{h}$ such that for all $\Lambda_{h} \in \bs{W}_{h}$,
\begin{align}\label{step4}
\zeta\left(\frac{3\bs{\omega}_{h}^{n+1}-4\bs{\omega}_{h}^{n}+\bs{\omega}_{h}^{n-1}}{2\delta_t},\Lambda_{h}\right)
&+ \upsilon(\nabla\bs{\omega}_{h}^{n+1}, \nabla \Lambda_{h})
+ \zeta((\bar{\bs{u}}_{h}^{n}\cdot\nabla)\bs{\omega}_{h}^{n+1}, \Lambda_{h}) \notag \\
&+ \eta(\nabla \cdot \bs{\omega}_{h}^{n+1}, \nabla \cdot \Lambda_{h})
+ 4\chi(\bs{\omega}_{h}^{n+1}, \Lambda_{h})
= 2\chi(\nabla\times \bar{\bs{u}}_{h}^{n}, \Lambda_{h}). 
\end{align}
where $\bar{\bs{u}}_{h}^{n}=2\bs{u}_{h}^{n}-\bs{u}_{h}^{n-1}$.

\textit{\textbf{Step 5.}} Find $\theta_{h}^{n+1} \in \mathcal{T}_{h}$ such that for all $\psi_{h} \in \mathcal{T}_{h}$,
\begin{align}\label{step5}
\left(\frac{3\theta_{h}^{n+1}-4\theta_{h}^{n}+\theta_{h}^{n-1}}{2\delta_t},\psi_{h}\right)
+ \kappa(\nabla\theta_{h}^{n+1}, \nabla \psi_{h})
+ ((\bar{\bs{u}}_{h}^{n}\cdot\nabla)\theta_{h}^{n+1}, \psi_{h})
= (\bar{\bs{u}}_{h}^{n}\cdot \bs{e}, \psi_{h}). 
\end{align}
The advantage of the numerical scheme \refe{chuzhi}-\refe{step5} lies in the linearity and fully decoupling, and $\bs{u}$, $p$, $\bs{\omega}$ and $\theta$ can be solved linearly and independently at each time step,  providing a feasible numerical approach to solve equations \refe{1.1} efficiently.

\subsection{Stability analysis}
In this subsection, the stability analysis of the second-order pressure projection scheme \refe{chuzhi}–\refe{step5} is presented.
    \begin{theorem}\label{3.1:}
      Suppose that Assumptions (A.1)-(A.2) are satisfied,
the scheme is unconditionally stable, for all $\delta_t>0$, there exists a prior bound such that
\begin{align*}
&\|\bs{u}_{h}^{N+1}\|_{0}^{2}
+ \|2\bs{u}_{h}^{N+1} - \bs{u}_{h}^{N}\|_{0}^{2}
+ \zeta\|\bs{\omega}_{h}^{N+1}\|_{0}^{2}
+ \zeta\|2\bs{\omega}_{h}^{N+1} - \bs{\omega}_{h}^{N}\|_{0}^{2}
+ \|\theta_{h}^{N+1}\|_{0}^{2}
+ \|2\theta_{h}^{N+1} - \theta_{h}^{N}\|_{0}^{2} \nonumber\\
& \quad + \delta_{t} \sum_{n=1}^{N} \Bigl(
\|\bs{u}_{h}^{n+1} - 2\bs{u}_{h}^{n} + \bs{u}_{h}^{n-1}\|_{0}^{2}
+ (\chi+\mu)\|\nabla\tilde{\bs{u}}_{h}^{n+1}\|_{0}^{2}
+ \zeta\|\bs{\omega}_{h}^{n+1} - 2\bs{\omega}_{h}^{n} + \bs{\omega}_{h}^{n-1}\|_{0}^{2}
+ \upsilon\|\nabla\bs{\omega}_{h}^{n+1}\|_{0}^{2} \nonumber\\
& \quad
+ 8\eta\|\nabla\cdot\bs{\omega}_{h}^{n+1}\|_{0}^{2}
+ 16\chi\|\bs{\omega}_{h}^{n+1}\|_{0}^{2}
+ \|\theta_{h}^{n+1} - 2\theta_{h}^{n} + \theta_{h}^{n-1}\|_{0}^{2}
+ \kappa\|\nabla\theta_{h}^{n+1}\|_{0}^{2}
\Bigr)
+ \frac{\delta_{t}^{2}}{3} \|\nabla p^{N+1}\|_{0}^{2} \nonumber\\
& \leqslant \Bigl(
\|\bs{u}_{h}^{0}\|_{0}^{2}
+ \|2\bs{u}_{h}^{1} - \bs{u}_{h}^{0}\|_{0}^{2}
+ \zeta\|\bs{\omega}_{h}^{0}\|_{0}^{2}
+ \zeta\|2\bs{\omega}_{h}^{1} - \bs{\omega}_{h}^{0}\|_{0}^{2}
+ \|\theta_{h}^{0}\|_{0}^{2}
+ \|2\theta_{h}^{1} - \theta_{h}^{0}\|_{0}^{2}
\Bigr)\nn\\
& \quad +C\delta_t(\|\bar{\bs{u}}_{h}^{n}\|_{0}^{2}+\|\bar{\bs{\omega}}_{h}^{n}\|_{0}^{2}+\|\bar{\theta}_{h}^{n}\|_{0}^{2})
+ \frac{4\delta_t^{2}}{3} \|\nabla p_{h}^{0}\|_{0}^{2},
\end{align*}
where $C$ is a positive constant independent of $h$ and $\delta_t$.
\end{theorem}\vspace{0.5em}

\begin{proof}
Setting $\bs{v}_{h}=4\delta_t\tilde{\bs{u}}_{h}^{n+1}$ in \refe{step1}, $\bs{\Lambda}_{h} = 4\delta_t \bs{\omega}_{h}^{n+1}$ in \refe{step4}, and $\psi_{h}=4\delta_t \theta_{h}^{n+1}$ in \refe{step5}, respectively, there are
\begin{align}
\label{3.2'}
&2(3\tilde{\bs{u}}_{h}^{n+1} - 4\bs{u}_{h}^{n} + \bs{u}_{h}^{n-1}, \tilde{\bs{u}}_{h}^{n+1})
+ 4(\chi+\mu)\delta_t \|\nabla \tilde{\bs{u}}_{h}^{n+1}\|_{0}^{2}
+ 4\delta_t (\nabla p_{h}^{n}, \tilde{\bs{u}}_{h}^{n+1}) \nonumber \\
& \hspace{16em} = 8\chi \delta_t (\nabla \times \bar{\bs{\omega}}_{h}^{n}, \tilde{\bs{u}}_{h}^{n+1})
+ 4\delta_t (\bar{\theta}_{h}^{n}\bs{e}, \tilde{\bs{u}}_{h}^{n+1}), \\
\label{3.8'}
&2\zeta(3\bs{\omega}_{h}^{n+1} - 4\bs{\omega}_{h}^{n} + \bs{\omega}_{h}^{n-1}, \bs{\omega}_{h}^{n+1})
+ 4\upsilon \delta_t \|\nabla \bs{\omega}_{h}^{n+1}\|_{0}^{2}
+ 4\eta \delta_t \|\nabla \cdot \bs{\omega}_{h}^{n+1}\|_{0}^{2} 
+ 16 \chi \delta_t \|\bs{\omega}_{h}^{n+1}\|_{0}^{2}\nonumber \\
& \hspace{16em} = 8 \chi \delta_t (\nabla \times \bar{\bs{u}}_{h}^{n}, \bs{\omega}_{h}^{n+1}), \\
\label{3.9'}
&2(3\theta_{h}^{n+1} - 4\theta_{h}^{n} + \theta_{h}^{n-1}, \theta_{h}^{n+1})
+ 4 \kappa \delta_t \|\nabla \theta_{h}^{n+1}\|_{0}^{2}
= 4 \delta_t (\bar{\bs{u}}_{h}^{n} \cdot \bs{e}, \theta_{h}^{n+1}).
\end{align}
From the weak divergence-free condition of $\bs{u}_{h}^{n+1}$, along with the identity $$2(3a-4b+c, a)=|a|^{2}-|b|^{2}+|2a-b|^{2}-|2b-c|^{2}+|a+2b-c|^{2}$$ and \refe{step3}, we obtain
\begin{align}
\label{3.4'}
2(3\tilde{\bs{u}}_{h}^{n+1}-4\bs{u}_{h}^{n}+\bs{u}_{h}^{n-1}, \tilde{\bs{u}}_{h}^{n+1})
&= \|\bs{u}_{h}^{n+1}\|_{0}^{2} - \|\bs{u}_{h}^{n}\|_{0}^{2}
+ \|2\bs{u}_{h}^{n+1} - \bs{u}_{h}^{n}\|_{0}^{2}
- \|2\bs{u}_{h}^{n} - \bs{u}_{h}^{n-1}\|_{0}^{2}\nn\\
& + \|\bs{u}_{h}^{n+1} - 2\bs{u}_{h}^{n} + \bs{u}_{h}^{n-1}\|_{0}^{2}
+6\|\tilde{\bs{u}}_{h}^{n+1} - \bs{u}_{h}^{n+1}\|_{0}^{2}.
\end{align}
We rewrite \refe{step3} as
\begin{align}
\label{3.3'}
&\bs{u}_{h}^{n+1}+\frac{2}{3}\delta_t\nabla p_{h}^{n+1}=\tilde{\bs{u}}_{h}^{n+1}+\frac{2}{3}\delta_t\nabla p_{h}^{n},
\end{align}
by taking the $L^2$ inner product of \refe{3.3'} with itself and applying relation \refe{step2}, we get
\begin{align}
\label{3.6'}
&(\nabla p_{h}^{n}, \tilde{\bs{u}}_{h}^{n+1})
= \frac{3}{4\delta_t}\|\bs{u}_{h}^{n+1}\|_{0}^{2}
- \frac{3}{4\delta_t}\|\tilde{\bs{u}}_{h}^{n+1}\|_{0}^{2}
+ \frac{\delta_t}{3}\|\nabla p_{h}^{n+1}\|_{0}^{2}
- \frac{\delta_t}{3}\|\nabla p_{h}^{n}\|_{0}^{2}, 
\end{align}
From \refe{3.2'}-\refe{3.6'}, one obtains
   \begin{align}\label{3.10'}
   &\|\bs{u}_{h}^{n+1}\|_{0}^{2}-\|\bs{u}_{h}^{n}\|_{0}^{2}+\|2\bs{u}_{h}^{n+1}-\bs{u}_{h}^{n}\|_{0}^{2}-\|2\bs{u}_{h}^{n}-\bs{u}_{h}^{n-1}\|_{0}^{2}
  +\|\bs{u}_{h}^{n+1}-2\bs{u}_{h}^{n}+\bs{u}_{h}^{n-1}\|_{0}^{2}+6\|\tilde{\bs{u}}_{h}^{n+1}-\bs{u}_{h}^{n+1}\|_{0}^{2}\nn\\
& \quad +4(\chi+\mu)\delta_t\|\nabla\tilde{\bs{u}}_{h}^{n+1}\|_{0}^{2}+\zeta\|\bs{\omega}_{h}^{n+1}\|_{0}^{2}-\zeta\|\bs{\omega}_{h}^{n}\|_{0}^{2}+\zeta\|2\bs{\omega}_{h}^{n+1}-\bs{\omega}_{h}^{n}\|_{0}^{2}-\zeta\|2\bs{\omega}_{h}^{n}-\bs{\omega}_{h}^{n-1}\|_{0}^{2}\nn\\
  & \quad +\zeta\|\bs{\omega}_{h}^{n+1}-2\bs{\omega}_{h}^{n}+\bs{\omega}_{h}^{n-1}\|_{0}^{2}
+4\upsilon\delta_t\|\nabla\bs{\omega}_{h}^{n+1}\|_{0}^{2}+4\eta\delta_t\|\nabla\cdot\bs{\omega}_{h}^{n+1}\|_{0}^{2}+16\chi\delta_t\|\bs{\omega}_{h}^{n+1}\|_{0}^{2}\nn\\
  & \quad +\|\theta_{h}^{n+1}\|_{0}^{2}-\|\theta_{h}^{n}\|_{0}^{2} +\|2\theta_{h}^{n+1}-\theta_{h}^{n}\|_{0}^{2}
  -\|2\theta_{h}^{n}-\theta_{h}^{n-1}\|_{0}^{2}+\|\theta_{h}^{n+1}-2\theta_{h}^{n}+\theta_{h}^{n-1}\|_{0}^{2}\nn\\
& \quad +4\kappa\delta_t\|\nabla\theta_{h}^{n+1}\|_{0}^{2}+\frac{4\delta_{t}^{2}}{3}(\|\nabla p_{h}^{n+1}\|_{0}^{2}-\|\nabla p_{h}^{n}\|_{0}^{2})+3\|\tilde{\bs{u}}_{h}^{n+1}-\bs{u}_{h}^{n+1}\|_{0}^{2}\nn\\
  & = 8\chi\delta_t (\nabla\times\bar{\bs{\omega}}_{h}^{n}, \tilde{\bs{u}}_{h}^{n+1})+4\delta_t (\bar{\theta}_{h}^{n}\bs{e},\tilde{\bs{u}}_{h}^{n+1})+8\chi\delta_t(\nabla\times\bar{\bs{u}}_{h}^{n}, \bs{\omega}_{h}^{n+1})+4\delta_t(\bar{\bs{u}}_{h}^{n}\cdot \bs{e}, \theta_{h}^{n+1}).
   \end{align} 
Invoking H\"{o}lder's inequality and Young's inequality,
\begin{align}
\label{3.11'}
&|4\chi\delta_t (\nabla\times\bar{\bs{\omega}}_{h}^{n}, \tilde{\bs{u}}_{h}^{n+1})|
= |4\chi\delta_t (\bar{\bs{\omega}}_{h}^{n}, \nabla\times\tilde{\bs{u}}_{h}^{n+1})| \nn \\
& \hspace{9em} \leqslant C\delta_t \|\bar{\bs{\omega}}_{h}^{n}\|_{0} \|\nabla \tilde{\bs{u}}_{h}^{n+1}\|_{0} 
\leqslant \frac{3(\chi+\mu)}{2} \delta_t \|\nabla \tilde{\bs{u}}_{h}^{n+1}\|_{0}^{2}
+ C\delta_t \|\bar{\bs{\omega}}_{h}^{n}\|_{0}^{2},
\\
\label{3.12'}
&|4\delta_t (e\bar{\theta}_{h}^{n}, \tilde{\bs{u}}_{h}^{n+1})|
\leqslant C\delta_t \|\bar{\theta}_{h}^{n}\|_{0} \|\nabla \tilde{\bs{u}}_{h}^{n+1}\|_{0}
\leqslant \frac{3(\chi+\mu)}{2} \delta_t \|\nabla \tilde{\bs{u}}_{h}^{n+1}\|_{0}^{2}
+ C\delta_t \|\bar{\theta}_{h}^{n}\|_{0}^{2}, 
\\
\label{3.13'}
&|8\chi\delta_t(\nabla\times\bar{\bs{u}}_{h}^{n}, \bs{\omega}_{h}^{n+1})|
\leqslant C\delta_t \|\bar{\bs{u}}_{h}^{n}\|_{0} \|\nabla \bs{\omega}_{h}^{n+1}\|_{0}
\leqslant 3\upsilon \delta_t \|\nabla \bs{\omega}_{h}^{n+1}\|_{0}^{2}
+ C\delta_t \|\bar{\bs{u}}_{h}^{n}\|_{0}^{2}, 
\\
\label{3.14'}
&|4\delta_t(\bar{\bs{u}}_{h}^{n} \cdot \bs{e}, \theta_{h}^{n+1})|
\leqslant C\delta_t \|\bar{\bs{u}}_{h}^{n}\|_{0} \|\nabla \theta_{h}^{n+1}\|_{0}
\leqslant 3\kappa \delta_t \|\nabla \theta_{h}^{n+1}\|_{0}^{2}
+ C\delta_t \|\bar{\bs{u}}_{h}^{n}\|_{0}^{2}.
\end{align}
Combining \refe{3.10'}-\refe{3.14'} and summing up from $n=0$ to $N$, we derive
   \begin{align*}
   &\|\bs{u}_{h}^{N+1}\|_{0}^{2}+\|2\bs{u}_{h}^{N+1}-\bs{u}_{h}^{N}\|_{0}^{2}+\zeta\|\bs{\omega}_{h}^{N+1}\|_{0}^{2}+\zeta\|2\bs{\omega}_{h}^{N+1}-\bs{\omega}_{h}^{N}\|_{0}^{2}+\|\theta_{h}^{N+1}\|_{0}^{2}
  +\|2\theta_{h}^{N+1}-\theta_{h}^{N}\|_{0}^{2}\nn\\
  & \quad +\sum_{n=0}^{n+1}(9\|\tilde{\bs{u}}_{h}^{n+1}-\bs{u}_{h}^{n+1}\|_{0}^{2}+ +\|\bs{u}_{h}^{n+1}-2\bs{u}_{h}^{n}+\bs{u}_{h}^{n-1}\|_{0}^{2}+\|\theta_{h}^{n+1}-2\theta_{h}^{n}+\theta_{h}^{n-1}\|_{0}^{2}\nn\\
  & \quad+\zeta\|\bs{\omega}_{h}^{n+1}-2\bs{\omega}_{h}^{n}+\bs{\omega}_{h}^{n-1}\|_{0}^{2}+(\chi+\mu)\delta_t\|\nabla\tilde{\bs{u}}_{h}^{n+1}\|_{0}^{2}+\upsilon\delta_t\|\nabla\bs{\omega}_{h}^{n+1}\|_{0}^{2}+4\eta\delta_t\|\nabla\cdot\bs{\omega}_{h}^{n+1}\|_{0}^{2}\nn\\
  & \quad+16\chi\delta_t\|\bs{\omega}_{h}^{n+1}\|_{0}^{2}+
  +\kappa\delta_t\|\nabla\theta_{h}^{n+1}\|_{0}^{2})+\frac{4\delta_{t}^{2}}{3}\|\nabla p_{h}^{N+1}\|_{0}^{2}\nn\\
  & \leqslant \|\bs{u}_{h}^{0}\|_{0}^{2}+\|2\bs{u}_{h}^{1}-\bs{u}_{h}^{0}\|_{0}^{2}+\zeta\|\bs{\omega}_{h}^{0}\|_{0}^{2}+\zeta\|2\bs{\omega}_{h}^{1}-\bs{\omega}_{h}^{0}\|_{0}^{2}+\|\theta_{h}^{0}\|_{0}^{2}
  +\|2\theta_{h}^{1}-\theta_{h}^{0}\|_{0}^{2}+\frac{4\delta_t^{2}}{3}\|\nabla p_{h}^{0}\|_{0}^{2}\nn\\
  & \quad+\sum_{n=0}^{n+1}(C\delta_t\|\bar{\bs{u}}_{h}^{n}\|_{0}^{2}+C\delta_t\|\bar{\bs{\omega}}_{h}^{n}\|_{0}^{2}+C\delta_t\|\bar{\theta}_{h}^{n}\|_{0}^{2}).
     \end{align*}
In light of Gr\"{o}nwall's lemma, the proof is completed.
\end{proof}

%% file: Error.tex
\section{Error estimates}
In this section, the convergence analysis of the fully discrete scheme \refe{chuzhi}-\refe{step5} is presented . We first derive suboptimal error estimates for the velocity, magnetic field, and angular velocity. Then, by introducing the inverse Stokes operator together with the dual-norm technique, we overcome the convergence-order reduction caused by the pressure term and subsequently obtain optimal error estimates. 

In the error analysis, the exact solutions at time $t^{n+1}$ are denoted as $\bs{u}(t^{n+1})$, $\bs{\omega}(t^{n+1})$, $p(t^{n+1})$, and $\theta(t^{n+1})$. From \refe{1.1}, we obtain
\begin{align}
\label{4.02'}
&\frac{3\bs{u}(t^{n+1}) - 4\bs{u}(t^{n}) + \bs{u}(t^{n-1})}{2\delta_t}
- (\chi+\mu)\Delta\bs{u}(t^{n+1})
+ (\bs{u}(t^{n+1})\cdot\nabla)\bs{u}(t^{n+1})
+ \nabla p(t^{n+1}) \notag\\
&\qquad\qquad\qquad\qquad\qquad\qquad\quad= 2\chi(\nabla\times\bs{\omega}(t^{n+1})) + \theta(t^{n+1})\bs{e} + R_{\bs{u}}^{n+1}, 
\\
\label{4.1}
&\nabla\cdot\bs{u}(t^{n+1}) = 0, 
\\
\label{4.2}
&\zeta\frac{3\bs{\omega}(t^{n+1}) - 4\bs{\omega}(t^{n}) + \bs{\omega}(t^{n-1})}{2\delta_t}
- \upsilon\Delta\bs{\omega}(t^{n+1})
+ \zeta(\bs{u}(t^{n+1})\cdot\nabla)\bs{\omega}(t^{n+1}) \notag\\
&\hspace{11em} - \eta\nabla(\nabla\cdot\bs{\omega}(t^{n+1}))
+ 4\chi\bs{\omega}(t^{n+1})
= 2\chi \nabla\times\bs{u}(t^{n+1}) + R_{\bs{\omega}}^{n+1}, 
\\
\label{4.3}
&\frac{3\theta(t^{n+1}) - 4\theta(t^{n}) + \theta(t^{n-1})}{2\delta_t}
- \kappa\Delta\theta(t^{n+1})
+ (\bs{u}(t^{n+1})\cdot\nabla)\theta(t^{n+1})
= \bs{u}(t^{n+1})\cdot \bs{e} + R_{\theta}^{n+1},
\end{align}
where 
\begin{align*}
R_{\bs{u}}^{n+1} &= \frac{3\bs{u}(t^{n+1}) - 4\bs{u}(t^{n}) + \bs{u}(t^{n-1})}{2\delta_t} - \bs{u}_{t}(t^{n+1}), \qquad
& R_{\bs{\omega}}^{n+1} &= \frac{3\bs{\omega}(t^{n+1}) - 4\bs{\omega}(t^{n}) + \bs{\omega}(t^{n-1})}{2\delta_t} - \bs{\omega}_{t}(t^{n+1}),\\
R_{\theta}^{n+1} &= \frac{3\theta(t^{n+1}) - 4\theta(t^{n}) + \theta(t^{n-1})}{2\delta_t} - \theta_{t}(t^{n+1}). 
\end{align*}
To meet the requirements of the fully discrete error analysis,
\label{Lem3.2}
the Stokes projection is defined by \cite{He2015}
\begin{align*}
(R_h(\bs{v},q),Q_h(\bs{v},q)):(\bs{X}, M)\longrightarrow(\bs{X}_h, M_h)
\end{align*}
satisfying
\begin{align*}
\begin{split}
&\upsilon(\nabla(R_h(\bs{v},q)-\bs{v}),\nabla \bs{v}_h)-(Q_h(\bs{v},q)-q, \nabla\cdot \bs{v}_h)=0, \ \forall \bs{v}_h\in \bs{X}_h,\\
&(\nabla\cdot (R_h(\bs{u},p)-\bs{u}), q_h)=0, \ \forall q_h\in M_h. \\
\end{split}
\end{align*}
Assume that $\bs{v} \in H_0^2(\Omega)^d\cap H^{r+1}(\Omega)^d$
 and $q\in L_0^2(\Omega)\cap H^r(\Omega) (d=2,3)$, the Stokes projection satisfies the estimate
\begin{align*}
&\|R_h(\bs{v},q) - \bs{v}\|_0 
\ + h \left( \|\nabla(R_h(\bs{v},q)-\bs{v})\|_0 + \|Q_h(\bs{v},q)-q\|_0 \right)
\leqslant C h^{r+1}(\|\bs{v}\|_{r+1} + \|q\|_r),\\
&\|R_h(\bs{v}, q)\|_{L^{\infty}} 
\ + \|\nabla R_h(\bs{v}, p)\|_{L^{3}}
\leqslant C(\|\bs{v}\|_2 + \|q\|_1),
\\
&\|\partial_{t} \nabla(R_h(\bs{v},q) - \bs{v})\|_0 
\ + h \|\partial_{t} \nabla(Q_h(\bs{v},q) - q)\|_0 
\leqslant C h^{r+1}(\|\partial_{t} \nabla \bs{v}\|_{r+1} + \|\partial_{t} \nabla q\|_r).
\end{align*}
Similarly, the $H^1$-orthogonal projection $R_{0h}:\bs{X}\rightarrow \bs{X}_h$ is defined by
\begin{align*}
(\nabla R_{0h}(\bs{v}) -\nabla \bs{v},\nabla \mathbf{w}_h)=0, \qquad  \forall \mathbf{w}_h\in \bs{X}_h,
\end{align*}
for any $\bs{v}\in \bs{X}\cap H^{r+1}(\Omega)^{d} (d=1,2,3)$, the projection satisfies the estimate
\begin{align*}
&\|R_{0h}(\bs{v}) -\bs{v}\|_0+h\|\nabla(R_{0h}(\bs{v})-\bs{v})\|_0\leqslant Ch^{r+1}\|\bs{v}\|_{r+1},\\
&\|R_{0h}(\bs{v})\|_{L^{\infty}} \leqslant C\|\bs{v}\|_2, \quad |\nabla R_{0h}(\bs{v})\|_{L^{\infty}} \leqslant C\|\bs{v}\|_{W^{1,\infty}},
\end{align*}
for simplicity, note the residuals as that
\begin{align*}
R_{\bs{\omega} h}^{n+1} &= R_{\bs{\omega} h}(\bs{\omega}(t^{n+1})), \quad 
& R_{\bs{u} h}^{n+1} &= R_{\bs{u} h}(\bs{u}(t^{n+1}),\ p(t^{n+1})), \\
Q_h^{n+1} &= Q_h(\bs{u}(t^{n+1}),\ p(t^{n+1})), \quad 
& R_{\theta h}^{n+1} &= R_{\theta h}(\theta(t^{n+1})), \quad n = 1, 2, \ldots, N+1.
\end{align*}
Accordingly, it follows that
\begin{align}
\label{4.4}
&\frac{3R_{\bs{u}h}^{n+1} - 4R_{\bs{u}h}^{n} + R_{\bs{u}h}^{n-1}}{2\delta_t} 
- (\chi + \mu)\Delta R_{\bs{u}h}^{n+1}
+ (R_{\bs{u}h}^{n+1} \cdot \nabla)R_{\bs{u}h}^{n+1}
+ \nabla Q_{h}^{n+1} \nonumber\\
&\hspace{18em} = 2\chi \nabla\times R_{\bs{\omega} h}^{n+1} + e_3 R_{\theta h}^{n+1} + R_{ul}^{n+1}, 
\\
\label{4.5}
&\nabla \cdot R_{\bs{u}h}^{n+1} = 0, 
\\
\label{4.6}
&\zeta\frac{3R_{\bs{\omega} h}^{n+1} - 4R_{\bs{\omega} h}^{n} + R_{\bs{\omega} h}^{n-1}}{2\delta_t}
- \upsilon \Delta R_{\bs{\omega} h}^{n+1}
+ \zeta(R_{\bs{u}h}^{n+1} \cdot \nabla)R_{\bs{\omega} h}^{n-1}
- \eta \nabla(\nabla \cdot R_{\bs{\omega} h}^{n-1}) + 4\chi R_{\bs{\omega} h}^{n+1} \nonumber \\
&\hspace{18em} = 2\chi \nabla\times R_{\bs{u} h}^{n+1} + R_{\bs{\omega} l}^{n+1}, \\
\label{4.7}
&\frac{3 R_{\theta h}^{n+1} - 4R_{\theta h}^{n} + R_{\theta h}^{n-1}}{2\delta_t}
- \kappa \Delta R_{\theta h}^{n+1}
+ (R_{\bs{u} h}^{n+1} \cdot \nabla)R_{\theta h}^{n+1}  = R_{\bs{u} h}^{n+1} \cdot \bs{e} + R_{\theta l}^{n+1},
\end{align}
where 
\begin{align*}
R_{\bs{u}l}^{n+1} &= \frac{3R_{\bs{u}h}^{n+1} - 4R_{\bs{u}h}^{n} + R_{\bs{u}h}^{n-1}}{2\delta_t} - \bs{u}_{t}(t^{n+1}), \qquad
& R_{\bs{\omega} l}^{n+1} &= \frac{3R_{\bs{\omega} h}^{n+1} - 4R_{\bs{\omega} h}^{n} + R_{\bs{\omega} h}^{n-1}}{2\delta_t} - \bs{\omega}_{t}(t^{n+1}), \\
R_{\theta l}^{n+1} &= \frac{3R_{\theta h}^{n+1} - 4R_{\theta h}^{n} + R_{\theta h}^{n-1}}{2\delta_t} - \theta_{t}(t^{n+1}).
\end{align*}
To simplify the analysis, the relevant errors and residuls are defined by
\begin{align*}
&\epsilon_{\bs{u}}^{n+1} = R_{\bs{u}h}^{n+1} - \bs{u}_{h}^{n+1}, &
& \tilde{\epsilon}_{\bs{u}}^{n+1} = R_{\bs{u}h}^{n+1} - \tilde{\bs{u}}_{h}^{n+1}, &
& \epsilon_{\bs{\omega}}^{n+1} = R_{\bs{\omega} h}^{n+1} - \bs{\omega}_{h}^{n+1}, \\
& \epsilon_{\theta}^{n+1} = R_{\theta h}^{n+1} - \theta_{h}^{n+1}, &
& \epsilon_{p}^{n+1} = Q_{h}^{n+1} - p_{h}^{n+1}, \\
& \bar{\epsilon}_{\bs{u}}^{n} = 2\epsilon_{\bs{u}}^{n} - \epsilon_{\bs{u}}^{n-1}, &
& \bar{\epsilon}_{\bs{\omega}}^{n} = 2\epsilon_{\bs{\omega}}^{n} - \epsilon_{\bs{\omega}}^{n-1}, &
& \bar{\epsilon}_{\theta}^{n} = 2\epsilon_{\theta}^{n} - \epsilon_{\theta}^{n-1},\\
& \bar{R}_{\bs{u} h}^{n} = 2R_{\bs{u} h}^{n} - R_{\bs{u} h}^{n-1}, &
& \bar{R}_{\bs{\omega} h}^{n} = 2R_{\bs{\omega} h}^{n} - R_{\bs{\omega} h}^{n-1}, &
& \bar{R}_{\theta h}^{n} = 2R_{\theta h}^{n} - R_{\theta h}^{n-1}.
\end{align*}

\begin{assumption}
[\cite{GuermondQuartapelle1998, GuermondShen2003}]
$\operatorname{max}(\|\epsilon_{\bs{u}}^{0}\|_{0}, \|\epsilon_{\bs{u}}^{1}\|_{0}, \|\epsilon_{\bs{\omega}}^{0}\|_{0}, \|\epsilon_{\bs{\omega}}^{1}\|_{0}, \|\epsilon_{\theta}^{0}\|_{0}, \|\epsilon_{\theta}^{1}\|_{0})\leqslant C h^{l+1}$, and $\|\epsilon_{p}^{1}\|_{1}\leqslant C$.
\end{assumption}
Based on Assumptions (A.1)-(A.3), error estimates for the velocity field and the angular velocity field can be obtained.

\subsection{The suboptimal error estimate}
To establish optimal error estimates for the second-order pressure projection scheme, a suboptimal error estimate is first derived. This preliminary estimate serves as a key step towards obtaining optimal convergence results.
\begin{theorem}\label{4.1a}   Suppose that Assumptions (A.1)-(A.3) hold. Then, the errors satisfy
\begin{align}\label{4.8}
&\|\epsilon_{\bs{u}}^{N+1}\|_{0}^{2}+\|2\epsilon_{\bs{u}}^{N+1}-\epsilon_{\bs{u}}^{N-1}\|_{0}^{2}+\zeta\|\epsilon_{\bs{\omega}}^{N+1}\|_{0}^{2}+\zeta\|2\epsilon_{\bs{\omega}}^{N+1}-\epsilon_{\bs{\omega}}^{N-1}\|_{0}^{2}
+\|\epsilon_{\theta}^{N+1}\|_{0}^{2}+\|2\epsilon_{\theta}^{N+1}-\epsilon_{\theta}^{N-1}\|_{0}^{2}\nn\\
&\quad+\frac{8\delta_{t}^{2}}{3}\|\nabla\epsilon_{p}^{N+1}\|_{0}^{2}+\sum_{n=0}^{N}\Bigl((\chi+\mu)\delta_{t}\|\nabla\tilde{\epsilon}_{\bs{u}}^{n+1}\|_{0}^{2}+\upsilon\delta_{t}\|\nabla\epsilon_{\bs{\omega}}^{n+1}\|_{0}^{2}+\kappa\delta_{t}\|\nabla\epsilon_{\theta}^{n+1}\|_{0}^{2}
\nn\\
&\quad+\|\epsilon_{\bs{u}}^{n+1}-2\epsilon_{\bs{u}}^{n}+\epsilon_{\bs{u}}^{n-1}\|_{0}^{2}+\zeta\|\epsilon_{\bs{\omega}}^{n+1}-2\epsilon_{\bs{\omega}}^{n}+\epsilon_{\bs{\omega}}^{n-1}\|_{0}^{2}
+\|\epsilon_{\theta}^{n+1}-2\epsilon_{\theta}^{n}+\epsilon_{\theta}^{n-1}\|_{0}^{2}\nn\\
&\quad+6\|\tilde{\epsilon}_{\bs{u}}^{n+1}-\epsilon_{\bs{u}}^{n+1}\|_{0}^{2}+4\eta\delta_{t}\|\nabla\cdot \epsilon_{\bs{\omega}}^{n+1}\|_{0}^{2}+16\chi\delta_{t}\|\epsilon_{\bs{\omega}}^{n+1}\|_{0}^{2}\Bigr)
\leqslant C (\delta_{t}^{2}+h^{2l+2}),
\end{align}
where $C$ is a positive number independent of $\delta_t$ and $h$.
\end{theorem}
\begin{proof}
Subtracting \refe{4.4}, \refe{4.6}, and \refe{4.7} from \refe{step1}, \refe{step4}, and \refe{step5}, respectively, we derive
\begin{align}
\label{4.9}
&\left(\frac{3\tilde{\epsilon}_{\bs{u}}^{n+1} - 4\epsilon_{\bs{u}}^{n} + \epsilon_{\bs{u}}^{n-1}}{2\delta_t}, \bs{v}_{h}\right)
    + (\chi + \mu)(\nabla \tilde{\epsilon}_{\bs{u}}^{n+1}, \nabla \bs{v}_{h}) 
 + \big((R_{\bs{u}h}^{n+1} \cdot \nabla) R_{\bs{u}h}^{n+1} - (\bar{\bs{u}}_{h}^{n} \cdot \nabla) \tilde{\bs{u}}_{h}^{n+1}, \bs{v}_{h}\big) \nonumber\\
    &\quad+ (\nabla Q_{h}^{n+1} - \nabla p_{h}^{n}, \bs{v}_{h}) 
= 2\chi(\nabla \times R_{\bs{\omega} h}^{n+1} - \nabla \times \bar{\bs{\omega}}_{h}^{n}, \bs{v}_{h})
    + ( R_{\theta h}^{n+1} \bs{e}- \bar{\theta}_{h}^{n}\bs{e}, \bs{v}_{h}) + (R_{ul}^{n+1}, \bs{v}_{h}), 
    \\
\label{4.10}
&\zeta\left(\frac{3\epsilon_{\bs{\omega}}^{n+1} - 4\epsilon_{\bs{\omega}}^{n} + \epsilon_{\bs{\omega}}^{n-1}}{2\delta_t}, \bs{\Lambda}_{h}\right)
    + \upsilon(\nabla \epsilon_{\bs{\omega}}^{n+1}, \nabla \bs{\Lambda}_{h})
+ \zeta\big((R_{\bs{u}h}^{n+1} \cdot \nabla) R_{\bs{\omega} h}^{n+1} - (\bar{\bs{u}}_{h}^{n} \cdot \nabla) \bs{\omega}_{h}^{n+1}, \bs{\Lambda}_{h}\big)\nonumber\\
    &\quad+ \eta(\nabla \cdot \epsilon_{\bs{\omega}}^{n+1}, \nabla \cdot \bs{\Lambda}_{h}) 
+ 4\chi(\epsilon_{\bs{\omega}}^{n+1}, \bs{\Lambda}_{h})
= 2\chi(\nabla \times R_{\bs{u} h}^{n+1} - \nabla \times \bar{\bs{u}}_{h}^{n}, \bs{\Lambda}_{h}) + (R_{\bs{\omega} l}^{n+1}, \bs{\Lambda}_{h}),
\\
\label{4.11}
&\left(\frac{3\epsilon_{\theta}^{n+1} - 4\epsilon_{\theta}^{n} + \epsilon_{\theta}^{n-1}}{2\delta_t}, \psi_{h}\right)
    + \kappa(\nabla \epsilon_{\theta}^{n+1}, \nabla \psi_{h})
+ \big((R_{\bs{u} h}^{n+1} \cdot \nabla) R_{\theta h}^{n+1} - (\bar{\bs{u}}_{h}^{n} \cdot \nabla) \theta_{h}^{n+1}, \psi_{h}\big) \nonumber \\
& \hspace{19em} = (R_{\bs{u} h}^{n+1} \cdot\bs{e} - \bar{\bs{u}}_{h}^{n} \cdot \bs{e}, \psi_{h}) + (R_{\theta l}^{n+1}, \psi_{h}). 
\end{align}
Testing \ \refe{4.9}, \refe{4.10} and \refe{4.11}\ by \ $4\delta_t \tilde{\epsilon}_{\bs{u}}^{n+1}$, $4\delta_t \epsilon_{\bs{\omega}}^{n+1}$ and $4\delta_t \epsilon_{\theta}^{n+1}$ respectively, we have
\begin{align}
\label{4.12}
&2(3\tilde{\epsilon}_{\bs{u}}^{n+1} - 4\epsilon_{\bs{u}}^{n} + \epsilon_{\bs{u}}^{n-1}, \tilde{\epsilon}_{\bs{u}}^{n+1})
  + 4(\chi+\mu)\delta_t \|\nabla \tilde{\epsilon}_{\bs{u}}^{n+1}\|_0^2
\nonumber\\
 &\quad  + 4\delta_t \big((R_{\bs{u}h}^{n+1} \cdot \nabla) R_{\bs{u}h}^{n+1} - (\bar{\bs{u}}_h^n \cdot \nabla) \tilde{\bs{u}}_h^{n+1}, \tilde{\epsilon}_{\bs{u}}^{n+1} \big) + 4\delta_t (\nabla Q_h^{n+1} - \nabla p_h^n, \tilde{\epsilon}_u^{n+1}) \nn \\
& \quad = 8\chi \delta_t (\nabla \times R_{\bs{\omega} h}^{n+1} - \nabla \times \bar{\bs{\omega}}_h^n, \tilde{\epsilon}_u^{n+1})
  + 4\delta_t ( R_{\theta h}^{n+1}\bs{e} -  \bar{\theta}_{h}^{n} \bs{e}, \tilde{\epsilon}_u^{n+1}) + 4\delta_t (R_{ul}^{n+1}, \tilde{\epsilon}_u^{n+1}),
  \\
\label{4.21}
&2\zeta(3 \epsilon_{\bs{\omega}}^{n+1} - 4 \epsilon_{\bs{\omega}}^n + \epsilon_{\bs{\omega}}^{n-1}, \epsilon_{\bs{\omega}}^{n+1})
  + 4 \upsilon \delta_t \|\nabla \epsilon_{\bs{\omega}}^{n+1}\|_0^2
 \nonumber \\
  &\quad+ 4 \delta_t \zeta \big((R_{\bs{u}h}^{n+1} \cdot \nabla) R_{\bs{\omega} h}^{n+1} - (\bar{\bs{u}}_h^n \cdot \nabla) \bs{\omega}_h^{n+1}, \epsilon_{\bs{\omega}}^{n+1}\big) + 4 \eta \delta_t \|\nabla \cdot \epsilon_{\bs{\omega}}^{n+1}\|_0^2 + 16 \chi \delta_t \|\epsilon_{\bs{\omega}}^{n+1}\|_0^2 \nn \\
& \quad = 8 \chi \delta_t (\nabla \times R_{\bs{u} h}^{n+1} - \nabla \times \bar{\bs{u}}_h^n, \epsilon_{\bs{\omega}}^{n+1}) + 4 \delta_t (R_{\bs{\omega} l}^{n+1}, \epsilon_{\bs{\omega}}^{n+1}), 
\\
\label{4.25}
&2(3 \epsilon_{\theta}^{n+1} - 4 \epsilon_{\theta}^n + \epsilon_{\theta}^{n-1}, \epsilon_{\theta}^{n+1})
  + 4 \kappa \delta_t \|\nabla \epsilon_{\theta}^{n+1}\|_0^2
+ 4 \delta_t \big((R_{\bs{u} h}^{n+1} \cdot \nabla) R_{\theta h}^{n+1} - (\bar{\bs{u}}_h^n \cdot \nabla) \theta_h^{n+1}, \epsilon_{\theta}^{n+1}\big) \nonumber \\
& \quad =\ 4 \delta_t (R_{\theta h}^{n+1} - \bar{\theta}_{h}^{n}\bs{e}, \epsilon_{\theta}^{n+1}) + 4 \delta_t (R_{\theta l}^{n+1}, \epsilon_{\theta}^{n+1}).
\end{align}
By adding and subtracting terms $\bs{u}(t^{n+1}), \nabla Q_{h}^{n+1}$ and $\nabla Q_{h}^{n}$ in \refe{step3} to obtain
\begin{align}\label{4.18}
\frac{3(\tilde{\epsilon}_{\bs{u}}^{n+1}-\epsilon_{\bs{u}}^{n+1})}{2\delta_t}=\nabla\epsilon_{p}^{n+1}-\nabla\epsilon_{p}^{n}-\nabla Q_{h}^{n+1}+\nabla Q_{h}^{n}.
\end{align}
Testing \refe{4.18} by $4\delta_t (\nabla\epsilon_{p}^{n}+\nabla Q_{h}^{n+1}-\nabla Q_{h}^{n})$ to get
\begin{align}\label{4.19'}
&6(\nabla\epsilon_{p}^{n}+\nabla Q_{h}^{n+1}-\nabla Q_{h}^{n}, \tilde{\epsilon}_{\bs{u}}^{n+1})
= 4\delta_{t}(\nabla\epsilon_{p}^{n+1}-\nabla\epsilon_{p}^{n}-\nabla Q_{h}^{n+1}+\nabla Q_{h}^{n}, \nabla\epsilon_{p}^{n+1})
-\frac{3}{2\delta_{t}}\|\tilde{\epsilon}_{\bs{u}}^{n+1}-\epsilon_{\bs{u}}^{n+1}\|_{0}^{2}.
\end{align}
Based on \refe{4.19'}, it follows that
\begin{align}\label{4.19}
&4\delta_t (\nabla\epsilon_{p}^{n}+\nabla Q_{h}^{n+1}-\nabla Q_{h}^{n}, \tilde{\epsilon}_{\bs{u}}^{n+1})= -6\|\tilde{\epsilon}_{\bs{u}}^{n+1}-\epsilon_{\bs{u}}^{n+1}\|_{0}^{2}-\frac{8\delta_{t}^{2}}{3}(\nabla Q_{h}^{n+1}-\nabla Q_{h}^{n}, \nabla\epsilon_{p}^{n+1})\nn\\ 
&\quad+\frac{4\delta_{t}^{2}}{3}(\|\nabla\epsilon_{p}^{n+1}\|_{0}^{2}-\|\nabla\epsilon_{p}^{n+1}\|_{0}^{2}+\|\nabla\epsilon_{p}^{n+1}-\nabla\epsilon_{p}^{n}\|_{0}^{2}).
\end{align}
The first term of \refe{4.12} can be rewritten as
\begin{align}\label{4.13}
&2(3\tilde{\epsilon}_{\bs{u}}^{n+1}-4\epsilon_{\bs{u}}^{n+1}+\epsilon_{\bs{u}}^{n-1}, \tilde{\epsilon}_{\bs{u}}^{n+1})
=\|\epsilon_{\bs{u}}^{n+1}\|_{0}^{2}-\|\epsilon_{\bs{u}}^{n}\|_{0}^{2}
+\|2\epsilon_{\bs{u}}^{n+1}-\epsilon_{\bs{u}}^{n}\|_{0}^{2}-\|2\epsilon_{\bs{u}}^{n}-\epsilon_{\bs{u}}^{n-1}\|_{0}^{2}\nn\\
&\quad+\|\epsilon_{\bs{u}}^{n+1}-2\epsilon_{\bs{u}}^{n}+\epsilon_{\bs{u}}^{n-1}\|_{0}^{2}+6\|\tilde{\epsilon}_{\bs{u}}^{n+1}-\epsilon_{\bs{u}}^{n+1}\|_{0}^{2}.
\end{align}
Using H\"{o}lder's inequality and Young's inequality, we get
\begin{spreadlines}{5pt}
\begin{align}
\label{4.14}
&|-4\delta_t((R_{\bs{u}h}^{n+1}\cdot\nabla)R_{\bs{u}h}^{n+1} - (\bar{\bs{u}}_{h}^{n}\cdot\nabla)\tilde{\bs{u}}_{h}^{n+1}, \tilde{\epsilon}_{\bs{u}}^{n+1})| \nn \\
& \hspace{14em} \leqslant C\delta_t\|2\epsilon_{\bs{u}}^{n}-\epsilon_{\bs{u}}^{n-1}\|_{0}^{2} + C\delta_t^{5} + (\chi+\mu)\|\nabla\tilde{\epsilon}_{\bs{u}}^{n+1}\|_{0}^{2}.
\\
\label{4.15}
&|8\chi\delta_t(\nabla\times R_{\bs{\omega} h}^{n+1}-\nabla \times \bar{\bs{\omega}}_{h}^{n}, \tilde{\epsilon}_{\bs{u}}^{n+1})| 
\leqslant C\delta_t\|2\epsilon_{\bs{\omega}}^{n}-\epsilon_{\bs{\omega}}^{n-1}\|_{0}^{2} + C\delta_t^5 + (\chi+\mu)\delta_t\|\nabla\tilde{\epsilon}_{\bs{u}}^{n+1}\|_{0}^{2},
\\
\label{4.16}
&|4\delta_t(R_{\theta h}^{n+1}\bs{e} - \bar{\theta}_{h}^{n}\bs{e}, \tilde{\epsilon}_{\bs{u}}^{n+1})|
\leqslant C\delta_t\|2\epsilon_{\theta}^{n}-\epsilon_{\theta}^{n-1}\|_{0}^{2} + C\delta_t^5 + (\chi+\mu)\delta_t\|\nabla\tilde{\epsilon}_{\bs{u}}^{n+1}\|_{0}^{2},
\\
\label{4.17}
&|4\delta_t(R_{u l}^{n+1}, \tilde{\epsilon}_{\bs{u}}^{n+1})| \leqslant C\delta_t(\delta_t^2 + h^{l+1})^2 + (\chi+\mu)\delta_t\|\nabla\tilde{\epsilon}_{\bs{u}}^{n+1}\|_{0}^{2},
\\
\label{4.20}
&\left| \frac{8\delta_t^2}{3}(\nabla Q_{h}^{n+1} - \nabla Q_{h}^{n}, \nabla\epsilon_{p}^{n+1}) \right| 
\leqslant \frac{4\delta_t}{3} \|\nabla Q_{h}^{n+1} - \nabla Q_{h}^{n}\|_0^2 + C\delta_t^3\|\nabla\epsilon_{p}^{n+1}\|_0^2,
\\
\label{4.24}
&|4\delta_t(R_{\bs{\omega} l}^{n+1}, \epsilon_{\bs{\omega}}^{n+1})| \leqslant C\delta_t\|R_{\bs{\omega} l}^{n+1}\|_0^2 + \upsilon\delta_t\|\nabla\epsilon_{\bs{\omega}}^{n+1}\|_0^2,
\\
 \label{4.22}
&|-4\delta_t \zeta((R_{\bs{u}h}^{n+1}\cdot\nabla)R_{\bs{\omega} h}^{n+1} - (\bar{\bs{u}}_{h}^{n}\cdot\nabla)\bs{\omega}_{h}^{n+1}, \epsilon_{\bs{\omega}}^{n+1})| \nn\\
&\hspace{8em} \leqslant \upsilon\delta_t\|\nabla\epsilon_{\bs{\omega}}^{n+1}\|_0^2+C\delta_t\|2\epsilon_{\bs{u}}^{n}-\epsilon_{\bs{u}}^{n-1}\|_0^2 
+ C\delta_t\|R_{\bs{u}h}^{n+1} - 2R_{\bs{u}h}^{n} + R_{\bs{u}h}^{n-1}\|_0^2, 
\\
\label{4.27}
&|4\delta_t(R_{\theta h}^{n+1} - \bar{\theta}_{h}^{n}\bs{e}, \epsilon_{\theta}^{n+1})| \nn\\
&\hspace{8em} \leqslant C\delta_t(\|2\epsilon_{\theta}^{n} - \epsilon_{\theta}^{n-1}\|_0^2 + \|R_{\theta h}^{n+1} - 2R_{\theta h}^{n} + R_{\theta h}^{n-1}\|_0^2) + \kappa\delta_t\|\nabla\epsilon_{\theta}^{n+1}\|_0^2,
\\
\label{4.23}
&|8\chi\delta_t(\nabla\times R_{\bs{u} h}^{n+1} - \nabla \times \bar{\bs{u}}_{h}^{n}, \epsilon_{\bs{\omega}}^{n+1})| \nn\\
&\hspace{8em} \leqslant C\delta_t(\|2\epsilon_{\bs{u}}^{n}-\epsilon_{\bs{u}}^{n-1}\|_0^2 + \|R_{\bs{u}h}^{n+1} - 2R_{\bs{u}h}^{n} + R_{\bs{u}h}^{n-1}\|_0^2)
+ \upsilon\delta_t\|\nabla\epsilon_{\bs{\omega}}^{n+1}\|_0^2,
\\
\label{4.26}
&|4\delta_t((R_{\bs{u}h}^{n+1}\cdot\nabla)R_{\theta h}^{n+1} - (\bar{\bs{u}}_{h}^{n}\cdot\nabla)\theta_{h}^{n+1}, \epsilon_{\theta}^{n+1})| \nn\\
&\hspace{8em}  \leqslant C\delta_t(\|2\epsilon_{\bs{u}}^{n}-\epsilon_{\bs{u}}^{n-1}\|_0^2 + \|R_{\bs{u}h}^{n+1} - 2R_{\bs{u}h}^{n} + R_{\bs{u}h}^{n-1}\|_0^2) 
+ \kappa\delta_t\|\nabla\epsilon_{\theta}^{n+1}\|_0^2,
\\
\label{4.28}
&|4\delta_t(R_{\theta l}^{n+1}, \epsilon_{\theta}^{n+1})| \leqslant C\delta_t\|R_{\theta l}^{n+1}\|_0^2 + \kappa\delta_t\|\nabla\epsilon_{\theta}^{n+1}\|_0^2.
\end{align}
\end{spreadlines}
Combining \refe{4.12}-\refe{4.28} and summing up $n=1$ to $N$, we derive
   \begin{align}\label{4.29}
   &\|\epsilon_{\bs{u}}^{N+1}\|_{0}^{2}+\|2\epsilon_{\bs{u}}^{N+1}-\epsilon_{\bs{u}}^{N}\|_{0}^{2}+\zeta\|\epsilon_{\bs{\omega}}^{N+1}\|_{0}^{2}+\zeta\|2\epsilon_{\bs{\omega}}^{N+1}-\epsilon_{\bs{\omega}}^{N-1}\|_{0}^{2}
+\zeta\|\epsilon_{\theta}^{N+1}\|_{0}^{2}+\|2\epsilon_{\theta}^{N+1}-\epsilon_{\theta}^{N-1}\|_{0}^{2}\nn\\
&\quad+\delta_{t}\sum_{n=0}^{N}((\chi+\mu)\|\nabla\tilde{\epsilon}_{\bs{u}}^{n+1}\|_{0}^{2}+\upsilon\|\nabla\epsilon_{\bs{\omega}}^{n+1}\|_{0}^{2}+\kappa\|\nabla\epsilon_{\theta}^{n+1}\|_{0}^{2})+\sum_{n=0}^{N}(\|\epsilon_{\bs{u}}^{n+1}-2\epsilon_{\bs{u}}^{n}+\epsilon_{\bs{u}}^{n-1}\|_{0}^{2}
\nn\\
&\quad+\|\epsilon_{\bs{\omega}}^{n+1}-2\epsilon_{\bs{\omega}}^{n}+\epsilon_{\bs{\omega}}^{n-1}\|_{0}^{2}
+\|\epsilon_{\theta}^{n+1}-2\epsilon_{\theta}^{n}+\epsilon_{\theta}^{n-1}\|_{0}^{2}+6\|\tilde{\epsilon}_{\bs{u}}^{n+1}-\epsilon_{\bs{u}}^{n+1}\|_{0}^{2})+\frac{8\delta_{t}^{2}}{3}\|\nabla\epsilon_{p}^{N+1}\|_{0}^{2}\nn\\
& \leqslant \|\epsilon_{\bs{u}}^{0}\|_{0}^{2}+\|2\epsilon_{\bs{u}}^{1}-\epsilon_{\bs{u}}^{0}\|_{0}^{2}+\zeta\|\epsilon_{\bs{\omega}}^{0}\|_{0}^{2}+\zeta\|2\epsilon_{\bs{\omega}}^{1}-\epsilon_{\bs{\omega}}^{0}\|_{0}^{2}
+\|\epsilon_{\theta}^{1}\|_{0}^{2}+\|2\epsilon_{\theta}^{1}-\epsilon_{\theta}^{0}\|_{0}^{2}\nn\\
&\quad+C(\sum_{n=1}^{N+1}\|2\epsilon_{\bs{u}}^{n}-\epsilon_{\bs{u}}^{n-1}\|_{0}^{2}+\sum_{n=1}^{N+1}\zeta\|2\epsilon_{\bs{\omega}}^{n}-\epsilon_{\bs{\omega}}^{n-1}\|_{0}^{2}+\sum_{n=1}^{N+1}\|2\epsilon_{\theta}^{n}-\epsilon_{\theta}^{n-1}\|_{0}^{2}+\sum_{n=1}^{N+1}\|R_{\bs{u} l}^{n+1}\|_{0}^{2}\nn\\
&\quad+\sum_{n=1}^{N+1}\|R_{\bs{\omega} l}^{n+1}\|_{0}^{2}+\sum_{n=1}^{N+1}\|R_{\theta l}^{n+1}\|_{0}^{2}+\sum_{n=1}^{N+1}\delta_t\|R_{\theta h}^{n+1}-2R_{\theta h}^{n}+R_{\theta h}^{n-1}\|_{0}^{2}+\delta_t\|R_{u h}^{n+1}-2R_{u h}^{n}+R_{u h}^{n-1}\|_{0}^{2}\nn\\
&\quad+\delta_t\|R_{\bs{\omega} h}^{n+1}-2R_{\bs{\omega} h}^{n}+R_{\bs{\omega} h}^{n-1}\|_{0}^{2}+\frac{16}{3}\delta_t\sum_{n=1}^{N+1}\|\nabla Q_{h}^{n+1}-\nabla Q_{h}^{n}\|_{0}^{2}+\delta_{t}^{2}\sum_{n=1}^{N+1}\|\nabla\epsilon_{p}^{n+1}\|_{0}^{2}),
   \end{align}
where we have used the following estimations
\cite{GuermondQuartapelle1998, LinCheZey2025, Shen1994}
\begin{alignat*}{2}
&\|\nabla (Q_{h}^{n+1} - Q_{h}^{n})\|_{0}\leqslant C\delta_{t}, \quad
\|R_{\bs{u}h}^{n+1} - 2R_{\bs{u}h}^{n} + R_{\bs{u}h}^{n-1}\|_{0} \leqslant C\delta_{t}^{2}, \quad
\|R_{\bs{\omega} h}^{n+1} - 2R_{\bs{\omega} h}^{n} + R_{\bs{\omega} h}^{n-1}\|_{0} \leqslant C\delta_{t}^{2}, \\
&\|R_{\theta h}^{n+1} - 2R_{\theta h}^{n} + R_{\theta h}^{n-1}\|_{0} \leqslant C\delta_{t}^{2}, \quad
\|R_{u l}^{n+1}\|_{0} + \|R_{\bs{\omega} l}^{n+1}\|_{0} + \|R_{\theta l}^{n+1}\|_{0} \leqslant C(\delta_{t}^{2} + h^{l+1}).
\end{alignat*}
By Assumption(A.3) and Gr\"{o}nwall's lemma, the proof is completed.
\end{proof}

Suboptimal error estimates indicate that the pressure term causes the loss of optimal convergence, thereby reducing the overall accuracy. To address this issue, the dual-norm technique is introduced. The appropriate operators and associated dual norms are then defined to recover the optimal error estimate and eliminate pressure-induced order reduction.

\subsection{The optimal error estimate} \label{sec:opt_err}
We define the inverse Stokes operator $\mathcal{A}^{-1}$~\cite{SiWangWang2022, GuermondShen2003, Shen1994} as follows.
\begin{align*}\label{2.16'}
\mathcal{A}^{-1}: V_{1}^{-1}(\Omega)^{d} \rightarrow V, \quad 
\forall \bs{v} \in V_{1}^{-1}(\Omega)^{d}, \ \mathcal{A}^{-1}(v) \in V,
\end{align*}
which satisfies the following problem:
\begin{equation*}\label{2.17}
\left\{
\begin{array}{lll}
(\nabla \mathcal{A}^{-1}(\bs{v}), \nabla \bs{w}) - (r, \nabla \!\cdot\! \bs{w}) = (\bs{v}, \bs{w}), & \forall \bs{w} \in H_{0}^{1}(\Omega)^{d},\\[3pt]
(r, \nabla\cdot\mathcal{A}^{-1}(\bs{v})) = 0, & \forall q \in L_{0}^{2}(\Omega),
\end{array}
\right.
\end{equation*}
where $(\cdot, \cdot)$ denotes the duality pairing between $H^{-1}(\Omega)^{d}$ and $H^{1}(\Omega)^{d}$. 
It follows that
\begin{align}\label{2.17'}
\|\mathcal{A}^{-1}(\bs{v})\|_1 + \|\nabla r\|_{0} \le C \|\bs{v}\|_{-1}, 
\quad \forall v \in H^{-1}(\Omega)^{d}.
\end{align}
If the boundary $\partial\Omega$ is sufficiently smooth, the following regularity estimate holds~\cite{Temam1984}:
\begin{align}\label{2.18'}
\|\mathcal{A}^{-1}(\bs{v})\|_2 + \|\nabla r\|_{0} \le C \|\bs{v}\|_{0}, 
\quad \forall \bs{v} \in L^{2}(\Omega)^{d}.
\end{align}
The bilinear form~\cite{GuermondShen2003}
\[
H^{-1}(\Omega)^{d} \times H^{-1}(\Omega)^{d} 
\ni (\bs{v}, \bs{w}) \mapsto (\mathcal{A}^{-1}(\bs{v}), \bs{w}) 
:= (\nabla \mathcal{A}^{-1}(\bs{v}), \nabla \mathcal{A}^{-1}(\bs{w})) \in \mathbb{R}
\]
induces a semi-norm in $H^{-1}(\Omega)^{d}$, denoted by $|\cdot|_{*}$, defined as
\begin{align}\label{2.19'}
|v|_{*} = \|\nabla \mathcal{A}^{-1}(v)\|_{0} 
\le C \|v\|_{-1}, 
\quad \forall v \in H^{-1}(\Omega)^{d}.
\end{align}
Similarly, the inverse elliptic operator $\tilde{\mathcal{A}}^{-1}$~\cite{HeywoodRannacher1982} 
is defined as $\tilde{\mathcal{A}}^{-1}: H^{-1}(\Omega)^{d} \rightarrow \bs{X}$, 
where for any $\bs{w} \in H^{-1}(\Omega)^{d}$, the function 
$\tilde{\mathcal{A}}^{-1}(\bs{w}) \in \bs{X}$ satisfies
\begin{align}\label{2.20'}
(\nabla \tilde{\mathcal{A}}^{-1}(\bs{w}), \nabla \bs{v}) = (\bs{w}, \bs{v}), 
\quad \forall \bs{v} \in H_{0}^{1}(\Omega)^{d}.
\end{align}
Then
\begin{align}\label{2.21'}
\|\tilde{\mathcal{A}}^{-1}(\bs{w})\|_1 \le C \|\bs{w}\|_{-1}, 
\quad \forall \bs{w} \in H^{-1}(\Omega)^{d}.
\end{align}
If the domain $\Omega$ is sufficiently smooth, the regularity property will be activated
\begin{align}\label{2.22'}
\|\tilde{\mathcal{A}}^{-1}(\bs{w})\|_2 \le C \|\bs{w}\|_{0}, 
\quad \forall \bs{w} \in L^{2}(\Omega)^{d}
\end{align}
holds. 
Following~\cite{GuermondShen2003}, the bilinear form
\[
H^{-1}(\Omega)^{d} \times H^{-1}(\Omega)^{d} 
\ni (\bs{v}, \bs{w}) \mapsto (\tilde{\mathcal{A}}^{-1}(\bs{v}), \bs{w}) 
:= (\nabla \tilde{\mathcal{A}}^{-1}(\bs{v}), \nabla \tilde{\mathcal{A}}^{-1}(\bs{w})) \in \mathbb{R}
\]
induces a semi-norm on $H^{-1}(\Omega)^{d}$, denoted by $|\cdot|_{*}$, satisfying
\begin{align}\label{w*}
|\bs{w}|_{*} = \|\nabla \tilde{\mathcal{A}}^{-1}(\bs{w})\|_{0} 
\le C \|\bs{w}\|_{-1}, 
\quad \forall \bs{w} \in H^{-1}(\Omega)^{d}.
\end{align}

\begin{theorem}\label{4.48} 
Suppose that Assumptions (A.1)-(A.3) are satisfied. The following optimal error estimates hold
 \begin{align}\label{4.49}
   &|\epsilon_{\bs{u}}^{N+1}|_{*}^{2}+|2\epsilon_{\bs{u}}^{N+1}-\epsilon_{\bs{u}}^{N}|_{*}^{2}+\zeta|\epsilon_{\bs{\omega}}^{N+1}|_{*}^{2}+\zeta|2\epsilon_{\bs{\omega}}^{N+1}-\epsilon_{\bs{\omega}}^{N-1}|_{*}^{2}
+|\epsilon_{\theta}^{N+1}|_{*}^{2}+|2\epsilon_{\theta}^{N+1}-\epsilon_{\theta}^{N-1}|_{*}^{2}\nn\\
&\quad+\delta_{t}\sum_{n=0}^{N}(\frac{\chi+\mu}{2}\|\epsilon_{\bs{u}}^{n+1}\|_{0}^{2}+4\upsilon\|\epsilon_{\bs{\omega}}^{n+1}\|_{0}^{2}+4\eta\|\epsilon_{\bs{\omega}}^{n+1}\|_{0}^{2}+16\chi\|\epsilon_{\bs{\omega}}^{n+1}\|_{0}^{2} +\frac{\kappa}{2}\|\epsilon_{\theta}^{n+1}\|_{0}^{2} \nn\\
&\quad+|\epsilon_{\bs{u}}^{n+1}-2\epsilon_{\bs{u}}^{n} +\epsilon_{\bs{u}}^{n-1}|_{*}^{2}+\zeta|\epsilon_{\bs{\omega}}^{n+1}-2\epsilon_{\bs{\omega}}^{n}+\epsilon_{\bs{\omega}}^{n-1}|_{*}^{2}
+|\epsilon_{\theta}^{n+1}-2\epsilon_{\theta}^{n}+\epsilon_{\theta}^{n-1}|_{*}^{2})\leqslant C(\delta_{t}^{2}+h^{l+1})^{2}.
\end{align}
\end{theorem}
\begin{proof}

Testing \refe{4.9}, \refe{4.10} and \refe{4.11} by $4\delta_t \mathcal{A}^{-1}\tilde{\epsilon}_{\bs{u}}^{n+1}$, $4\delta_{t}\tilde{\mathcal{A}}^{-1}\epsilon_{\bs{\omega}}^{n+1}$, and $4\delta_{t}\tilde{\mathcal{A}}^{-1}\epsilon_{\theta}^{n+1}$, respectively, we have
\begin{align}
\label{4.50}
&2(3\tilde{\epsilon}_{\bs{u}}^{n+1}-4\epsilon_{\bs{u}}^{n}+\epsilon_{\bs{u}}^{n-1}, \mathcal{A}^{-1}\tilde{\epsilon}_{\bs{u}\bs{u}}^{n+1})  +4(\chi+\mu)\delta_t(\nabla\tilde{\epsilon}_{\bs{u}}^{n+1},\nabla\mathcal{A}^{-1}\tilde{\epsilon}_{\bs{u}}^{n+1}) +4\delta_t((R_{\bs{u}h}^{n+1}\cdot\nabla)R_{\bs{u} h}^{n+1} \nonumber\\
&\quad-(\bar{\bs{u}}_{h}^{n}\cdot\nabla)\tilde{\bs{u}}_{h}^{n+1},\mathcal{A}^{-1}\tilde{\epsilon}_{\bs{u}}^{n+1})
+4\delta_t(\nabla Q_{h}^{n+1}-\nabla p_{h}^{n}, \mathcal{A}^{-1}\tilde{\epsilon}_{\bs{u}}^{n+1})\nonumber\\
    & = 4\chi\delta_t(\nabla\times R_{\bs{\omega} h}^{n+1}-\nabla \times \bar{\bs{\omega}}_{h}^{n},\mathcal{A}^{-1}\tilde{\epsilon}_{\bs{u}}^{n+1})+4\delta_t(R_{\theta h}^{n+1}\bs{e}-\bar{\theta}_{h}^{n}\bs{e}, \mathcal{A}^{-1}\tilde{\epsilon}_{\bs{u}}^{n+1})
    +4\delta_t(R_{\bs{u} l}^{n+1}, \mathcal{A}^{-1}\tilde{\epsilon}_{\bs{u}}^{n+1}),
\\
\label{4.57}
&\zeta|\epsilon_{\bs{\omega}}^{n+1}|_{*}-\zeta|\epsilon_{\bs{\omega}}^{n}|_{*}+\zeta|2\epsilon_{\bs{\omega}}^{n+1}-\epsilon_{\bs{\omega}}^{n}|_{*}-\zeta|2\epsilon_{\bs{\omega}}^{n}-\epsilon_{\bs{\omega}}^{n-1}|_{*}+\zeta|\epsilon_{\bs{\omega}}^{n+1}-2\epsilon_{\bs{\omega}}^{n}+\epsilon_{\bs{\omega}}^{n-1}|_{*}
    +4\upsilon\delta_{t}\|\epsilon_{\bs{\omega}}^{n+1}\|_{0}^{2}\nn\\
    &\quad+4\delta_{t}\zeta((R_{\bs{u}h}^{n+1}\cdot\nabla)R_{\bs{\omega} h}^{n+1}-(\bar{\bs{u}}_{h}^{n}\cdot\nabla)\bs{\omega}_{h}^{n+1}, \tilde{\mathcal{A}}^{-1}\epsilon_{\bs{\omega}}^{n+1})+4\eta\delta_{t}\|\epsilon_{\bs{\omega}}^{n+1}\|_{0}^{2}+16\chi\delta_{t}\|\epsilon_{\bs{\omega}}^{n+1}\|_{0}^{2}\nn\\
    & \quad = 8\chi\delta_{t}(\nabla\times R_{\bs{u} h}^{n+1}-\nabla \times \bar{\bs{u}}_{h}^{n}, \tilde{\mathcal{A}}^{-1}\epsilon_{\bs{\omega}}^{n+1})+4\delta_{t}(R_{\bs{\omega} l}^{n+1}, \tilde{\mathcal{A}}^{-1}\epsilon_{\bs{\omega}}^{n+1}),
\\
\label{4.61}
    &|\epsilon_{\theta}^{n+1}|_{*}-|\epsilon_{\theta}^{n}|_{*}+|2\epsilon_{\theta}^{n+1}-\epsilon_{\theta}^{n}|_{*}-|2\epsilon_{\theta}^{n}-\epsilon_{\theta}^{n-1}|_{*}+|\epsilon_{\theta}^{n+1}-2\epsilon_{\theta}^{n}+\epsilon_{\theta}^{n-1}|_{*}
    +4\kappa\delta_{t}\|\epsilon_{\theta}^{n+1}\|_{0}^{2}\nn\\
    & \quad =-4\delta_{t}((R_{\bs{u}h}^{n+1}\cdot\nabla)R_{\theta h}^{n+1}-(\bar{\bs{u}}_{h}^{n}\cdot\nabla)\theta_{h}^{n+1}, \tilde{\mathcal{A}}^{-1}\epsilon_{\theta}^{n+1}) +4\delta_{t}(R_{\theta h}^{n+1}-\bar{\theta}_{h}^{n}\bs{e}, \tilde{\mathcal{A}}^{-1}\epsilon_{\theta}^{n+1})\nn \\
    & \qquad 
    +4\delta_{t}(R_{\theta l}^{n+1}, \tilde{\mathcal{A}}^{-1}\epsilon_{\theta}^{n+1}).
\end{align}
The combination of \refe{step2}, \refe{2.17} and \refe{2.19'} yields
\begin{align}
\label{4.51}
&2(3\tilde{\epsilon}_{\bs{u}}^{n+1}-4\epsilon_{\bs{u}}^{n}+\epsilon_{\bs{u}}^{n-1}, \mathcal{A}^{-1}\tilde{\epsilon}_{\bs{u}}^{n+1})=2(3\epsilon_{\bs{u}}^{n+1}-4\epsilon_{\bs{u}}^{n}+\epsilon_{\bs{u}}^{n-1}, \mathcal{A}^{-1}\tilde{\epsilon}_{\bs{u}}^{n+1})\nn\\
& \quad =|\epsilon_{\bs{u}}^{n+1}|_{*}^{2}-|\epsilon_{\bs{u}}^{n}|_{*}^{2}+|2\epsilon_{\bs{u}}^{n+1}-\epsilon_{\bs{u}}^{n}|_{*}^{2}-|2\epsilon_{\bs{u}}^{n}-\epsilon_{\bs{u}}^{n-1}|_{*}^{2}+|\epsilon_{\bs{u}}^{n+1}-2\epsilon_{\bs{u}}^{n}+\epsilon_{\bs{u}}^{n-1}|_{*}^{2}.
\\
\label{4.52}
&4(\chi+\mu)\delta_t(\nabla\tilde{\epsilon}_{\bs{u}}^{n+1},\nabla\mathcal{A}^{-1}\tilde{\epsilon}_{\bs{u}}^{n+1})=4(\chi+\mu)\delta_t(\tilde{\epsilon}_{\bs{u}}^{n+1},\tilde{\epsilon}_{\bs{u}}^{n+1})=4(\chi+\mu)\delta_t\|\tilde{\epsilon}_{\bs{u}}^{n+1}\|_{0}^{2}.
\end{align}
Using H\"{o}lder's inequality, Young's inequality, \refe{2.18'}, \refe{2.19'}, \refe{2.22'} and \refe{w*}, we get
\begin{align}\label{4.54}
&|8\chi\delta_{t}(\nabla\times R_{\bs{\omega} h}^{n+1}-\nabla \times \bar{\bs{\omega}}_{h}^{n}, \mathcal{A}^{-1}\tilde{\epsilon}_{\bs{u}}^{n+1})|\nn\\
&\quad=|8\chi\delta_{t}(\nabla\times \bar{R}_{\bs{\omega} h}^{n}-\nabla \times \bar{\bs{\omega}}_{h}^{n}, \mathcal{A}^{-1}\tilde{\epsilon}_{\bs{u}}^{n+1})+8\chi\delta_{t}(\nabla\times (R_{\bs{\omega} h}^{n+1}-2R_{\bs{\omega} h}^{n}+R_{\bs{\omega} h}^{n-1}), \mathcal{A}^{-1}\tilde{\epsilon}_{\bs{u}}^{n+1})|\nn\\
&\quad\leqslant C\delta_{t}|2\epsilon_{\bs{\omega}}^{n}-\epsilon_{\bs{\omega}}^{n-1}|_{*}\|\nabla\mathcal{A}^{-1}\tilde{\epsilon}_{\bs{u}}^{n+1}\|_{1}+C\delta_{t}\|R_{\bs{\omega} h}^{n+1}-2R_{\bs{\omega} h}^{n}+R_{\bs{\omega} h}^{n-1}\|_{0}\|\nabla\mathcal{A}^{-1}\tilde{\epsilon}_{\bs{u}}^{n+1}\|_{0}\nn\\
&\quad\leqslant C\delta_t|\bar{\epsilon}_{\bs{\omega}}^{n}|_{*}^{2}+C\delta_{t}^{5}+\frac{\chi+\mu}{2}\delta_t\|\tilde{\epsilon}_{\bs{u}}^{n+1}\|_{0}^{2},
\\
\label{4.56}
&|4\delta_t(R_{\bs{u} l}^{n+1}, \mathcal{A}^{-1}\epsilon_{\bs{u}}^{n+1})|\leqslant 4\delta_t\|R_{\bs{u} l}^{n+1}\|_{0}\|\mathcal{A}^{-1}\tilde{\epsilon}_{\bs{u}}^{n+1}\|_{0}\nn \\
& \quad \leqslant C\delta_t\|R_{\bs{u} l}^{n+1}\|_{0}^{2}+\frac{(\chi+\mu)}{2}\delta_t\|\nabla\mathcal{A}^{-1}\tilde{\epsilon}_{\bs{u}}^{n+1}\|_{0}^{2}
\leqslant C\delta_t(\delta_t^{2}+h^{l+1})^{2}+\frac{(\chi+\mu)}{2}\delta_t\|\tilde{\epsilon}_{\bs{u}}^{n+1}\|_{0}^{2},
\\
\label{4.55}
&|4\delta_t(R_{\theta h}^{n+1}\bs{e}-\bar{\theta}_{h}^{n}\bs{e}, \mathcal{A}^{-1}\tilde{\epsilon}_{\bs{u}}^{n+1})|
=|4\delta_t(\bar{R}_{\theta h}^{n}\bs{e}-\bar{\theta}_{h}^{n}\bs{e}, \mathcal{A}^{-1}\tilde{\epsilon}_{\bs{u}}^{n+1})
+4\delta_t((R_{\theta h}^{n+1}-2R_{\theta h}^{n}+R_{\theta h}^{n-1})\bs{e}, \mathcal{A}^{-1}\tilde{\epsilon}_{\bs{u}}^{n+1})|\nn\\
& \quad \leqslant 4\delta_t|\bar{\epsilon}_{\theta}^{n}|_{*}\|\nabla\mathcal{A}^{-1}\epsilon_{\bs{u}}^{n+1}\|_{0}+4\delta_t\|R_{\bs{\omega} h}^{n+1}-2R_{\bs{\omega} h}^{n}+R_{\bs{\omega} h}^{n-1}\|_{0}\|\nabla\mathcal{A}^{-1}\tilde{\epsilon}_{\bs{u}}^{n+1}\|_{0}\nn\\
& \quad \leqslant C\delta_t|\bar{\epsilon}_{\theta}^{n}|_{*}^{2}+C\delta_t^{5}+\frac{\chi+\mu}{2}\delta_t\|\tilde{\epsilon}_{\bs{u}}^{n+1}\|_{0}^{2},
\end{align}
\begin{align}
\label{4.53}
&|-4\delta_t((R_{\bs{u}h}^{n+1}\cdot\nabla)R_{\bs{u} h}^{n+1}-(\bar{\bs{u}}_{h}^{n}\cdot\nabla)\tilde{\bs{u}}_{h}^{n+1},\mathcal{A}^{-1}\tilde{\epsilon}_{\bs{u}}^{n+1})|\nn\\
& \quad =| 4\delta_t(((2\epsilon_{\bs{u}}^{n}-\epsilon_{\bs{u}}^{n-1})\cdot\nabla)R_{\bs{u} h}^{n+1},\mathcal{A}^{-1}\tilde{\epsilon}_{\bs{u}}^{n+1})+4\delta_t(((R_{\bs{u}h}^{n+1}-2R_{\bs{u}h}^{n}+R_{\bs{u}h}^{n-1})\cdot\nabla)R_{\bs{u} h}^{n+1},\mathcal{A}^{-1}\tilde{\epsilon}_{\bs{u}}^{n+1})\nn\\
& \qquad+4\delta_t(\bar{\bs{u}}_{h}^{n}\cdot\nabla)\tilde{\epsilon}_{\bs{u}}^{n+1},\mathcal{A}^{-1}\tilde{\epsilon}_{\bs{u}}^{n+1})+4\delta_t(((2\epsilon_{\bs{u}}^{n}-\epsilon_{\bs{u}}^{n-1})\cdot\nabla)\tilde{\epsilon}_{\bs{u} }^{n+1},\mathcal{A}^{-1}\tilde{\epsilon}_{\bs{u}}^{n+1})|\nn\\
& \qquad+4\delta_t(((R_{\bs{u}h}^{n+1}-2R_{\bs{u}h}^{n}+R_{\bs{u}h}^{n-1})\cdot\nabla)\tilde{\bs{u}}_{h}^{n+1},\mathcal{A}^{-1}\tilde{\epsilon}_{\bs{u}}^{n+1})\nn\\
& \quad \leqslant C\delta_t|\bar{\epsilon}_{\bs{u}}^{n}|_{*}^{2}+C\delta_t\|R_{\bs{u}h}^{n+1}-2R_{\bs{u}h}^{n}+R_{\bs{u}h}^{n-1}\|_{0}^{2}+\frac{\chi+\mu}{2}\delta_t\|\nabla\mathcal{A}^{-1}\tilde{\epsilon}_{\bs{u}}^{n+1}\|_{0}^{2}\nn\\
& \qquad +C\delta_t(\|R_{\bs{u}h}^{n+1}-2R_{\bs{u}h}^{n}+R_{\bs{u}h}^{n-1}\|_{0}^{2}+\|\epsilon^{n}\|_{0}^{2}+\|\epsilon^{n-1}\|_{0}^{2})\|\nabla\tilde{\epsilon}_{\bs{u}}^{n+1}\|_{0}^{2}+C\delta_t(\|\epsilon_{\bs{u}}^{n}\|_{0}^{2}+\|\epsilon_{\bs{u}}^{n-1}\|_{0}^{2})\|\nabla\tilde{\epsilon}_{\bs{u}}^{n+1}\|_{0}^{2}\nn\\
& \quad \leqslant C\delta_t|\bar{\epsilon}_{\bs{u}}^{n}|_{*}^{2}+C\delta_t^{5}+C\delta_t(\delta_t^2+h^{2l+2})\|\nabla\tilde{\epsilon}_{\bs{u}}^{n+1}\|_{0}^{2}+\frac{\chi+\mu}{2}\delta_t\|\tilde{\epsilon}_{\bs{u}}^{n+1}\|_{0}^{2},
\\
\label{4.58}
&|-4\delta_{t}\zeta((R_{\bs{u}h}^{n+1}\cdot\nabla)R_{\bs{\omega} h}^{n+1}-(\bar{\bs{u}}_{h}^{n}\cdot\nabla)\bs{\omega}_{h}^{n+1}, \tilde{\mathcal{A}}^{-1}\epsilon_{\bs{\omega}}^{n+1})|\nn\\
& \quad =|-4\delta_{t}\zeta((\bar{\epsilon}_{\bs{u}}^{n}\cdot\nabla)R_{\bs{\omega} h}^{n+1}, \tilde{\mathcal{A}}^{-1}\epsilon_{\bs{\omega}}^{n+1})-4\delta_{t}\zeta(((R_{\bs{u}h}^{n+1}-2R_{\bs{u}h}^{n}+R_{\bs{u}h}^{n-1})\cdot\nabla)R_{\bs{\omega} h}^{n+1}, \tilde{\mathcal{A}}^{-1}\epsilon_{\bs{\omega}}^{n+1})\nn\\
& \qquad-4\delta_{t}\zeta((\bar{\epsilon}_{\bs{u}}^{n}\cdot\nabla)\epsilon_{\bs{\omega}}^{n+1}, \tilde{\mathcal{A}}^{-1}\epsilon_{\bs{\omega}}^{n+1})-4\delta_{t}\zeta(((R_{\bs{u}h}^{n+1}-2R_{\bs{u}h}^{n}+R_{\bs{u}h}^{n-1})\cdot\nabla)\epsilon_{\bs{\omega} }^{n+1}, \tilde{\mathcal{A}}^{-1}\epsilon_{\bs{\omega}}^{n+1})|\nn\\
& \quad \leqslant C\delta_{t}\zeta\|\bar{\epsilon}_{\bs{u}}^{n}\|_{0}\|R_{\bs{\omega} h}^{n+1}\|_{L^{\infty}}\|\nabla\tilde{\mathcal{A}}^{-1}\epsilon_{\bs{\omega}}^{n+1}\|_{0}+4\delta_{t}\zeta\|R_{\bs{u}h}^{n+1}-2R_{\bs{u}h}^{n}+R_{\bs{u}h}^{n-1}\|_{0}\|R_{\bs{\omega} h}^{n+1}\|_{L^{\infty}}\|\nabla\tilde{\mathcal{A}}^{-1}\epsilon_{\bs{\omega}}^{n+1}\|_{0}\nn\\
& \qquad+C\delta_{t}\zeta\|\bar{\epsilon}_{\bs{u}}^{n}\|_{0}\|\nabla\epsilon_{\bs{\omega}}^{n+1}\|_{0}\|\tilde{\mathcal{A}}^{-1}\epsilon_{\bs{\omega}}^{n+1}\|_{2}
+4\delta_{t}\zeta\|R_{\bs{u}h}^{n+1}-2R_{\bs{u}h}^{n}+R_{\bs{u}h}^{n-1}\|_{0}\|\nabla\epsilon_{\bs{\omega}}^{n+1}\|_{0}\|\tilde{\mathcal{A}}^{-1}\epsilon_{\bs{\omega}}^{n+1}\|_{2}\nn\\
& \quad \leqslant \frac{\chi+\mu}{4}\delta_t(\|\epsilon_{\bs{u}}^{n}\|_{0}^{2}+\|\epsilon_{\bs{u}}^{n-1}\|_{0}^{2})+C\delta_t^{5}+C\delta_t(\delta_t^2+h^{2l+2})\|\nabla\epsilon_{\bs{\omega}}^{n+1}\|_{0}^{2}+C\delta_t|\epsilon_{\bs{\omega}}^{n+1}|_{*}^{2},
\\
\label{4.59}
&|8\chi\delta_{t}(\nabla\times R_{\bs{u} h}^{n+1}-\nabla \times \bar{\bs{u}}_{h}^{n}, \tilde{\mathcal{A}}^{-1}\epsilon_{\bs{\omega}}^{n+1})|\nn\\
& \quad=|8\chi\delta_{t}(\nabla\times \bar{R}_{\bs{u} h}^{n}-\nabla \times \bar{\bs{u}}_{h}^{n}, \tilde{\mathcal{A}}^{-1}\epsilon_{\bs{\omega}}^{n+1})+8\chi\delta_{t}(\nabla\times (R_{\bs{u} h}^{n+1}-2R_{\bs{u} h}^{n}+R_{\bs{u} h}^{n-1}), \tilde{\mathcal{A}}^{-1}\epsilon_{\bs{\omega}}^{n+1})|\nn\\
& \quad \leqslant C\delta_{t}(\|\epsilon_{\bs{u}}^{n}\|_{0}+\|\epsilon_{\bs{u}}^{n-1}\|_{0})\|\nabla\tilde{\mathcal{A}}^{-1}\epsilon_{\bs{\omega}}^{n+1}\|_{0}+C\delta_{t}\|R_{\bs{u}h}^{n+1}-2R_{\bs{u}h}^{n}+R_{\bs{u}h}^{n-1}\|_{0}\|\nabla\tilde{\mathcal{A}}^{-1}\epsilon_{\bs{\omega}}^{n+1}\|_{0}\nn\\
& \quad \leqslant \frac{\chi+\mu}{4}\delta_t(\|\epsilon_{\bs{u}}^{n}\|_{0}^{2}+\|\epsilon_{\bs{u}}^{n-1}\|_{0}^{2})+C\delta_t^{5}+C\delta_{t}|\epsilon_{\bs{\omega}}^{n+1}|_{*}^{2},
\\
\label{4.60}
&|4\delta_{t}(R_{\bs{\omega} l}^{n+1}, \tilde{\mathcal{A}}^{-1}\epsilon_{\bs{\omega}}^{n+1})|\leqslant 4\delta_{t}\|R_{\bs{\omega} l}^{n+1}\|_{0} \|\tilde{\mathcal{A}}^{-1}\epsilon_{\bs{\omega}}^{n+1}\|_{0}
\leqslant C\delta_{t}\|R_{\bs{\omega} l}^{n+1}\|_{0}^{2}+\upsilon\delta_{t}\|\nabla\tilde{\mathcal{A}}^{-1}\epsilon_{\bs{\omega}}^{n+1}\|_{0}^{2}\nn\\
& \quad \leqslant C\delta_{t}\|R_{\bs{\omega} l}^{n+1}\|_{0}^{2}+\upsilon\delta_{t}|\epsilon_{\bs{\omega}}^{n+1}|_{*}^{2}
\leqslant C\delta_{t}(\delta_t^{2}+h^{l+1})^{2}+\upsilon\delta_{t}|\epsilon_{\bs{\omega}}^{n+1}|_{*}^{2},
\\
\label{4.62}
&|4\delta_{t}((R_{\bs{u}h}^{n+1}\cdot\nabla)R_{\theta h}^{n+1}-(\bar{\bs{u}}_{h}^{n}\cdot\nabla)\theta_{h}^{n+1}, \tilde{\mathcal{A}}^{-1}\epsilon_{\theta}^{n+1})|\nn\\
& \quad = |4\delta_{t}((\bar{\epsilon}_{\bs{u}}^{n}\cdot\nabla)R_{\theta h}^{n+1}, \tilde{\mathcal{A}}^{-1}\epsilon_{\theta}^{n+1})+
4\delta_{t}(((R_{\bs{u}h}^{n+1}-2R_{\bs{u}h}^{n}+R_{\bs{u}h}^{n-1})\cdot\nabla)R_{\theta h}^{n+1}, \tilde{\mathcal{A}}^{-1}\epsilon_{\theta}^{n+1})\nn\\
& \qquad+4\delta_{t}((\bar{\epsilon}_{\bs{u}}^{n}\cdot\nabla)\epsilon_{\theta}^{n+1}, \tilde{\mathcal{A}}^{-1}\epsilon_{\theta}^{n+1})+4\delta_{t}(((R_{\bs{u}h}^{n+1}-2R_{\bs{u}h}^{n}+R_{\bs{u}h}^{n-1})\cdot\nabla)\epsilon_{\theta}^{n+1}, \tilde{\mathcal{A}}^{-1}\epsilon_{\theta}^{n+1})|\nn\\
& \quad \leqslant 4\delta_{t}\|\bar{\epsilon}_{\bs{u}}^{n}\|_{0}\|R_{\theta h}^{n+1}\|_{L^{\infty}}\|\nabla\tilde{\mathcal{A}}^{-1}\epsilon_{\theta}^{n+1}\|_{0}+
4\delta_{t}\|R_{\bs{u}h}^{n+1}-2R_{\bs{u}h}^{n}+R_{\bs{u}h}^{n-1}\|_{0}\|R_{\theta h}^{n+1}\|_{L^{\infty}}\|\nabla \tilde{\mathcal{A}}^{-1}\epsilon_{\theta}^{n+1}\|_{0}\nn\\
& \qquad+4\delta_{t}\|\bar{\epsilon}_{\bs{u}}^{n}\|_{0}\|\nabla\epsilon_{\theta}^{n+1}\|_{0} \|\tilde{\mathcal{A}}^{-1}\epsilon_{\theta}^{n+1}\|_{2}+4\delta_{t}\|R_{\bs{u}h}^{n+1}-2R_{\bs{u}h}^{n}+R_{\bs{u}h}^{n-1}\|_{0}\|\nabla\epsilon_{\theta}^{n+1}\|_{0} \|\tilde{\mathcal{A}}^{-1}\epsilon_{\theta}^{n+1}\|_{2}\nn\\
& \quad \leqslant \frac{\chi+\mu}{4}\delta_t(\|\epsilon_{\bs{u}}^{n}\|_{0}^{2}+\|\epsilon_{\bs{u}}^{n-1}\|_{0}^{2})+C\delta_t^{5}+C\delta_t(\delta_t^2+h^{2l+2})\|\nabla\epsilon_{\theta}^{n+1}\|_{0}^{2}+C\delta_t|\epsilon_{\theta}^{n+1}|_{*}^{2}.
\\
\label{4.63}
&|4\delta_{t}(R_{\theta h}^{n+1}-\bar{\theta}_{h}^{n}\bs{e}, \tilde{\mathcal{A}}^{-1}\epsilon_{\theta}^{n+1})|\nn\\
& \quad \leqslant 4\delta_{t}(2\|\epsilon_{\theta}^{n}\|_{0}+\|\epsilon_{\theta}^{n-1}\|_{0})\|\tilde{\mathcal{A}}^{-1}\epsilon_{\theta}^{n+1}\|_{0}+4\delta_{t}\|R_{\theta h}^{n+1}-2R_{\theta h}^{n}+R_{\theta h}^{n-1}\|_{0}\|\tilde{\mathcal{A}}^{-1}\epsilon_{\theta}^{n+1}\|_{0}\nn\\
& \quad \leqslant \frac{7\kappa\delta_{t}}{4}(\|\epsilon_{\theta}^{n}\|_{0}^{2}+\|\epsilon_{\theta}^{n-1}\|_{0}^{2})+C\delta_{t}^{5}+C\delta_t|\epsilon_{\theta}^{n+1}|_{*}^{2}.
\\
\label{4.64}
&\frac{\chi+\mu}{2}\delta_t\|\tilde{\epsilon}_{\bs{u}}^{n+1}\|_{0}^{2}=\frac{\chi+\mu}{2}\delta_t\|\epsilon_{\bs{u}}^{n+1}\|_{0}^{2}+\frac{2(\chi+\mu)\delta_t^{3}}{9}\|\nabla (\epsilon_{p}^{n+1}-\epsilon_{p}^{n}))\|_{0}^{2} +\frac{2(\chi+\mu)\delta_t^{3}}{9}\|\nabla (Q_{h}^{n+1}-Q_{h}^{n}))\|_{0}^{2}.
\end{align}
Combining \refe{4.50}-\refe{4.64} and summing up from $n=1$ to $N$, we derive
   \begin{align*}\label{Optimal-Gronwall}
   &\|\epsilon_{\bs{u}}^{N+1}\|_{0}^{2}+\|2\epsilon_{\bs{u}}^{N+1}-\epsilon_{\bs{u}}^{N}\|_{0}^{2}+|\epsilon_{\bs{\omega}}^{N+1}|_{*}^{2}+|2\epsilon_{\bs{\omega}}^{N+1}-\epsilon_{\bs{\omega}}^{N-1}|_{*}^{2}
+|\epsilon_{\theta}^{N+1}|_{*}^{2}+|2\epsilon_{\theta}^{N+1}-\epsilon_{\theta}^{N-1}|_{*}^{2}\nn\\
& \quad +\delta_{t}\sum_{n=0}^{N}(\frac{\chi+\mu}{2}\|\epsilon_{\bs{u}}^{n+1}\|_{0}^{2}+\upsilon\|\epsilon_{\bs{\omega}}^{n+1}\|_{0}^{2}+4\eta\|\epsilon_{\bs{\omega}}^{n+1}\|_{0}^{2}+16\chi\|\epsilon_{\bs{\omega}}^{n+1}\|_{0}^{2}+\frac{\kappa}{2}\|\epsilon_{\theta}^{n+1}\|_{0}^{2}\nn\\
& \quad +|\epsilon_{\bs{u}}^{n+1}-2\epsilon_{\bs{u}}^{n}+\epsilon_{\bs{u}}^{n-1}|_{*}^{2}+|\epsilon_{\bs{\omega}}^{n+1}-2\epsilon_{\bs{\omega}}^{n}+\epsilon_{\bs{\omega}}^{n-1}|_{*}^{2}
+|\epsilon_{\theta}^{n+1}-2\epsilon_{\theta}^{n}+\epsilon_{\theta}^{n-1}|_{*}^{2})\nn\\
& \leqslant \|\epsilon_{\bs{u}}^{1}\|_{0}^{2}+\|2\epsilon_{\bs{u}}^{1}-\epsilon_{\bs{u}}^{0}\|_{0}^{2}+|\epsilon_{\bs{\omega}}^{1}|_{*}^{2}+|2\epsilon_{\bs{\omega}}^{1}-\epsilon_{\bs{\omega}}^{0}|_{*}^{2}
+|\epsilon_{\theta}^{1}|_{*}^{2}+|2\epsilon_{\theta}^{1}-\epsilon_{\theta}^{0}|_{*}^{2}+C\delta_{t}(\|\epsilon_{\bs{u}}^{0}\|_{0}^{2}+\|\epsilon_{\bs{u}}^{1}\|_{0}^{2}+\|\epsilon_{\bs{\omega}}^{0}\|_{0}^{2}\nn\\
& \quad+\|\epsilon_{\bs{\omega}}^{1}\|_{0}^{2}+\|\epsilon_{\theta}^{0}\|_{0}^{2}+\|\epsilon_{\theta}^{1}\|_{0}^{2})+C\delta_t(\delta_{t}^{2}+h^{l+1})^{2}+C\delta_{t}(\sum_{n=1}^{N+1}(|\bar{\epsilon}_{\bs{u}}^{n}|_{*}^{2}+|\epsilon_{\theta}^{n+1}|_{*}^{2}+|\epsilon_{\bs{\omega}}^{n+1}|_{*}^{2}).
   \end{align*}
By Assumption(A.3) and Gr\"{o}nwall's lemma, we complete the proof.
\end{proof}

%% file: Examples.tex
\section{Numerical examples}

In this section, we conduct several 2D and 3D numerical simulations to verify the accuracy and stability of the second-order pressure projection method for the micropolar Rayleigh-B{\'e}nard convection system and illustrate interesting phenomena in benchmark problems. The implementation of our codes is based on the open source element method software \textit{Freefem++} \cite{Hecht2012}. 
\subsection{Accuracy tests in 2D}

 We examine the two-dimensional micropolar Rayleigh-B{\'e}nard convection system with $\Omega=(0,1)^2$ satisfying the exact solution:
\begin{align*}
\left\{
\begin{aligned}
\mathbf{u} &= \left( \sin(2\pi x + t)\sin(2\pi y + t),\ \cos(2\pi x + t)\cos(2\pi y + t) \right), \\
p &= \sin\left(2\pi(x - y) + t\right), \\
\omega &= \sin(2\pi x + t)\sin(2\pi y + t), \\
\theta &= \cos(2\pi x + t)\cos(2\pi y + t).
\end{aligned}
\right.
\end{align*}
The finite element spaces are selected as $P2-P1-P2-P2$ finite element spaces which satisfy the Ladyzhenskaya-Babu\v{s}ka-Brezzi condition. To achieve the optimal convergence rate, we set $\delta_t = h^{\frac{3}{2}}$ and $h = 1/n$, where $n = 8 k$ with $k = 1, 2, \ldots, 7$.

Table \ref{Ta1} presents the numerical results of the model with $\upsilon=0.1$, $\mu=1.0$, $\chi=0.1$, $\kappa=0.1$ and $T=1.0$. It is evident that the errors decrease as the spatial steps diminish, with the convergence rates being optimal. To verify the robustness of the numerical method, the numerical results for various parameters in time $T = 1.0$ are presented in Tables \ref{Ta2} to \ref{Ta4}. From the results, it can be observed that the errors decrease as the mesh is refined. The \( L^2 \) convergence rates for velocity, angular velocity, and temperature reach third order, while the \( H^1 \) convergence rates reach second order. The results show a good agreement between the numerical solutions and the theoretical analysis. 

\begin{table}[ht!] \label{Ta1}
	\caption{Numerical results for $\upsilon=0.1, \chi=0.1, \mu=1.0, \kappa=0.1$ with various $h$ in 2D.}
	\footnotesize
	\centering
	\begin{tabular}{ccccccccc}
		\hline
		\hline
		$1/h$ & $\|\textbf{u}_h-\textbf{u}\|_0$ & $\|\nabla(\textbf{u}_{h}-\textbf{u})\|_0$ &  $\|p_h-p\|_0$ & $\|\omega_h-\omega\|_0$ & $\|\nabla(\omega_{h}-\omega)\|_0$ &  $\|\theta_h-\theta\|_0$ &  $\|\nabla(\theta_h-\theta)\|_0$ \\
		\hline
		8 & 6.1732e-03      & 3.7374e-01       & 1.3940e-01      & 4.3030e-03      &2.5955e-01  &   4.3709e-03   &    2.5926e-01\\
		16 & 7.4497e-04      & 9.5095e-02       & 1.9781e-02      & 4.9888e-04      &6.6827e-02  &  5.0054e-04   &    6.6812e-02\\
		24 &  2.1942e-04      & 4.2423e-02      & 6.7066e-03     &  1.4719e-04     & 2.9878e-02  &   1.4623e-04    &    2.9875e-02\\
		32 & 9.2460e-05     & 2.3892e-02       & 3.1722e-03      & 6.2173e-05      & 1.6842e-02  &   6.1486e-05  &    1.6841e-02\\
		40 & 4.7353e-05       & 1.5297e-02       & 1.7860e-03     & 3.1877e-05      & 1.0790e-02  &   3.1455e-05    &  1.0789e-02\\
		48 & 2.7405e-05     & 1.0625e-02       & 1.1213e-03     & 1.8477e-05      & 7.4970e-03  &   1.8203e-05   &    7.4968e-03\\
		56 & 1.7261e-05   &   7.8061e-03    &   7.5903e-04    &   1.1651e-05  &    5.5098e-03    &   1.1466e-05   &  5.5097e-03 \\
		\hline
		$ 1/h$ & $\textbf{u}_{\text{order}L^2}$ &  $\textbf{u}_{\text{order}H^1}$ & $ p_{\text{order}L^2}$&  $\omega_{\text{order}L^2}$ & $\omega_{\text{order}H^1}$ &  $ \theta_{\text{order}L^2}$ & $ \theta_{\text{order}H^1}$ \\
		\hline
		16 & 3.0508 & 1.9746 & 2.8170 & 3.1086 & 1.9575 & 3.1264 & 1.9562 \\
		24 & 3.0147 & 1.9908 & 2.6676 & 3.0105 & 1.9853 & 3.0348 & 1.9849 \\
		32 & 3.0041 & 1.9958 & 2.6024 & 2.9957 & 1.9926 & 3.0116 & 1.9925 \\
		40 & 2.9987 & 1.9982 & 2.5743 & 2.9938 & 1.9954 & 3.0037 & 1.9955 \\
		48 & 2.9997 & 1.9989 & 2.5531 & 2.9912 & 1.9968 & 3.0000 & 1.9968 \\
		56 & 2.9989 & 2.0000 & 2.5313 & 2.9915 & 1.9979 & 2.9984 & 1.9978 \\
		\hline
		\hline
	\end{tabular}
\end{table}

\begin{table}[ht!]\label{Ta2}
	\caption{Numerical results for $\upsilon=0.1, \chi=0.1, \mu=1.0, \kappa=0.1$ with various $h$ in 2D.}
	\footnotesize
	\centering
	\begin{tabular}{ccccccccc}
		\hline 
		\hline
		$1/h$ & $\|\textbf{u}_h-\textbf{u}\|_0$ & $\|\nabla(\textbf{u}_{h}-\textbf{u})\|_0$ &  $\|p_h-p\|_0$ & $\|\omega_h-\omega\|_0$ & $\|\nabla(\omega_{h}-\omega)\|_0$ &  $\|\theta_h-\theta\|_0$ &  $\|\nabla(\theta_h-\theta)\|_0$ \\
		\hline
		8 & 6.3831e-04      & 3.7854e-01       & 3.3886e-02      & 4.2889e-03      & 2.5957e-01  &   4.3875e-03   &    2.5930e-01\\
		16 & 7.5275e-04      & 9.5312e-02       & 5.6920e-03      & 5.0347e-04      & 6.6829e-02  &   5.0106e-04   &    6.6813e-02\\
		24 & 2.2132e-04      & 4.2462e-02      & 2.2377e-03      &  1.4834e-04     & 2.9879e-02  &   1.4674e-04   &    2.9875e-02\\
		32 & 9.3289e-05     & 2.3904e-02       & 1.1930e-04      & 6.2598e-05      & 1.6842e-02  &   6.1794e-05  &    1.6841e-02\\
		40 & 4.7797e-05       & 1.5303e-02       & 7.4032e-04     & 3.2077e-05     & 1.0790e-02  &   3.1641e-05    &  1.0789e-02\\
		48 & 2.7675e-05      & 1.0628e-02       & 5.0391e-04     & 1.8586e-05      & 7.4970e-03  &   1.8323e-05  &    7.4968e-03\\
		56 & 1.7437e-05   &    7.8081e-03    &    3.6505e-04    &   1.1717e-05  &    5.5098e-03    &   1.1547e-05   &  5.5097e-03 \\
		\hline
		$ 1/h$ & $\textbf{u}_{\text{order}L^2}$ &  $\textbf{u}_{\text{order}H^1}$ & $ p_{\text{order}L^2}$&  $\omega_{\text{order}L^2}$ & $\omega_{\text{order}H^1}$ &  $ \theta_{\text{order}L^2}$ & $ \theta_{\text{order}H^1}$ \\
		\hline
		8  &  -- &   --&     -- &     --&    --& --&  --\\
		16 & 3.0840 & 1.9897 & 2.5737 & 3.0906  &  1.9576   & 3.1304  &  1.9564 \\
		24 & 3.0191 & 1.9941 & 2.3026 & 3.0139 & 1.9854 & 3.0288 & 1.9850 \\
		32 & 3.0030 & 1.9972 & 2.1865 & 2.9989 & 1.9926 & 3.0063 & 1.9925 \\
		40 & 2.9970 & 1.9987 & 2.1381 & 2.9962 & 1.9956 & 2.9997 & 1.9955 \\
		48 & 2.9971 & 1.9996 & 2.1100 & 2.9932 & 1.9970 & 2.9962 & 1.9970 \\
		56 & 2.9964 & 2.0000 & 2.0912 & 2.9930 & 1.9979 & 2.9951 & 1.9978 \\
		\hline
		\hline
	\end{tabular}
\end{table}

\begin{table}[ht!]\label{Ta3}
 \caption{Numerical results for $\upsilon=0.1, \chi=0.1, \mu=0.01, \kappa=0.1$ with various $h$ in 2D.}
	\footnotesize
\centering
\begin{tabular}{ccccccccc}
	\hline 
	\hline
	$1/h$ & $\|\textbf{u}_h-\textbf{u}\|_0$ & $\|\nabla(\textbf{u}_{h}-\textbf{u})\|_0$ &  $\|p_h-p\|_0$ & $\|\omega_h-\omega\|_0$ & $\|\nabla(\omega_{h}-\omega)\|_0$ &  $\|\theta_h-\theta\|_0$ &  $\|\nabla(\theta_h-\theta)\|_0$ 
	\\
	\hline
    8 & 6.4267e-03      & 3.8252e-01       & 2.4320e-02      & 4.3786e-03      & 2.5964e-01  &   4.3623e-03   &    2.5929e-01\\
   16 & 7.5920e-04      & 9.5573e-02       & 4.7747e-03      & 5.1153e-04      & 6.6832e-02  &  5.0221e-04   &    6.6814e-02\\
   24 & 2.2340e-04      & 4.2528e-02      & 2.0077e-03     &  1.5035e-04     & 2.9879e-02  &   1.4764e-04    &    2.9876e-02\\
   32 & 9.4598e-05     & 2.3929e-02       & 1.0993e-03      & 6.3369e-05      & 1.6843e-02  &   6.2240e-05  &    1.6842e-02\\
   40 & 4.8512e-05       & 1.5313e-02       & 6.9091e-04     & 3.2450e-05      & 1.0790e-02  &   3.1935e-05    &  1.0790e-02\\
   48 & 2.8112e-05     & 1.0633e-02       & 4.7390e-04     & 1.8793e-05      & 7.4970e-03  &   1.8509e-05   &    7.4968e-03\\
   56 & 1.7723e-05   &   7.8108e-03    &   3.4549e-04    &   1.1844e-05  &    5.5098e-03    &   1.1671e-05   &  5.5097e-03 \\
		\hline
$ 1/h$ & $\textbf{u}_{\text{order}L^2}$ &  $\textbf{u}_{\text{order}H^1}$ & $ p_{\text{order}L^2}$&  $\omega_{\text{order}L^2}$ & $\omega_{\text{order}H^1}$ &  $ \theta_{\text{order}L^2}$ & $ \theta_{\text{order}H^1}$ \\
\hline
8  &  -- &   --&     -- &     --&    --& --&  --\\
16 & 3.0815 & 2.0009 & 2.3487 & 3.0976 & 1.9579 & 3.1187 & 1.9563 \\
24 & 3.0170 & 1.9970 & 2.1367 & 3.0198 & 1.9854 & 3.0194 & 1.9850 \\
32 & 2.9871 & 1.9990 & 2.0937 & 3.0033 & 1.9925 & 3.0025 & 1.9924 \\
40 & 2.9928 & 2.0005 & 2.0813 & 2.9993 & 1.9956 & 2.9904 & 1.9954 \\
48 & 2.9926 & 2.0005 & 2.0678 & 2.9959 & 1.9971 & 2.9917 & 1.9973 \\
56 & 2.9927 & 2.0010 & 2.0502 & 2.9949 & 1.9979 & 2.9916 & 1.9978
\\
\hline
\hline
\end{tabular}
\end{table}

\begin{table}[ht!]\label{Ta4}
	 \caption{Numerical results for $\upsilon=0.1, \chi=0.1, \mu=0.001, \kappa=0.1$ with various $h$ in 2D.}	\footnotesize
	 \centering
	 \begin{tabular}{ccccccccc}
	 	\hline 
	 	\hline
	 	$1/h$ & $\|\textbf{u}_h-\textbf{u}\|_0$ & $\|\nabla(\textbf{u}_{h}-\textbf{u})\|_0$ &  $\|p_h-p\|_0$ & $\|\omega_h-\omega\|_0$ & $\|\nabla(\omega_{h}-\omega)\|_0$ &  $\|\theta_h-\theta\|_0$ &  $\|\nabla(\theta_h-\theta)\|_0$ 
	 	\\
	 	\hline
    8 & 6.4307e-03      & 3.8350e-01       & 2.3522e-02      & 4.4017e-03      & 2.5965e-01  &   4.3582e-03   &    2.5929e-01\\
   16 & 7.6082e-04      & 9.5647e-02       & 4.7114e-03      & 5.1356e-04      & 6.6833e-02  &  5.0267e-04   &    6.6815e-02\\
   24 & 2.2467e-04      & 4.2546e-02      & 1.9901e-03     &  1.5086e-04     & 2.9879e-02  &   1.4790e-04    &    2.9876e-02\\
   32 & 9.4920e-05     & 2.3935e-02       & 1.0910e-03      & 6.3568e-05      & 1.6843e-02  &   6.2434e-05  &    1.6842e-02\\
   40 & 4.8688e-05       & 1.5316e-02       & 6.8632e-04     & 3.2547e-05      & 1.0790e-02  &   3.2011e-05    &  1.0790e-02\\
   48 & 2.8218e-05     & 1.0634e-02       & 4.7115e-04     & 1.8847e-05      & 7.4970e-03  &   1.8556e-05   &    7.4968e-03\\
   56 & 1.7793e-05   &   7.8114e-03    &   3.4381e-04    &   1.1877e-05  &    5.5098e-03    &   1.1702e-05   &  5.5097e-03 \\
  \hline
  $ 1/h$ & $\textbf{u}_{\text{order}L^2}$ &  $\textbf{u}_{\text{order}H^1}$ & $ p_{\text{order}L^2}$&  $\omega_{\text{order}L^2}$ & $\omega_{\text{order}H^1}$ &  $ \theta_{\text{order}L^2}$ & $ \theta_{\text{order}H^1}$ \\
  \hline
8  &  -- &   --&     -- &     --&    --& --&  --\\
16 & 3.0793 & 2.0034 & 2.3198 & 3.0995 & 1.9579 & 3.1160 & 1.9563 \\
24 & 3.0083 & 1.9979 & 2.1255 & 3.0213 & 1.9855 & 3.0173 & 1.9851 \\
32 & 2.9950 & 1.9996 & 2.0894 & 3.0042 & 1.9925 & 2.9978 & 1.9924 \\
40 & 2.9918 & 2.0007 & 2.0772 & 3.0000 & 1.9956 & 2.9937 & 1.9954 \\
48 & 2.9918 & 2.0011 & 2.0632 & 2.9965 & 1.9971 & 2.9908 & 1.9973 \\
56 & 2.9916 & 2.0011 & 2.0440 & 2.9954 & 1.9979 & 2.9908 & 1.9978
\\
\hline
\hline
\end{tabular}
\end{table}

\subsection{Accuracy tests in 3D}
We consider the three-dimensional ($\Omega=(0,1)^3$) micropolar Rayleigh-B{\'e}nard convection system with the following exact solution: 
\begin{align*}
\left\{
\begin{aligned}
\mathbf{u} &= \left(
\sin(t)\sin(\pi x)\sin(\pi (y + 0.5))\sin(\pi z),\ 
\sin(t)\cos(\pi x)\cos(\pi (y + 0.5))\sin(\pi z),\ 
0
\right), \\
p &= \sin(t)(2x - 1)(2y - 1)(2z - 1), \\
\omega &= \left(
\sin(2\pi x + t),\ 
\sin(2\pi y + t),\ 
\sin(2\pi z + t)
\right), \\
\theta &= \cos(2\pi x + t)\cos(2\pi y + t)\cos(2\pi z + t).
\end{aligned}
\right.
\end{align*}

The three-dimensional distributions of velocity, pressure, angular velocity, and temperature at time $t = 1$, $h=1/32$, $\delta_t=h^2$ are plotted in \Cref{fig:total03}  - \Cref{fig:total04}. The results indicate that all physical quantities are uniformly distributed. To verify the accuracy of the numerical scheme, we compute the errors and convergence rates. As shown in Table \ref{Ta7}, the spatial mesh sizes are taken as $h = 1/n$, where $n = 4 k $ with $k =2, 3, \dots, 8$, and the time step is set to \( \delta_t = h^2 \). From the results, it can be observed that the errors decrease as the mesh is refined. The \( L^2 \) convergence rates for velocity, angular velocity, and temperature attain third order, while the corresponding \( H^1 \) convergence rates attain second order, all of which are optimal and consistent with the theoretical predictions. 
\begin{table}[ht!]\label{Ta7}
	 \caption{Numerical results for $\upsilon=1.0, \chi=1.0, \mu=1.0, \kappa=1.0$ with various $h$ in 3D.}\footnotesize
	 \centering
	 \begin{tabular}{ccccccccc}
	 	\hline 
	 	\hline
	 	$1/h$ & $\|\textbf{u}_h-\textbf{u}\|_0$ & $\|\nabla(\textbf{u}_{h}-\textbf{u})\|_0$ &  $\|p_h-p\|_0$ & $\|\omega_h-\omega\|_0$ & $\|\nabla(\omega_{h}-\omega)\|_0$ &  $\|\theta_h-\theta\|_0$ &  $\|\nabla(\theta_h-\theta)\|_0$ 
	 	\\
	 	\hline
   8 & 7.6130e-04      & 5.3669e-02      & 1.5355e-02    & 3.3160e-03   & 1.7524e-01  &     6.2245e-03      & 3.3628e-01\\
   12& 2.2484e-04      & 2.4161e-02      & 3.8462e-03    & 9.8079e-04   & 7.8294e-02  &     1.7056e-03      &1.5665e-01\\
   16& 9.4811e-05      & 1.3656e-02      & 1.5023e-03    & 4.1404e-04   & 4.4122e-02  &     6.8933e-04      & 8.9855e-02\\
   20& 4.8506e-05      & 8.7564e-03      & 7.6056e-04    & 2.1255e-04   & 2.8262e-02  &     3.4461e-04      & 5.8076e-02\\
   24& 2.8023e-05      & 6.0827e-03      & 4.5769e-04    & 1.2370e-04   & 1.9636e-02  &     1.9662e-04      & 4.0557e-02\\
   28& 1.7536e-05      & 4.4672e-03      & 3.0906e-04    & 7.8866e-05   & 1.4430e-02    &   1.2270e-04      & 2.9901e-02 \\
   32& 1.1892e-05      & 3.4234e-03      & 2.2695e-04    & 5.4536e-05   & 1.1051e-02    &   8.1747e-05      & 2.2946e-02 \\
  \hline
  $ 1/h$ & $\textbf{u}_{\text{order}L^2}$ &  $\textbf{u}_{\text{order}H^1}$ & $ p_{\text{order}L^2}$&  $\omega_{\text{order}L^2}$ & $\omega_{\text{order}H^1}$ &  $ \theta_{\text{order}L^2}$ & $ \theta_{\text{order}H^1}$ \\
  \hline
8  &  -- &   --&     -- &     --&    --& --&  --\\
12 & 3.0080 & 1.9683 & 3.4142 & 3.0043 & 1.9871 & 3.1930 & 1.8842 \\
16 & 3.0016 & 1.9833 & 3.2677 & 2.9977 & 1.9936 & 3.1491 & 1.9320 \\
20 & 3.0035 & 1.9915 & 3.0506 & 2.9882 & 1.9962 & 3.1070 & 1.9559 \\
24 & 3.0093 & 1.9983 & 2.7855 & 2.9688 & 1.9974 & 3.0779 & 1.9693 \\
28 & 3.0410 & 2.0025 & 2.5473 & 2.9201 & 1.9981 & 3.0586 & 1.9774 \\
32 & 2.9087 & 1.9928 & 2.3125 & 2.7626 & 1.9984 & 3.0415 & 1.9828
 \\
 \hline
 \hline
 \end{tabular}
\end{table}

\begin{figure}[ht!]
 \centering
  \subfigure[Multiple isosurfaces of velocity]{
    \includegraphics[width=0.45\textwidth]{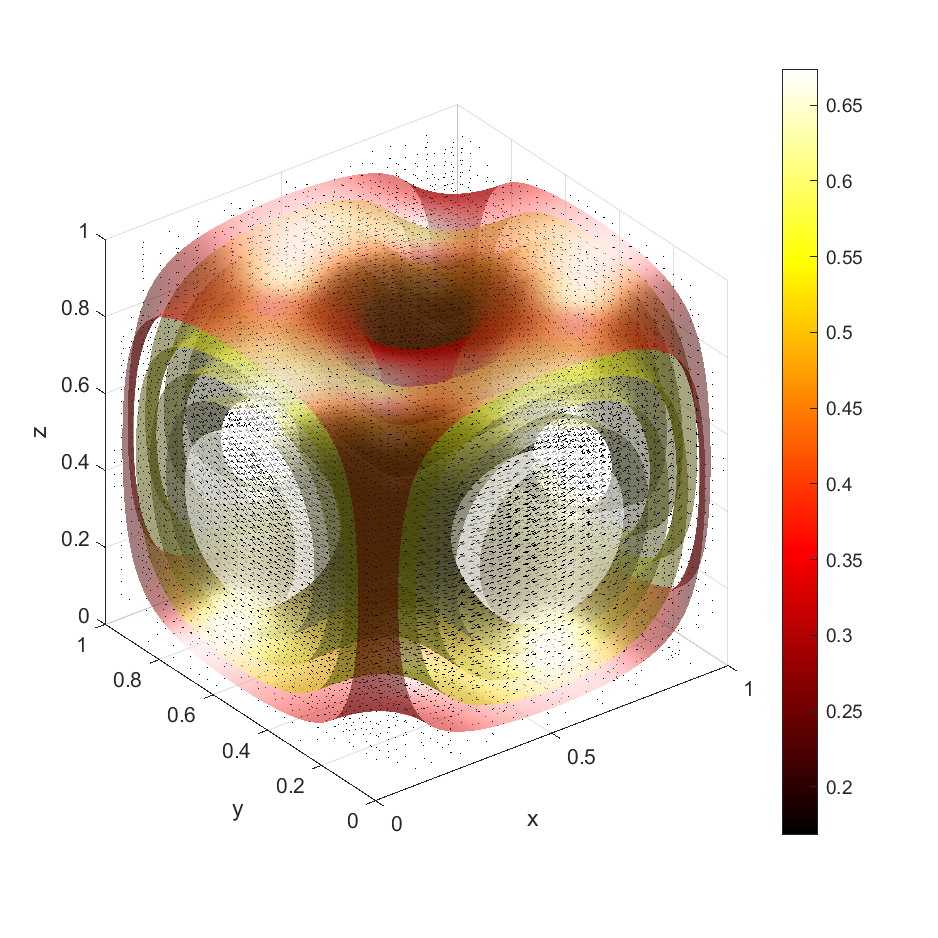}
    \label{fig:sub-u}
  } 
  \subfigure[Top view of velocity]{
    \includegraphics[width=0.45\textwidth]{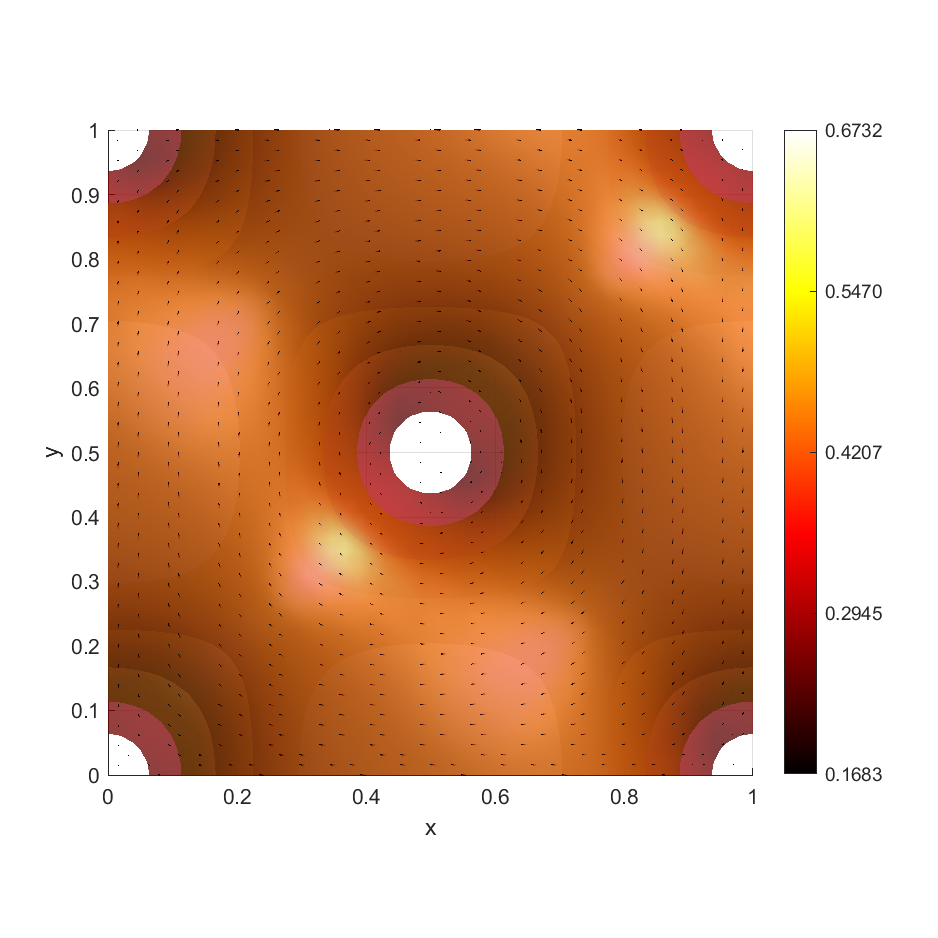}
    \label{fig:sub-u-top}
  }
 \subfigure[Multiple isosurfaces of pressure]{
	\includegraphics[width=0.45\textwidth]{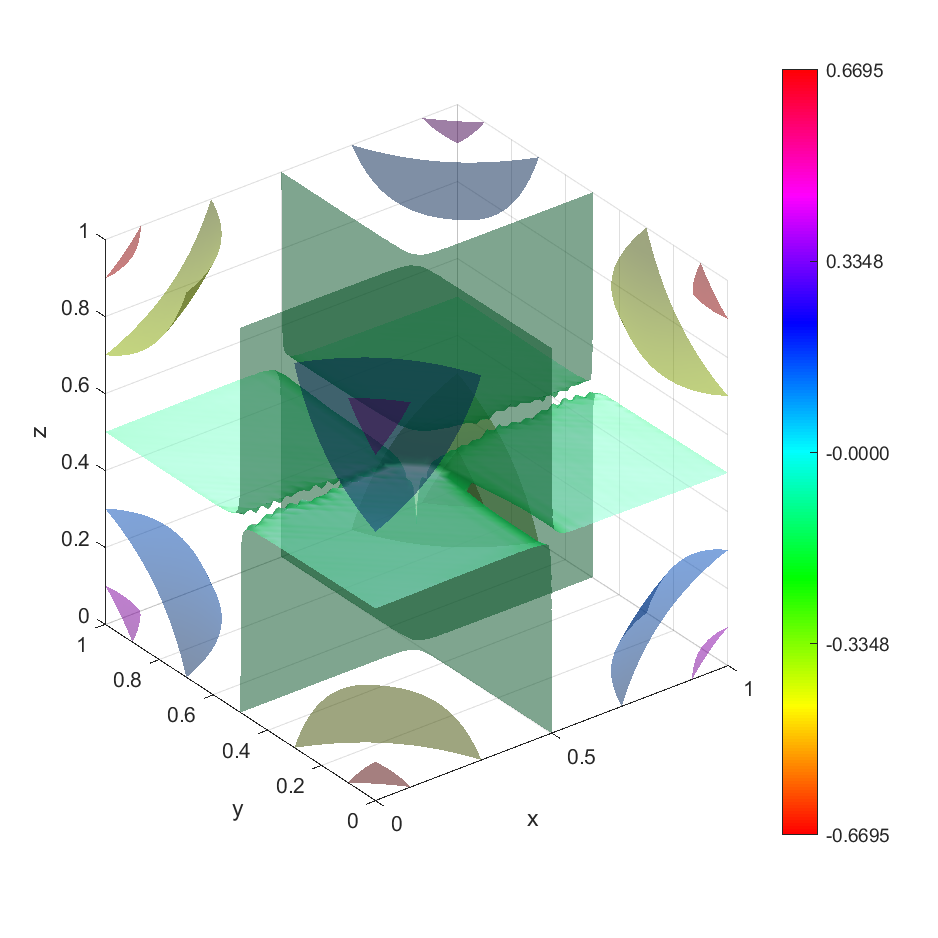}
	\label{fig:sub-p}  }
	\subfigure[Top view of pressure]{
		\includegraphics[width=0.45\textwidth]{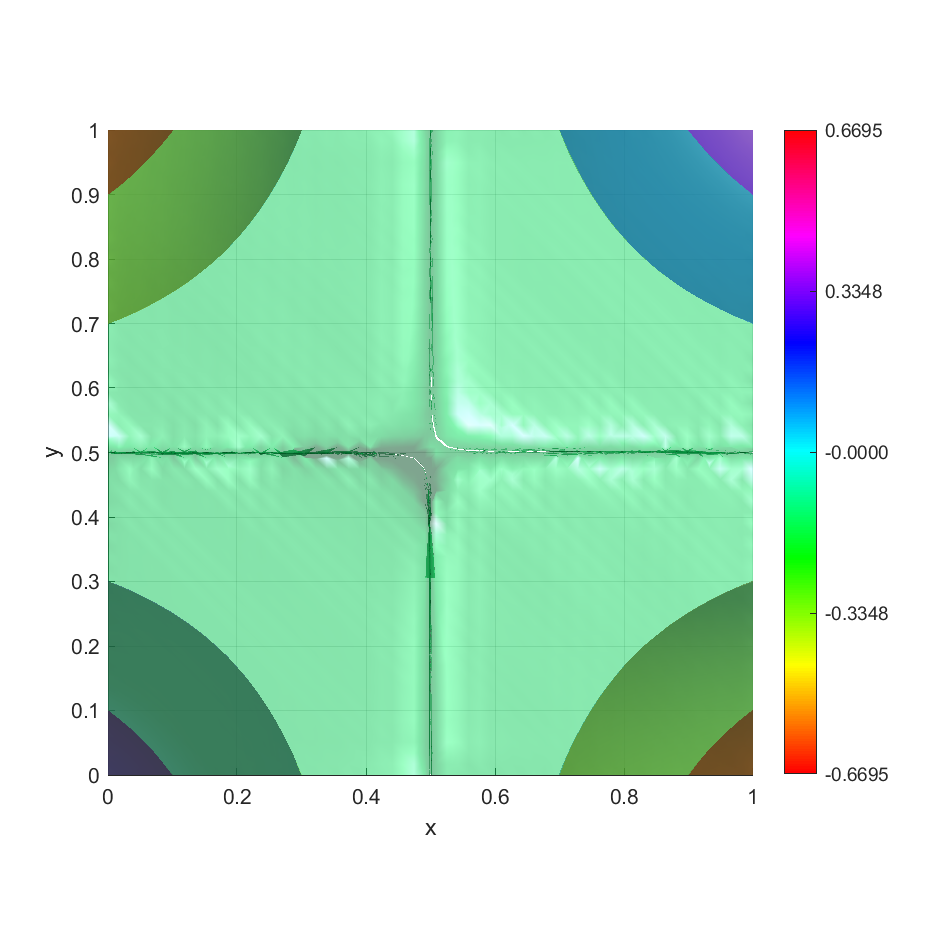}
		\label{fig:sub-p-top}}
\caption{Numerical distributions of the velocity and pressure.}
\label{fig:total03}
\end{figure}

\begin{figure}[ht!]
\centering
 \subfigure[Multiple isosurfaces of angular velocity.]{
      \includegraphics[width=0.45\textwidth]{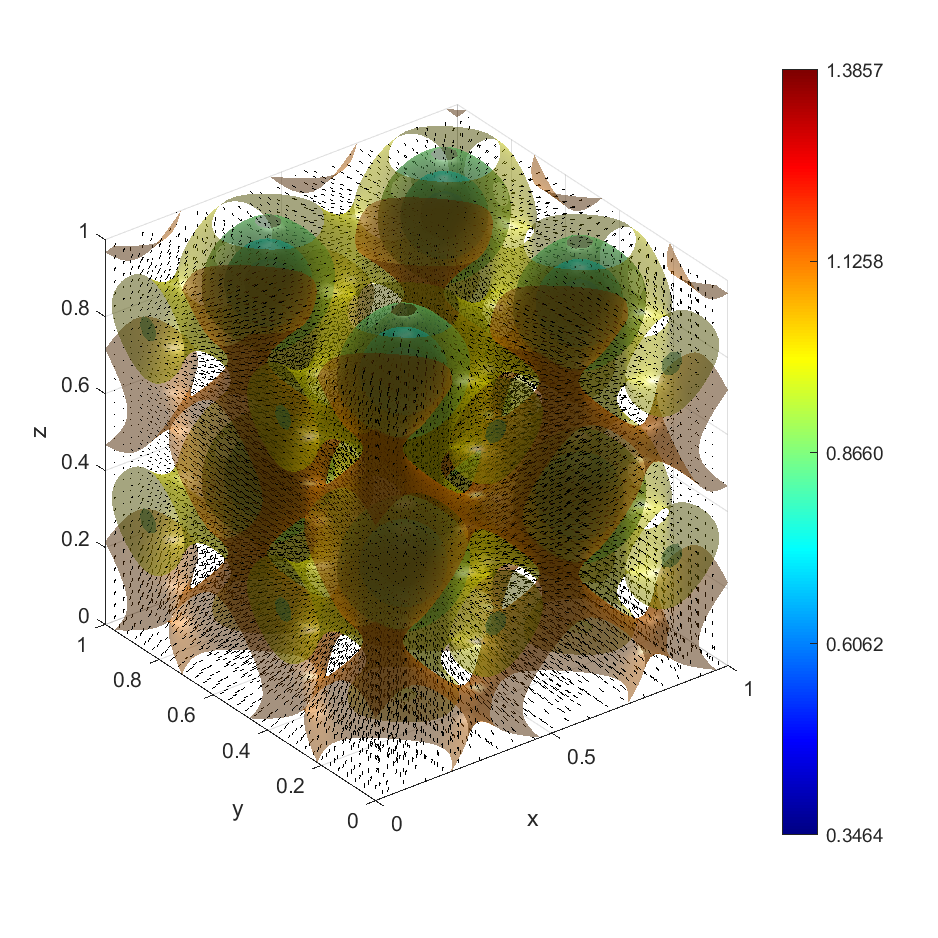}
   \label{fig:sub-w}
  }
  \subfigure[Top view of angular velocity.]{
   \includegraphics[width=0.45\textwidth]{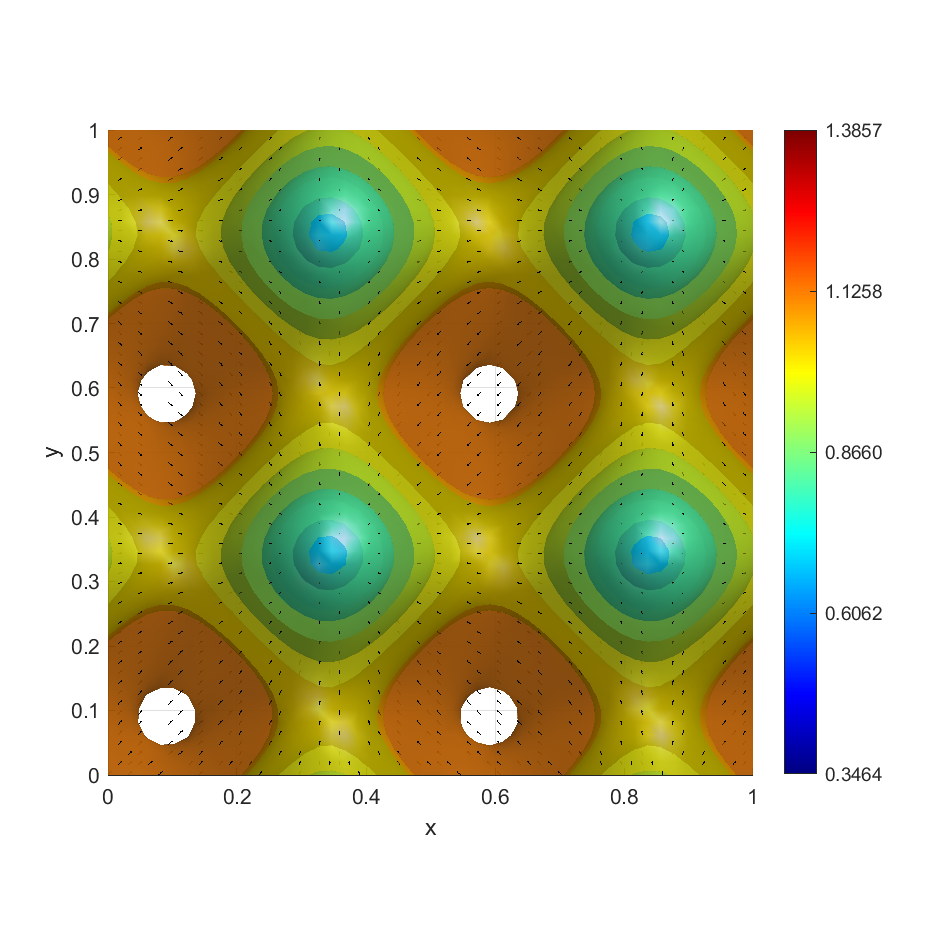}
    \label{fig:sub-w-top}
  }
    \subfigure[Multiple isosurfaces of temperature.]{
  	\includegraphics[width=0.45\textwidth]{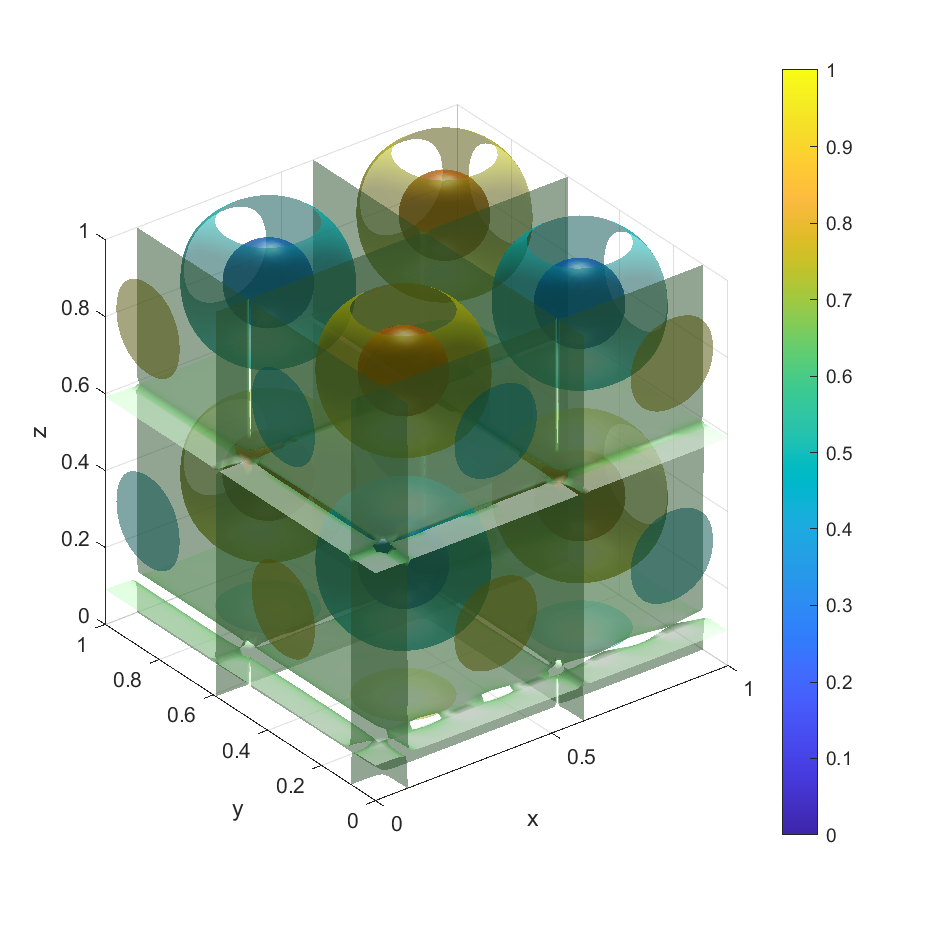}
  	\label{fig:sub-theta}
  }
 \subfigure[Top view of temperature.]{
 	\includegraphics[width=0.45\textwidth]{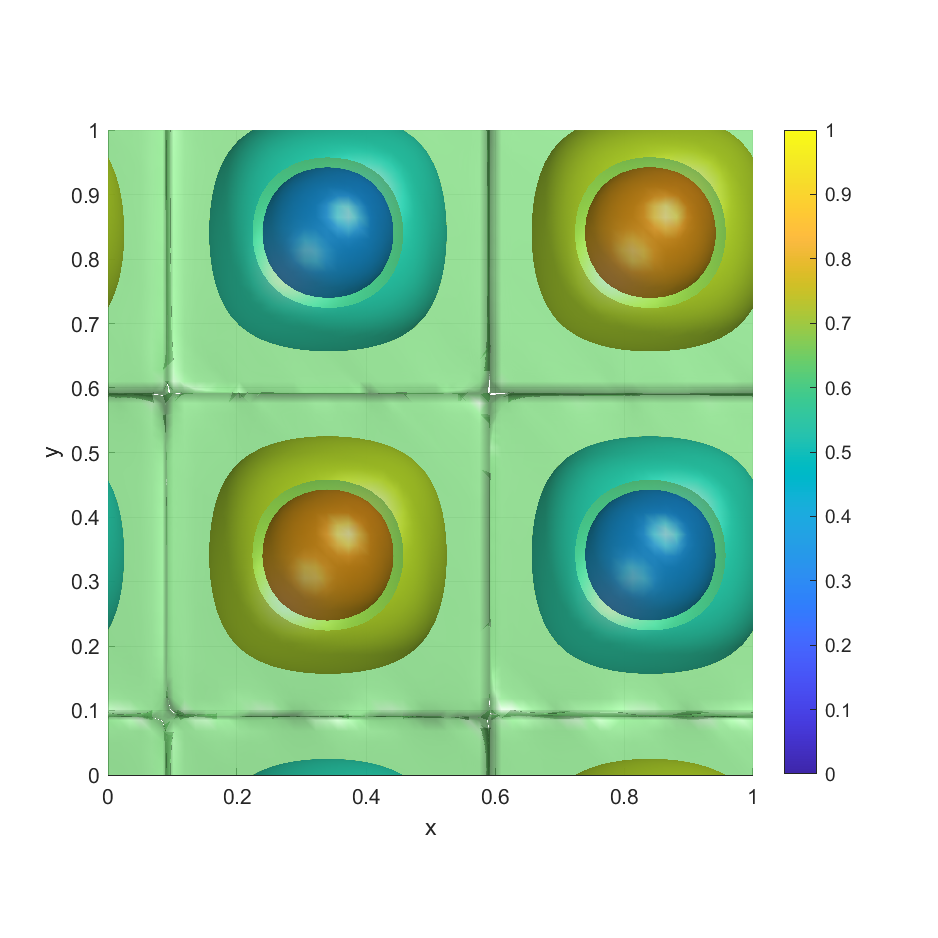}
 	\label{fig:sub-theta-top}
 }
  \caption{Numerical distributions of the angularr velocity and temperature.}
  \label{fig:total04}
\end{figure}

\subsection{Stirring of a passive scalar}

 In fluid dynamics, a passive scalar, such as temperature, concentration, or dye, is transported by the flow without affecting it. This numerical experiment confirms the stability of the second-order pressure projection scheme and its potential for engineering applications. In laminar flows, weak diffusion limits mixing and slows reactant dispersion. Active ferromagnetic nanoparticles can overcome this limitation by being manipulated with external magnetic fields to stir the fluid and enhance passive scalar transport. The dynamics of such systems can be modeled using the micropolar Navier-Stokes equations\cite{Rajagopal1982, Khanafer2003, ZhangHuangHeYang2025}. In this work, we adopt the micropolar Rayleigh-B{\'e}nard convection system to describe and analyze this process.

We consider system \refe{1.1} in the domain $\Omega=(-1,1)^2$ and set $\upsilon =2.0, \chi=0.1, \kappa=1.0$\ and\ $ \mu=0.1.$
 The micropolar Rayleigh-B{\'e}nard convection system is further supplemented with a convertive equation given by
$$
\phi_t + (\mathbf{u} \cdot \nabla) \phi = 0,  \phi_{t=0}=0.5(1 - \tanh(y/\delta)), \qquad \delta=0.5h.
$$
There is no diffusion term in this equation, so any observed mixing is solely induced by the flow pattern. To avoid nonphysical undershoot or overshoot caused by numerical dispersion, a simple limiter is applied at each time step:
$
\phi(x) \leftarrow \min\bigl(1, \max(0, \phi(x))\bigr).
$
This limiter ensures that the scalar field \(\phi\) remains within the physically admissible bounds throughout the simulation.
The velocity, angular velocity field and pressure field
 at time $t=50$ are displayed in Figure\Cref{fig:total13}.
Due to the applied torque $f_w=25(x-1)$, the linear velocity field $\mathbf{u}$ exhibits a rotational pattern, as shown in \Cref{fig:sub-u-string}. \Cref{fig:total14} illustrates the temporal evolution of the passive scalar $\phi$ under the applied laminar flow. Initially, $\phi$ is initialized as a smooth profile (\Cref{fig:sub-T=0}), which is gradually deformed and stretched by the flow field over time. As the simulation progresses, the scalar exhibits complex spiral structures due to continuous advection, indicating enhanced mixing (\Cref{fig:sub-T=1}-\Cref{fig:sub-T=45}). By $t=50$, the scalar distribution demonstrates a highly folded pattern, suggesting that the flow efficiently enhances the dispersion of the passive scalar. This sequence not only confirms the stability of the numerical scheme but also visualizes the mixing process induced by the underlying velocity field.

\begin{figure}[ht!]
  \centering
  \subfigure[Velocity]{
    \includegraphics[width=0.22\textwidth]{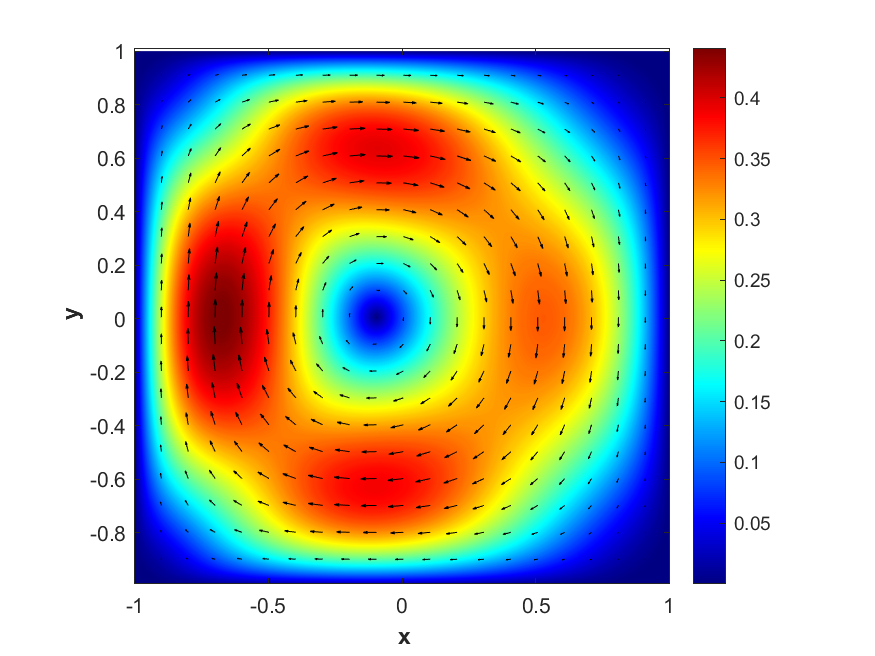}
    \label{fig:sub-u-string}
 }
  \subfigure[Pressure]{
    \includegraphics[width=0.22\textwidth]{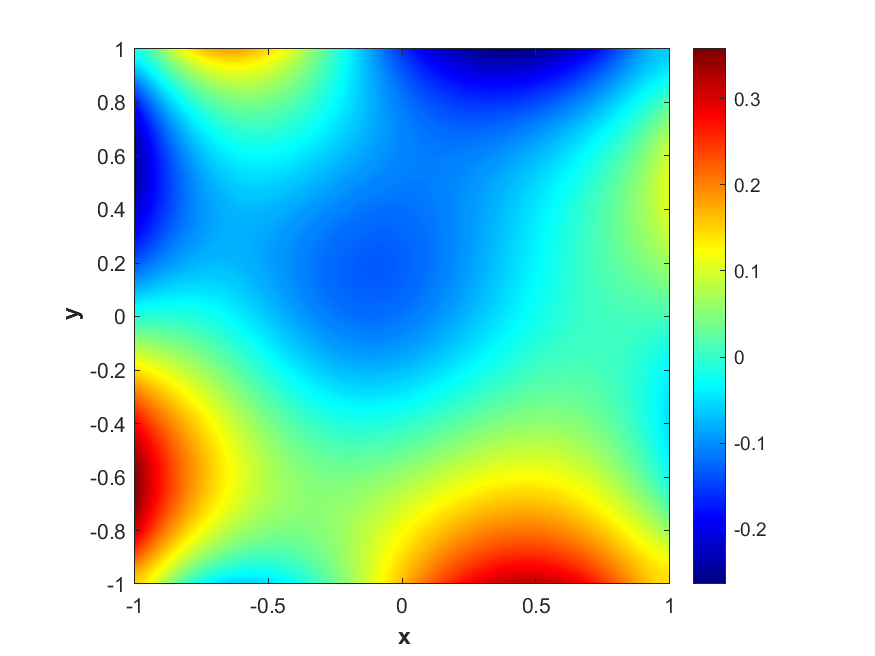} 
    \label{fig:sub-p-string}
  }
  \subfigure[Angular velocity]{
    \includegraphics[width=0.22\textwidth]{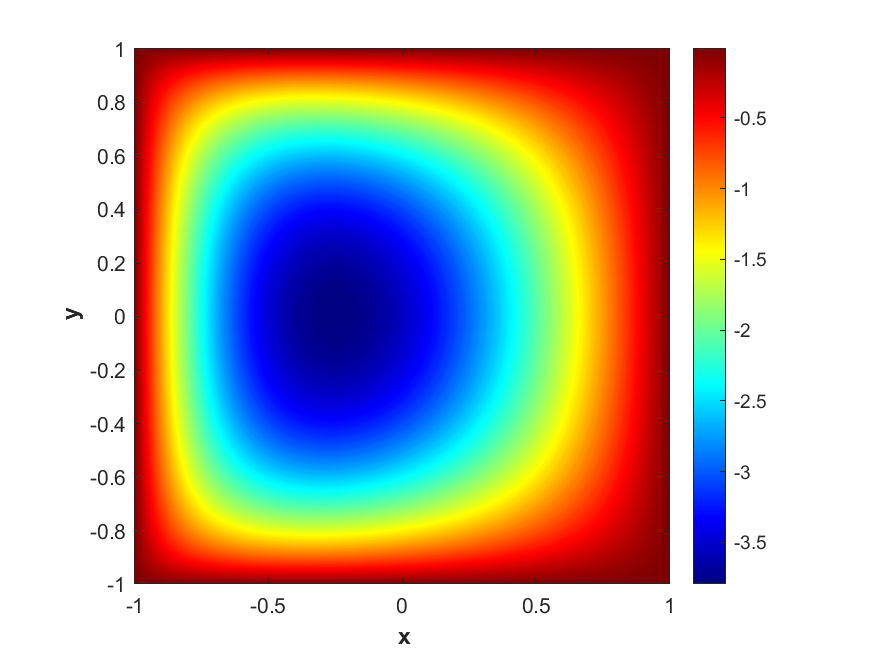}
    \label{fig:sub-w-string}
  }
  \subfigure[Temperature]{
    \includegraphics[width=0.22\textwidth]{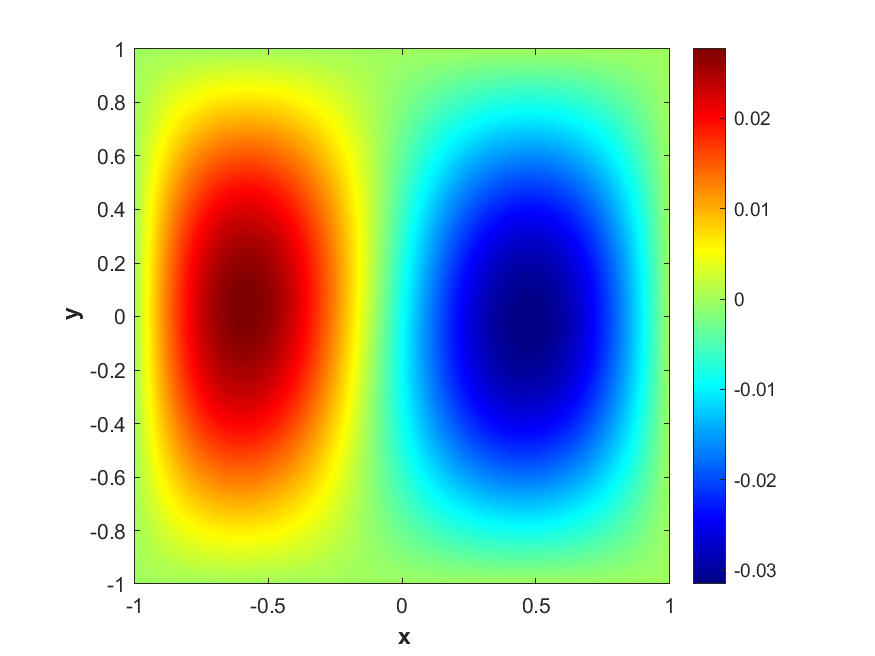} 
    \label{fig:sub-theta-string}
  }
  \caption{Distribution map of numerical solutions for velocity $\mathbf{u}$, angular velocity $w$, pressure $p$ and temperature $\theta$ at $t=50$.}
  \label{fig:total13}
\end{figure}

\begin{figure}[ht!]
  \centering
  \subfigure[$t=0$]{
    \includegraphics[width=0.175\textwidth]{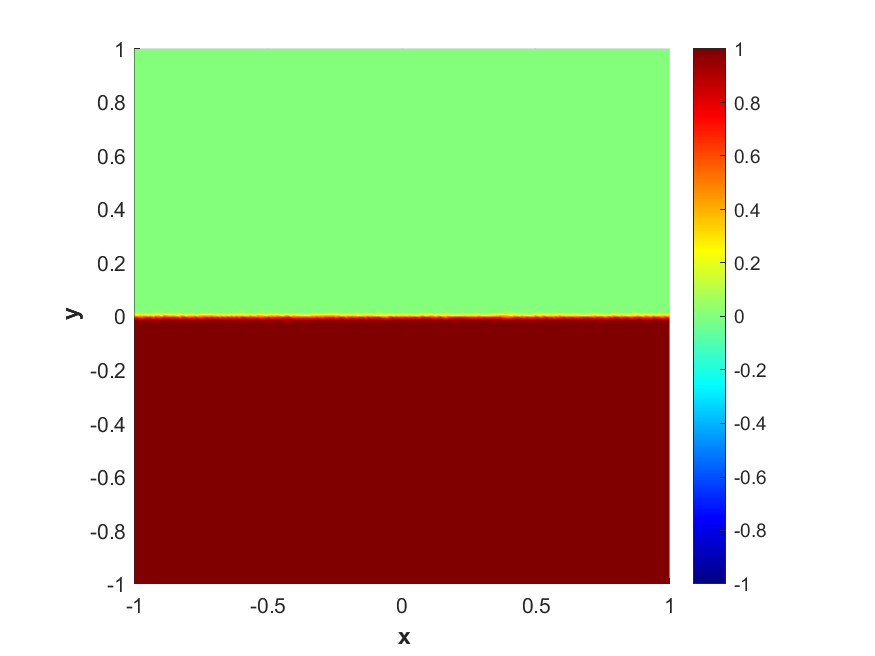}  
    \label{fig:sub-T=0}
  }%
  \subfigure[$t=1$]{
    \includegraphics[width=0.175\textwidth]{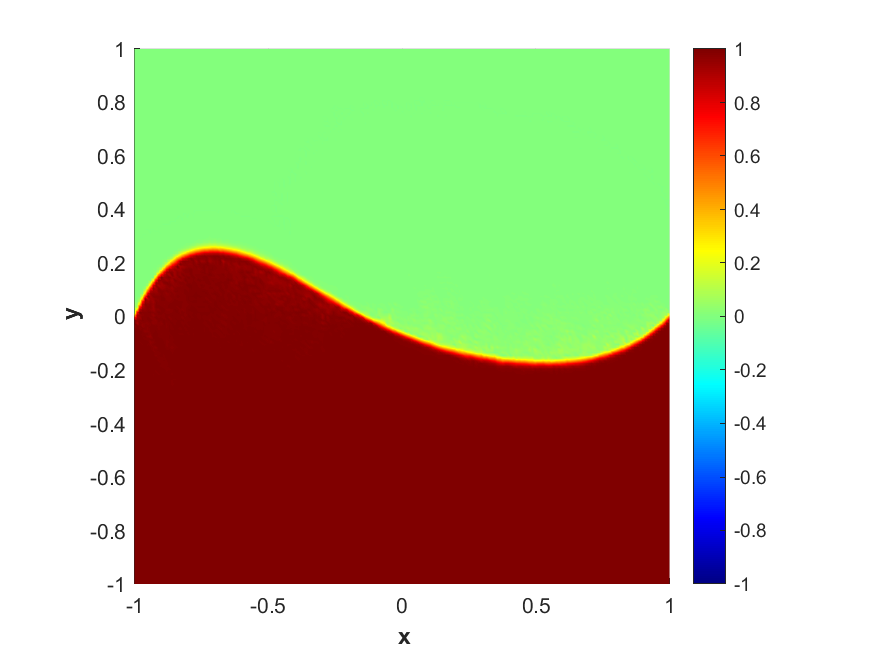} 
    \label{fig:sub-T=1}
  }%
  \subfigure[$t=2$]{
    \includegraphics[width=0.175\textwidth]{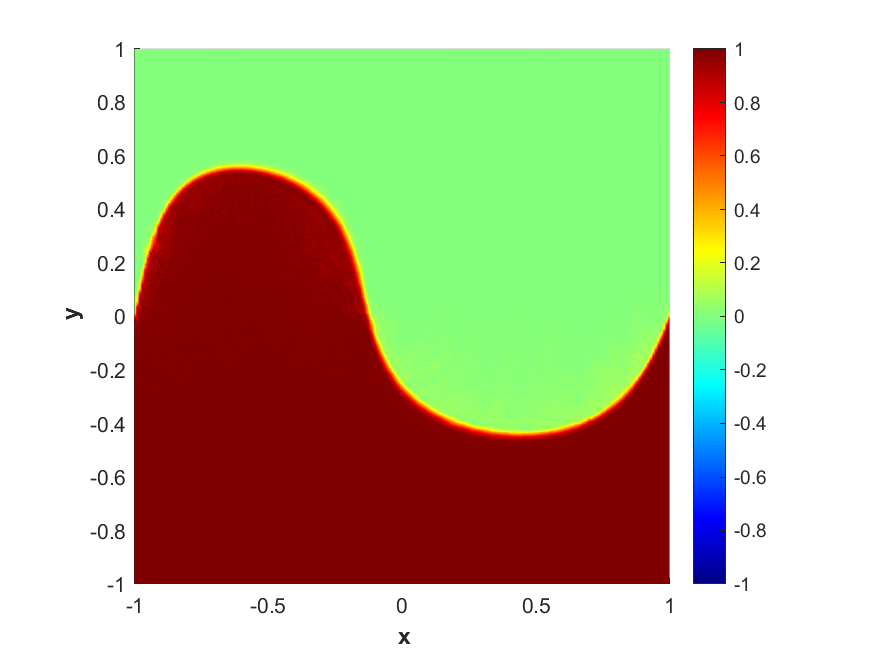}    
    \label{fig:sub-T=2}
  }%
  \subfigure[$t=3$]{
    \includegraphics[width=0.175\textwidth]{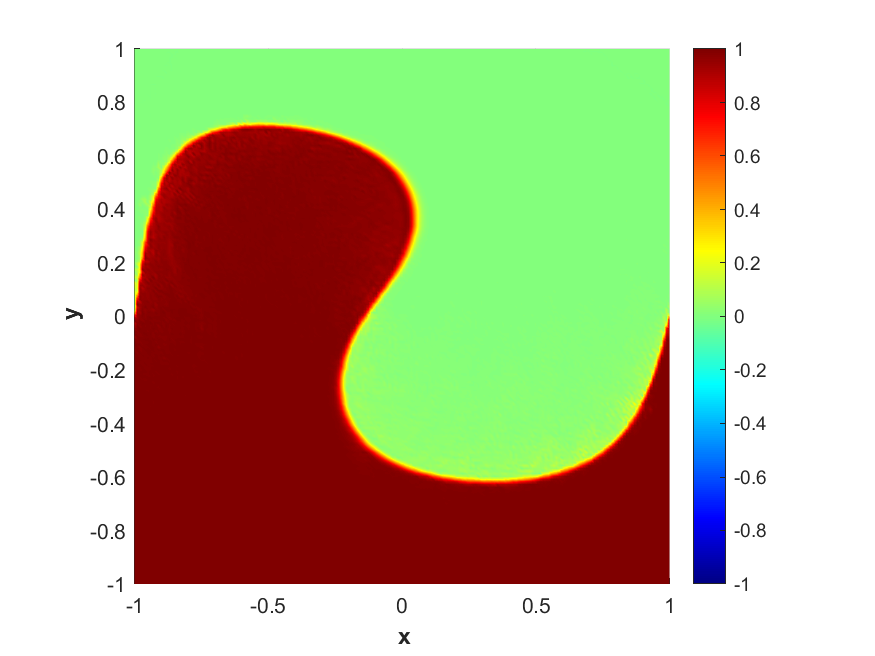}    
    \label{fig:sub-T=3}
  }
  \subfigure[$t=4$]{
    \includegraphics[width=0.175\textwidth]{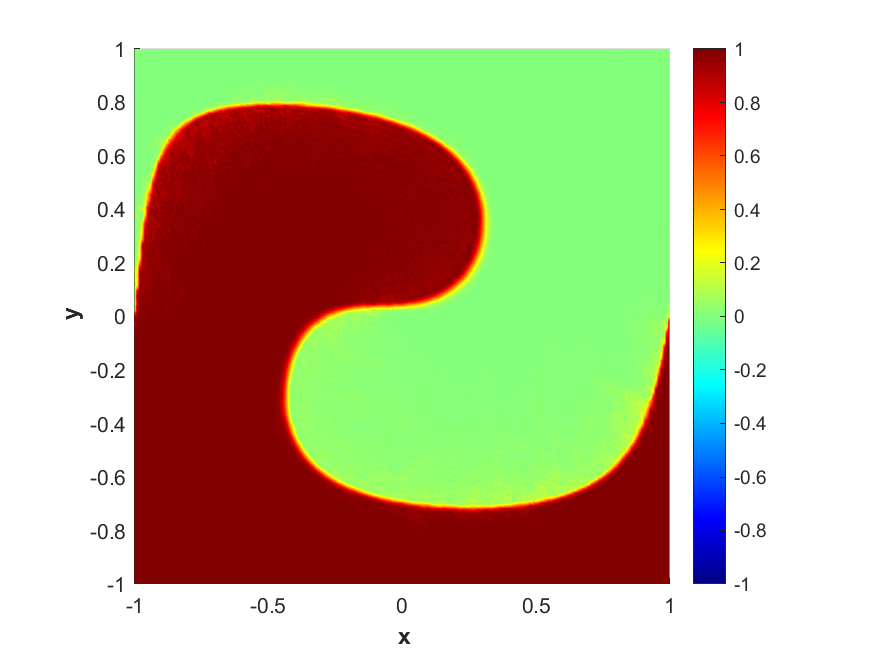}  
    \label{fig:sub-T=4}
  }\\%
  \subfigure[$t=5$]{
    \includegraphics[width=0.175\textwidth]{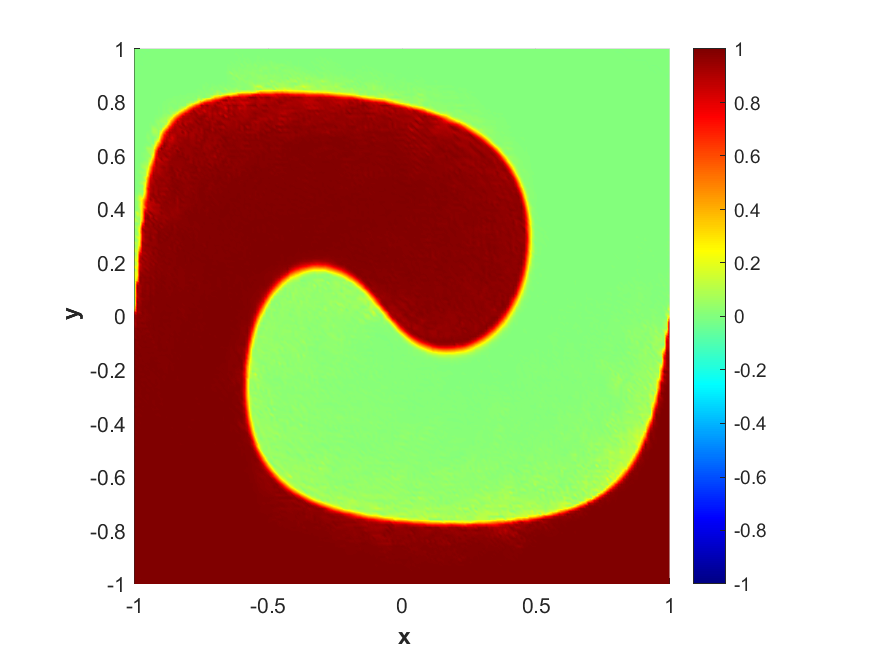}  
    \label{fig:sub-T=5}
  }%
    \label{fig:sub-T=6}
  \subfigure[$t=10$]{
    \includegraphics[width=0.175\textwidth]{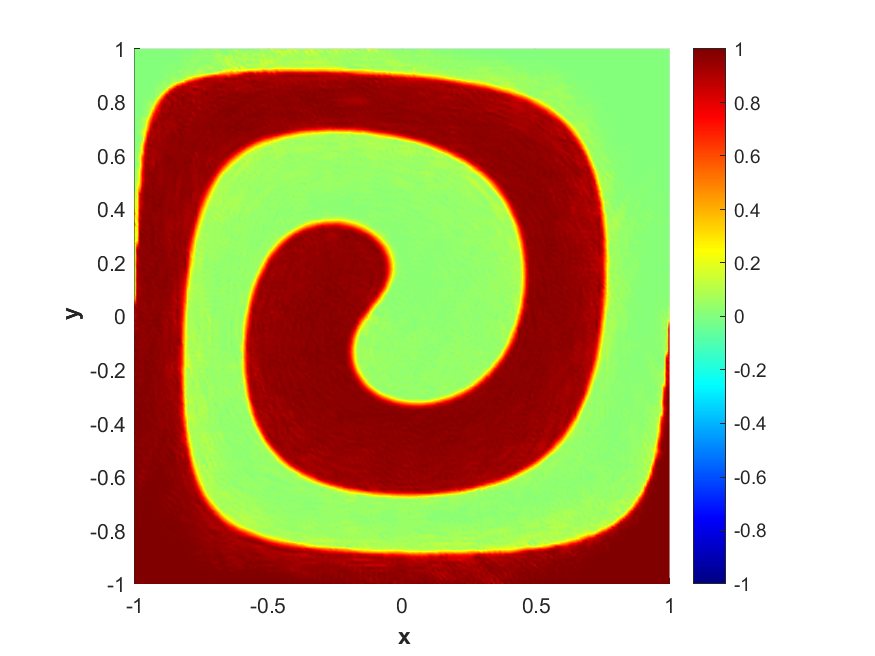}
    \label{fig:sub-T=10}
  }
  \subfigure[$t=15$]{
    \includegraphics[width=0.175\textwidth]{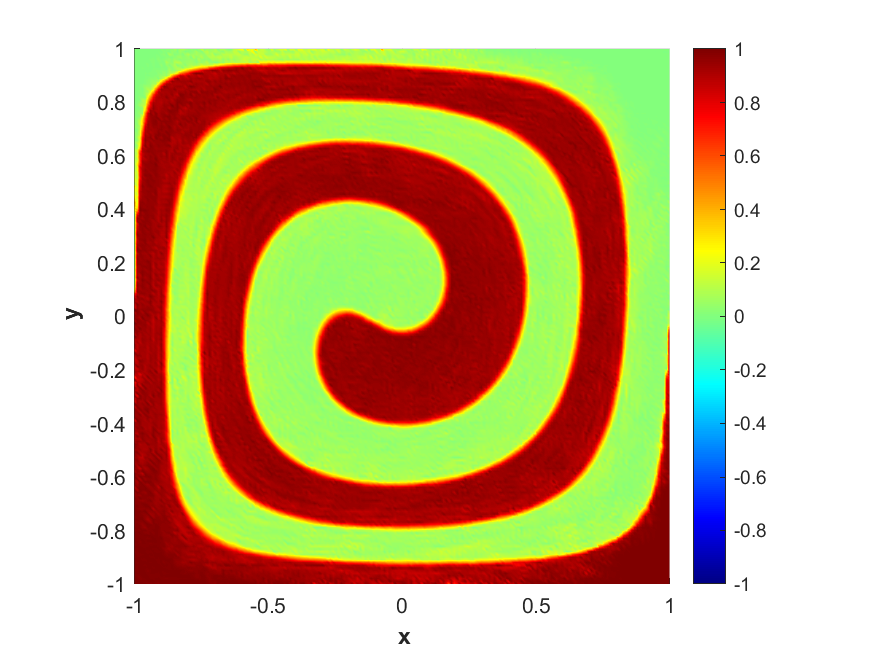}
    \label{fig:sub-T=15}
  }%
  \subfigure[$t=20$]{
    \includegraphics[width=0.175\textwidth]{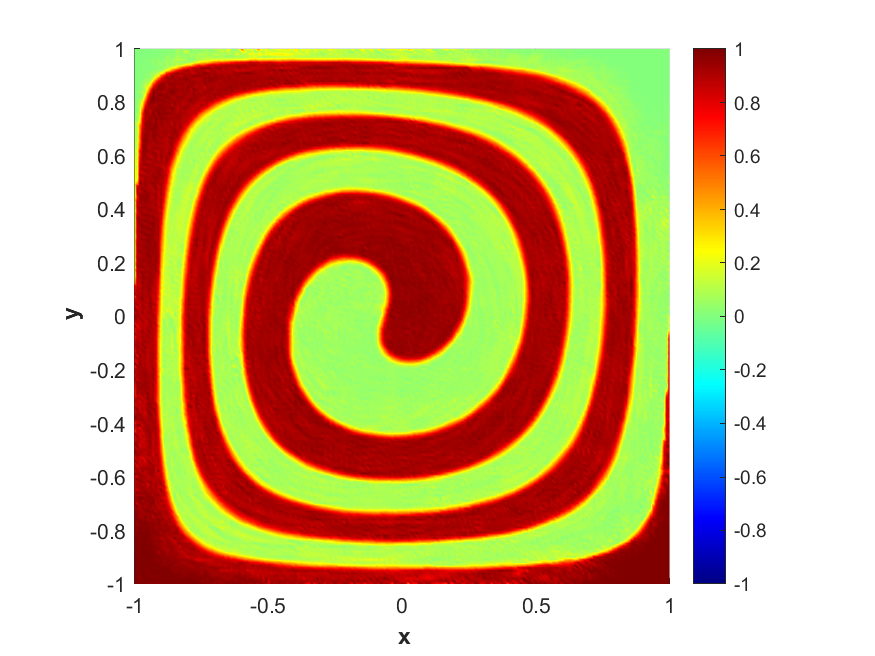}  
    \label{fig:sub-T=20}
  }%
  \subfigure[$t=25$]{
    \includegraphics[width=0.175\textwidth]{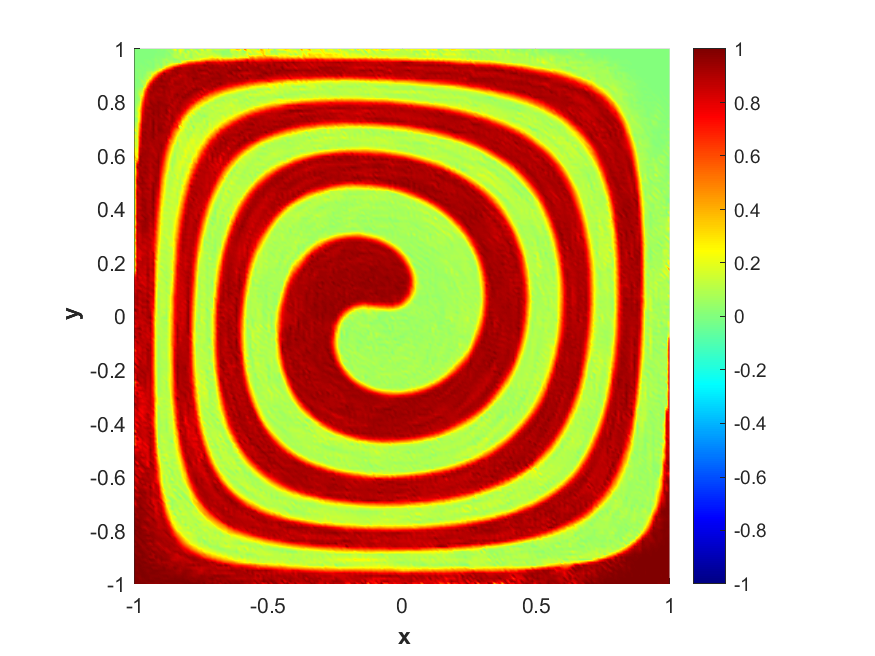}  
    \label{fig:sub-T=25}
  }\\%
  \subfigure[$t=30$]{
    \includegraphics[width=0.175\textwidth]{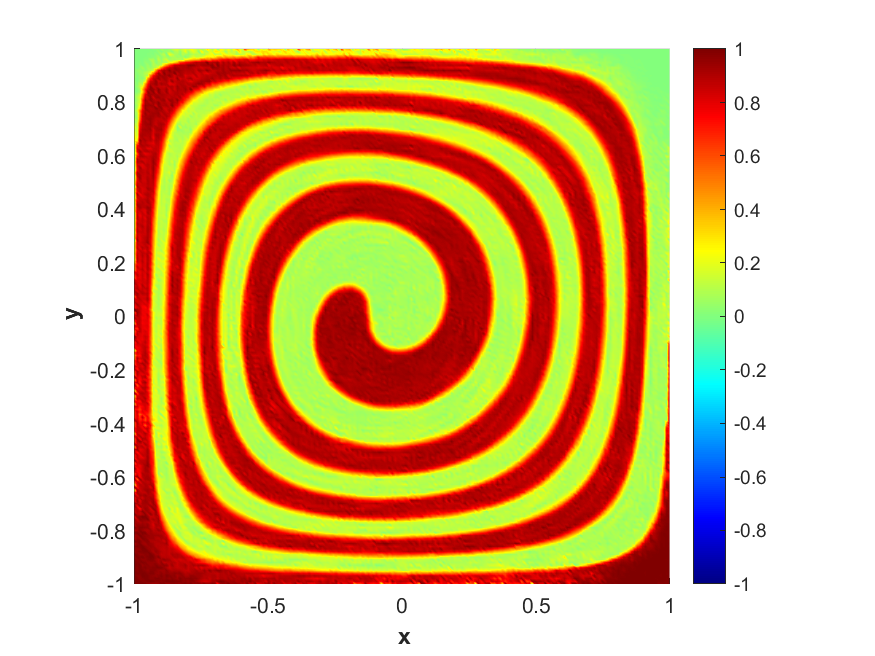}    
    \label{fig:sub-T=30}
  }
  \subfigure[$t=35$]{
    \includegraphics[width=0.175\textwidth]{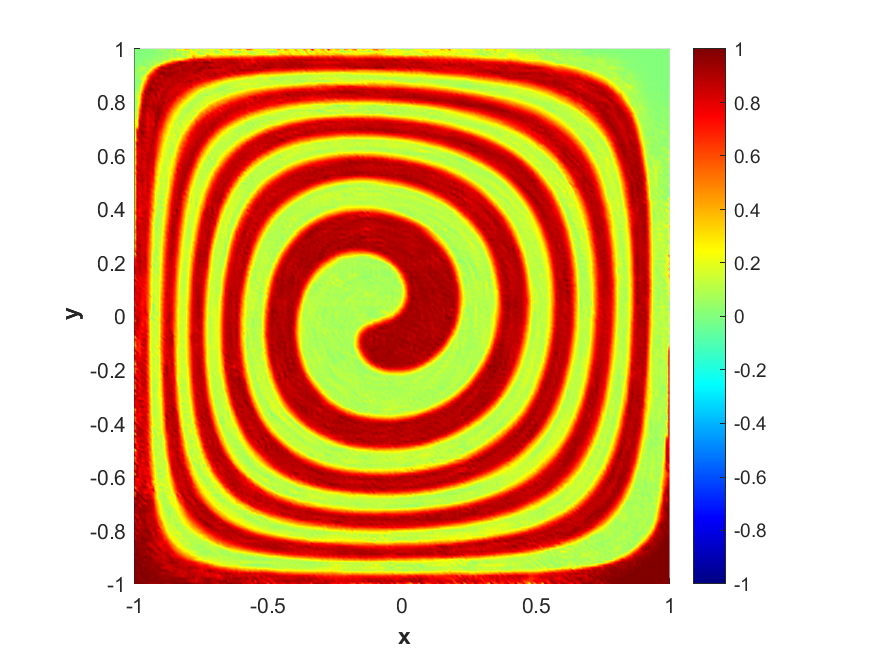}    
    \label{fig:sub-T=35}
  }%
  \subfigure[$t=40$]{
    \includegraphics[width=0.175\textwidth]{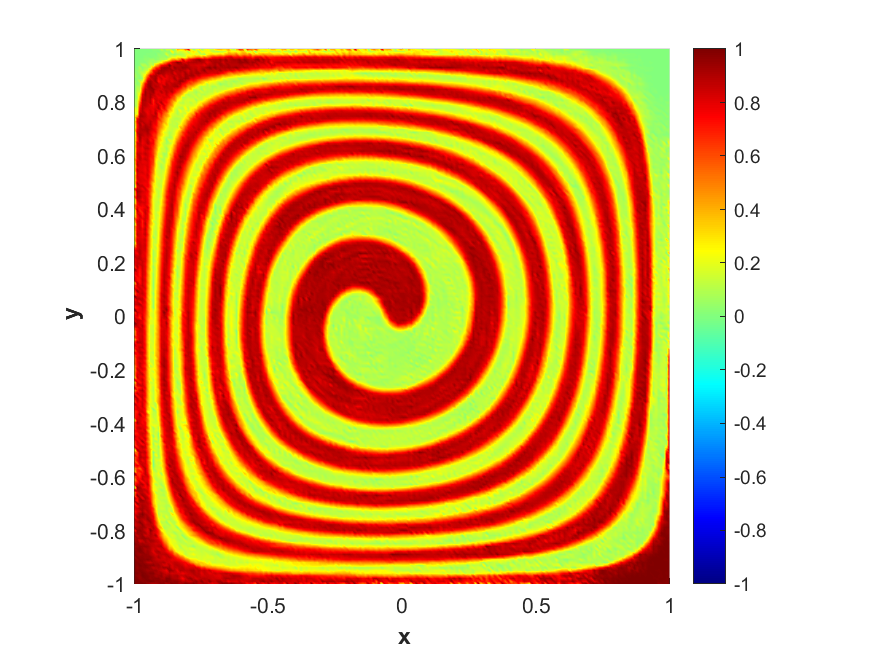}  
    \label{fig:sub-T=40}
  }%
  \subfigure[$t=45$]{
    \includegraphics[width=0.175\textwidth]{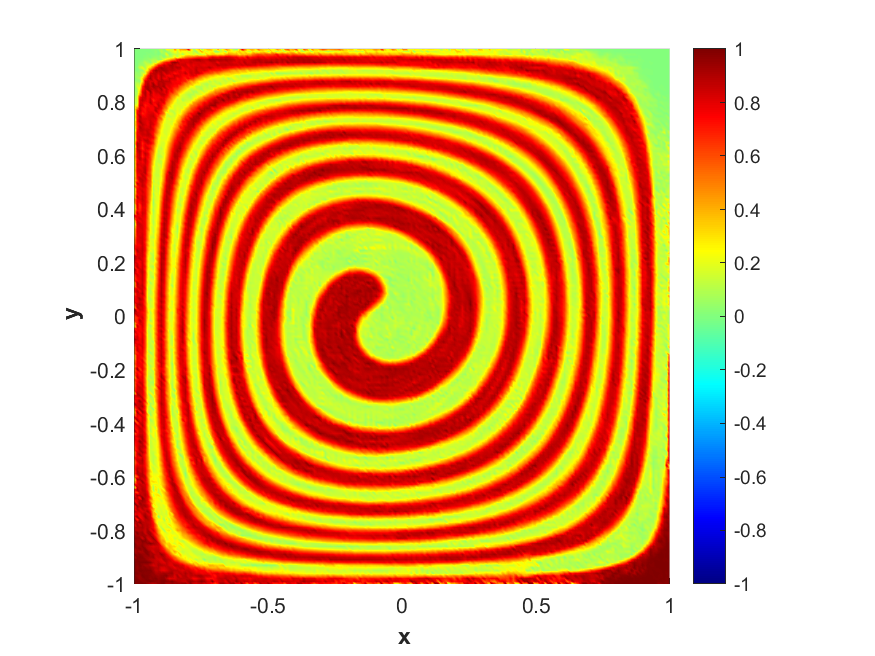}    
    \label{fig:sub-T=45}
  }%
  \subfigure[$t=50$]{
    \includegraphics[width=0.175\textwidth]{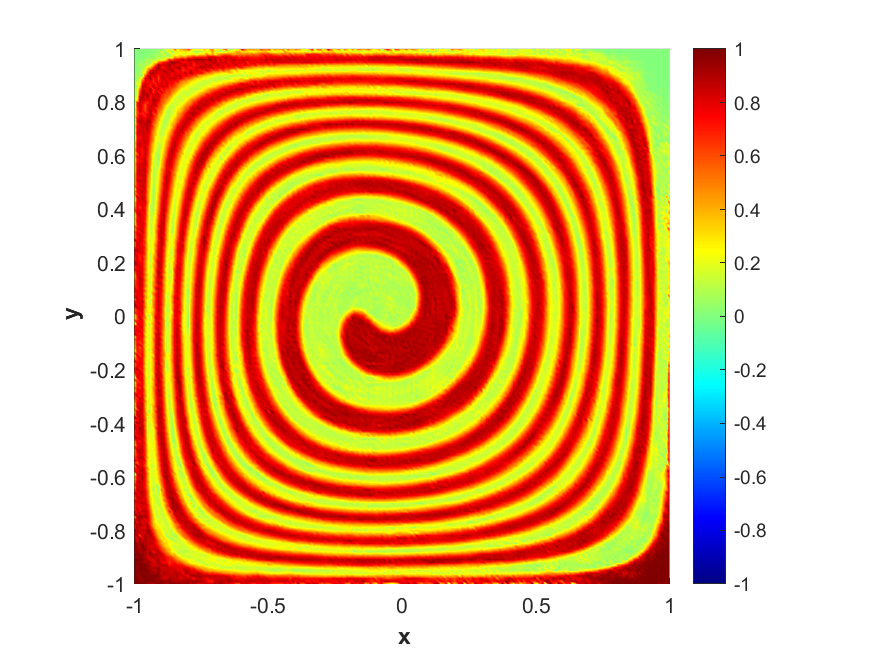}  
    \label{fig:sub-T=50}
  }
  \caption{
    Mixing evolution of a convected passive scalar \(\phi\) over time.
    Each subfigure shows the scalar field at a specific time \(t\).
    For better visual comparison, the color axis is fixed to [-1,1], although the actual scalar field\ $\phi$\ is constrained in [0, 1]}.
  \label{fig:total14}
\end{figure}

\subsection{Driven cavity flows in 2D}

 In this subsection, the stability of the second-order pressure projection scheme is validated through a lid-driven cavity flow in 2D with a simple heating distribution. The computational domain is defined as the unit square $\Omega = (0,1)^{2}$.  The initial conditions are given by $\omega_{0} = 0, \theta_0 = 0$ and $\mathbf{u}_0 = (0,0).$ The boundary conditions are given as follows: the top lid applies a horizontal velocity $\mathbf{u}_1=(1, 0)$, while the other boundaries adopt a no-slip boundary condition $\mathbf{u}=(0, 0)$.
 The angular velocity satisfies $\omega=0$; for the temperature, $\theta = 1.0$\ is prescribed on the right boundary, $\theta = 0$\ on the left boundary, and $\frac{\partial\theta}{\partial n}
= 0$\ on the upper and lower boundaries.  The spatial step is chosen as $h = 1/80$\ and the time step is $\delta_t = 10^{-3}$. The finite element spaces is chosen as $P2-P1$. In order to verify the influence of the top and thermal drives in the system, we set the following parameters. Since the buoyancy drive term $e\theta$ in the system, we use a normalization process. Using the system parameters, the Rayleigh number has the following expression $Ra=\frac{1}{(\chi+\mu)\kappa}$. In \Cref{fig:total5}–\Cref{fig:total8}, the physical parameters are set as $\upsilon =1.0, \kappa =0.01$, and $\chi=\mu=0.1, 0.01, 0.001$, and $0.0001$ respectively.
As shown in \Cref{fig:total5}, when \( Ra = 50 \) and \( Ra = 500 \), the thermal driving force is weak, and the flow is primarily dominated by the motion of the top lid. In contrast, for \( Ra = 5000 \) and \( Ra = 50000 \), the thermal effect competes with the lid-driven motion, leading to the emergence of a mixed flow pattern. This phenomenon is also clearly observed in the contour plots of pressure, angular velocity, and temperature. These results further confirm the effectiveness of the proposed second-order pressure projection scheme, as the velocity, angular velocity, pressure, and temperature fields all converge toward a steady state.

To further investigate the influence of microspin on the fluid behavior, we designed a series of numerical experiments with different sets of physical parameters: $\kappa =\chi = \mu = 0.1$, and $\nu =1.0, 0.01, 0.001$, and $0.0001$, respectively. From Figure~\ref{fig:total12}, it can be observed that for $\upsilon = 1.0$, the microspin effect strongly suppresses local variations in angular velocity, leading to a smooth angular velocity field with a very small amplitude, on the order of $10^{-3}$. As $\upsilon$ decreases, the response of the angular velocity becomes more pronounced: local peaks and shear-layer structures emerge, particularly near the wall, and the maximum amplitude of angular velocity increases with a sharper spatial distribution. This trend indicates that weaker microspin diffusion enhances the degrees of freedom of the angular motion and intensifies microrotation activity within the system, which is consistent with the fundamental predictions of micropolar fluid theory.

\begin{figure}[ht!]\label{Ta124}
  \centering
  \subfigure[$Ra = 50$]{
    \includegraphics[width=0.45\textwidth]{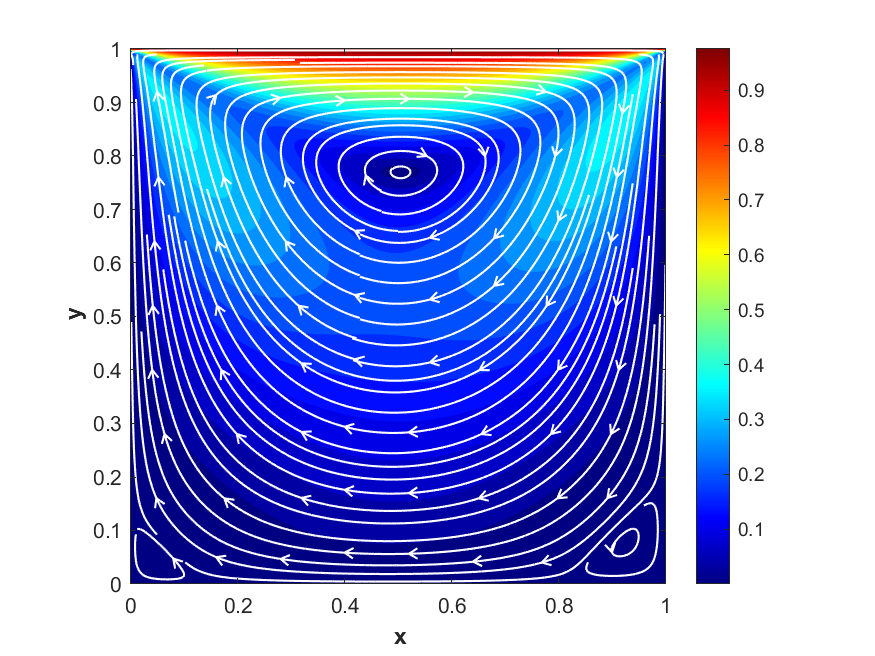}
    \label{fig:sub-u-Ra=50}
  }
  \subfigure[$Ra = 500$]{
    \includegraphics[width=0.45\textwidth]{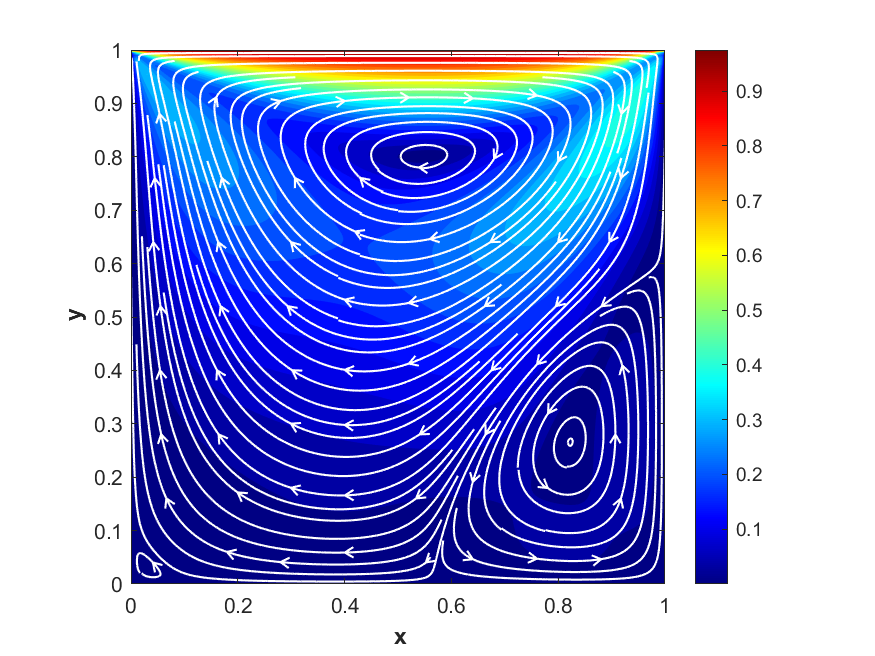}
    \label{fig:sub-u-Ra=500}
  }\\
  \subfigure[$Ra = 5000$]{
    \includegraphics[width=0.45\textwidth]{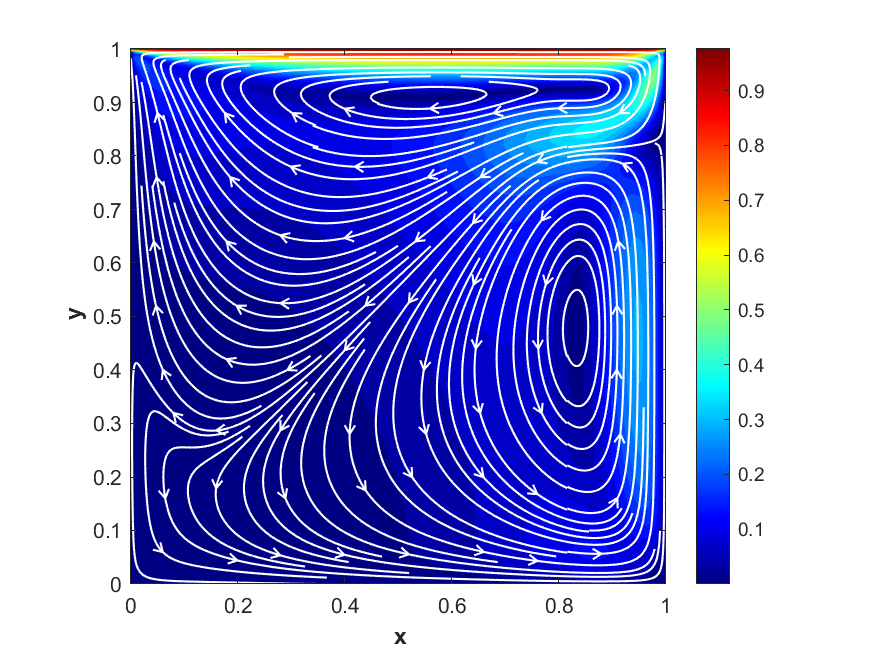}
    \label{fig:sub-u-Ra=5000}
  }
  \subfigure[$Ra = 50000$]{
    \includegraphics[width=0.45\textwidth]{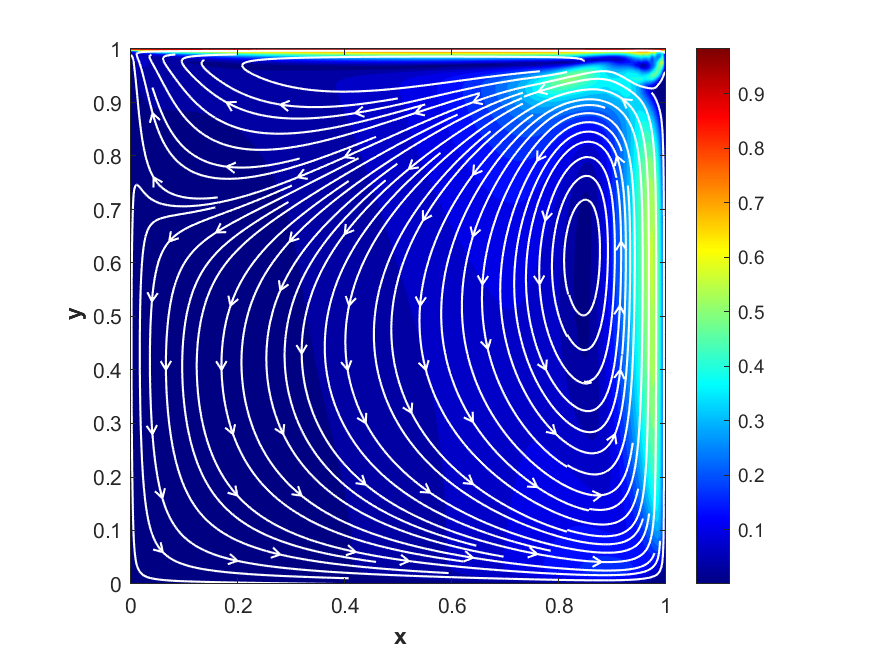}
    \label{fig:sub-u-Ra=50000}
  }
  \caption{Visualization of the numerical solution distribution for  streamlines.}
  \label{fig:total5}
\end{figure}

\begin{figure}[ht!]\label{Ta125}
  \centering
  \subfigure[$Ra = 50$]{
  \includegraphics[width=0.33\textwidth]{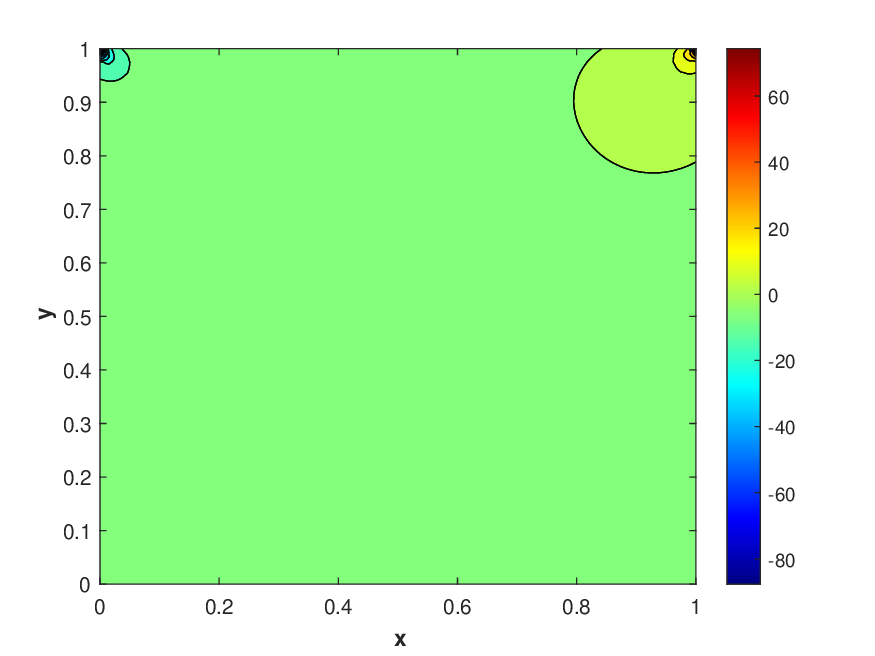}
    \label{fig:sub-p-Ra=50}
  }
  \subfigure[$Ra = 500$]{
    \includegraphics[width=0.33\textwidth]{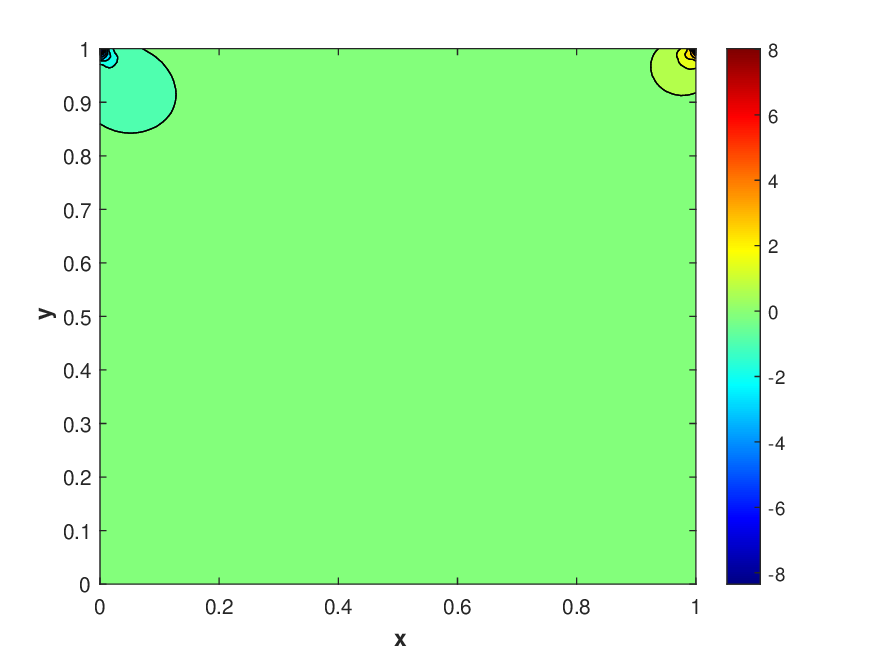}
    \label{fig:sub-p-Ra=500}
  }\\
  \subfigure[$Ra = 5000$]{
    \includegraphics[width=0.33\textwidth]{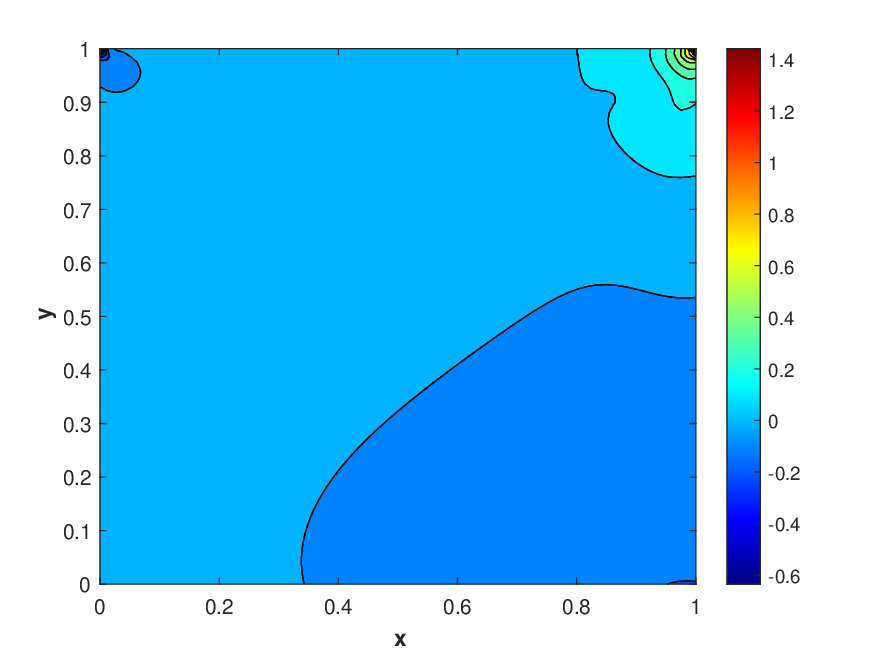}   
   \label{fig:sub-p-Ra=5000}
  }
  \subfigure[$Ra = 50000$]{
    \includegraphics[width=0.33\textwidth]{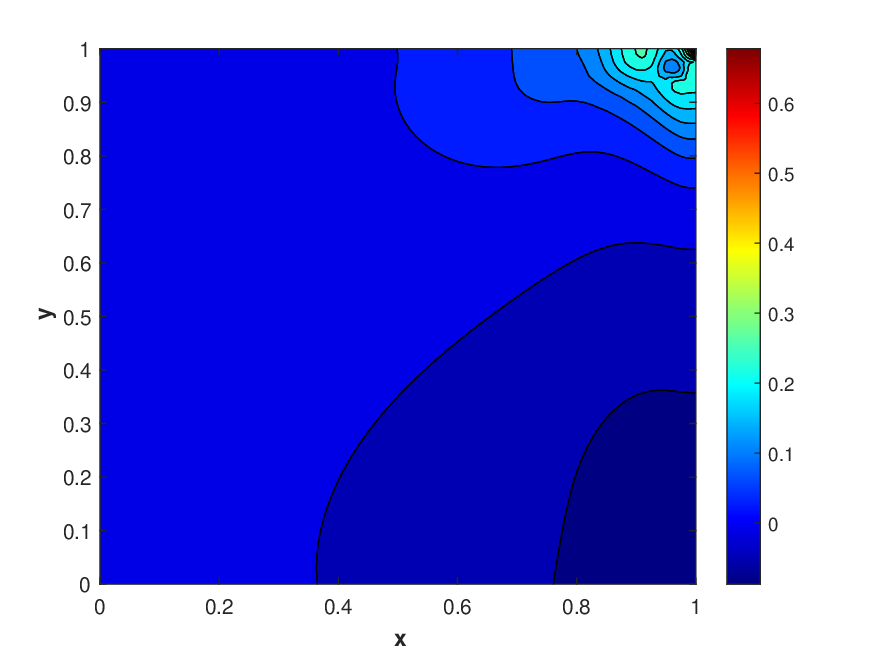}    
   \label{fig:sub-p-Ra=50000}
  }
  \caption{Visualization of the numerical solution distribution for  pressure.}
 \label{fig:total6}
\end{figure}

\begin{figure}[ht!]\label{Ta126}
  \centering
  \subfigure[$Ra = 50$]{
    \includegraphics[width=0.33\textwidth]{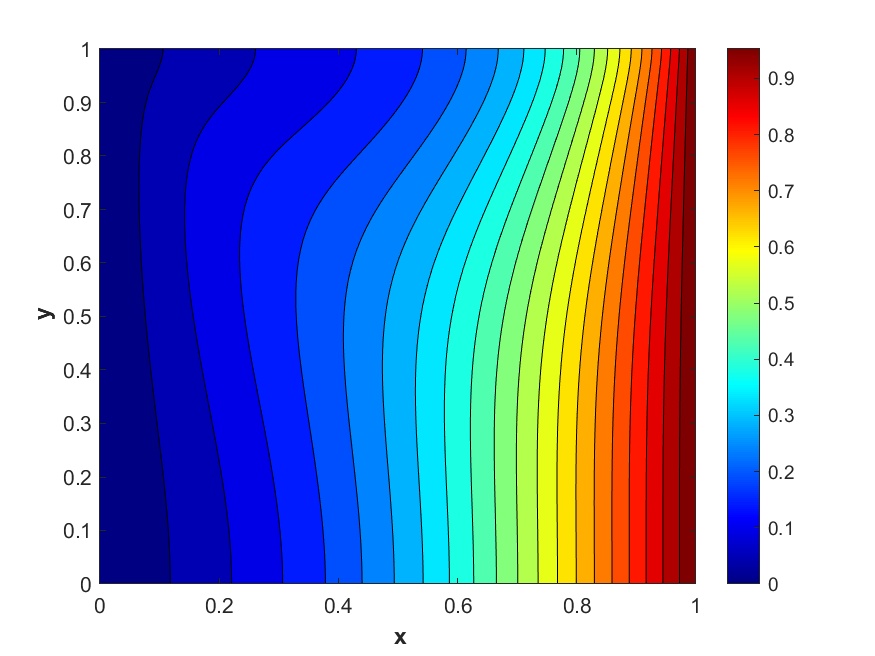}   
    \label{fig:sub-theta-Ra=50}
  }
  \subfigure[$Ra = 500$]{
   \includegraphics[width=0.33\textwidth]{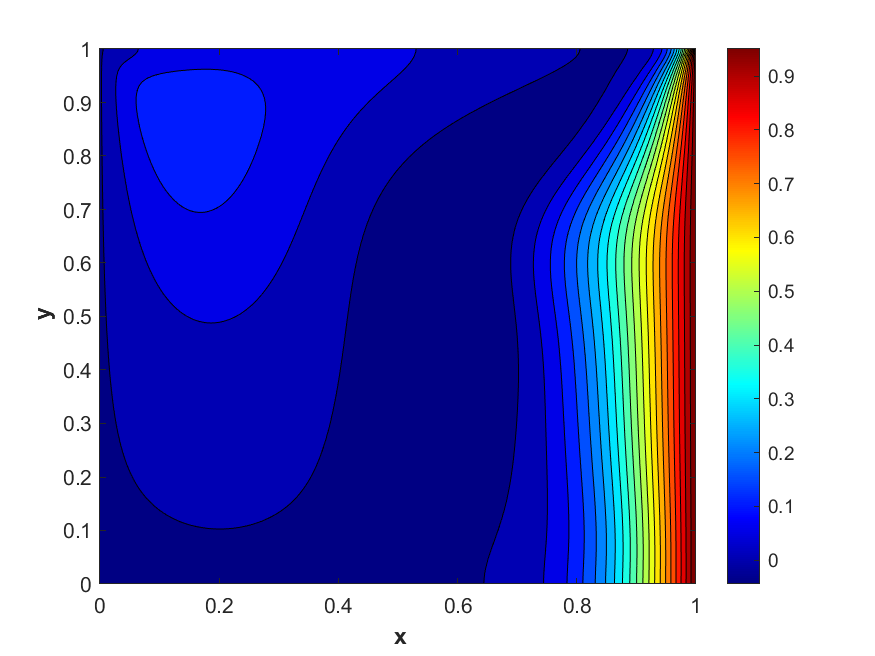}   
    \label{fig:sub-theta-Ra=500}
  }\\
  \subfigure[$Ra = 5000$]{
    \includegraphics[width=0.33\textwidth]{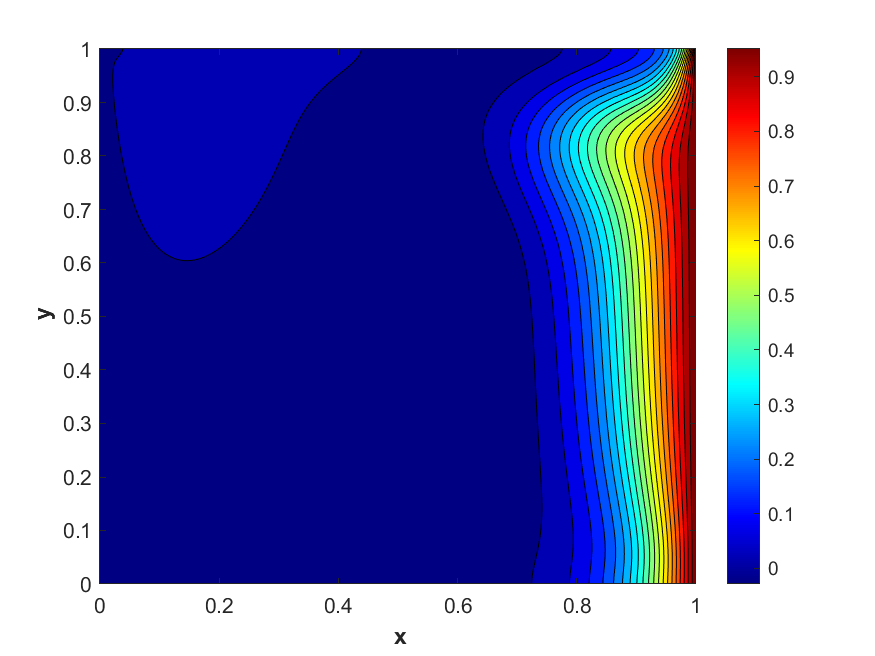}    
    \label{fig:sub-theta-Ra=5000}
  }
  \subfigure[$Ra = 50000$]{
    \includegraphics[width=0.33\textwidth]{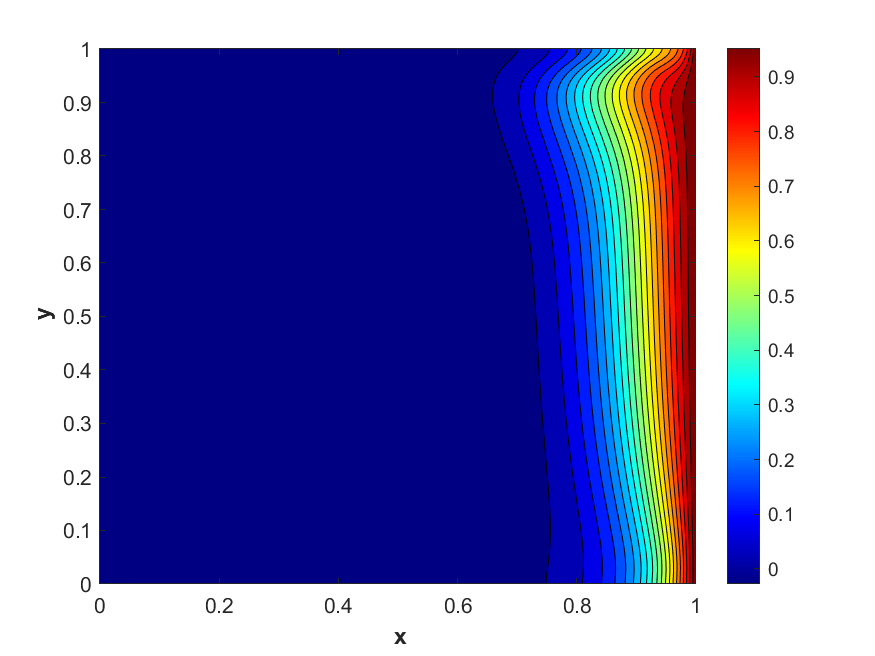}   
    \label{fig:sub-theta-Ra=50000}
  }
  \caption{Visualization of the numerical solution distribution for  temperature.}
  \label{fig:total7}
\end{figure}

\begin{figure}[!htbp]\label{Ta127}
  \centering
  \subfigure[$Ra = 50$]{
    \includegraphics[width=0.33\textwidth]{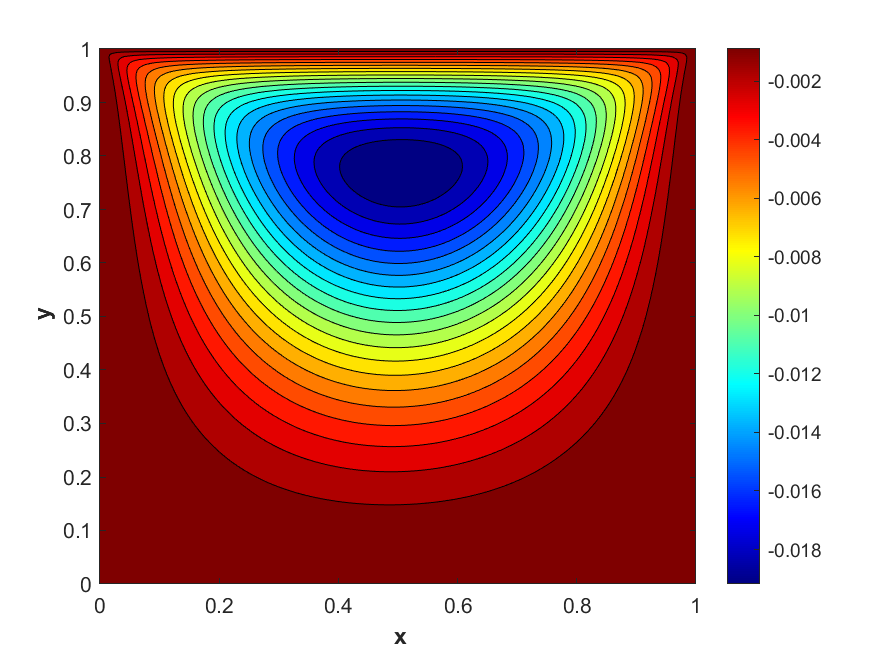}  
    \label{fig:sub-w-Ra=50}
  }
  \subfigure[$Ra = 500$]{
    \includegraphics[width=0.33\textwidth]{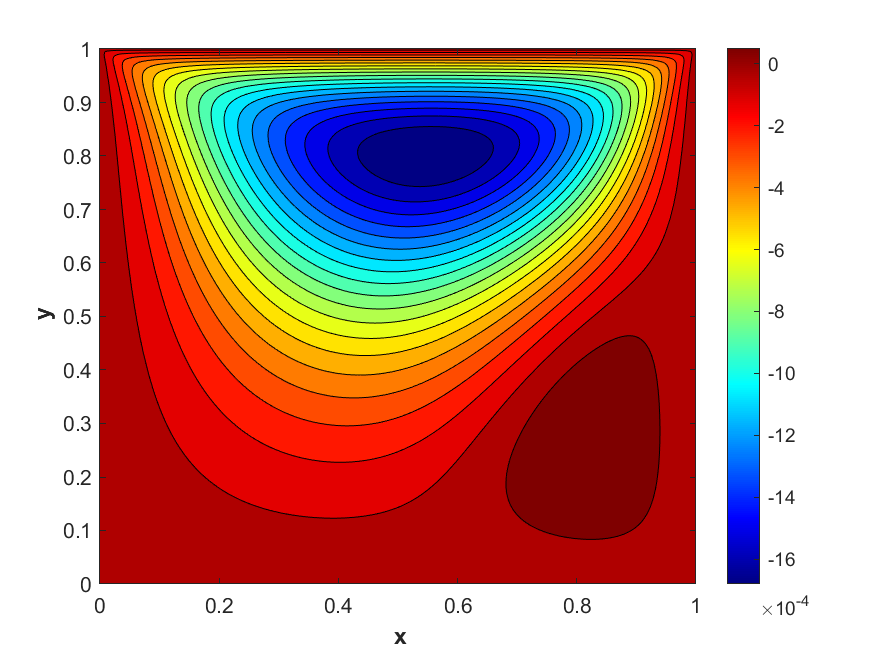}   
    \label{fig:sub-w-Ra=500}
  }\\
  \subfigure[$Ra = 5000$]{
    \includegraphics[width=0.33\textwidth]{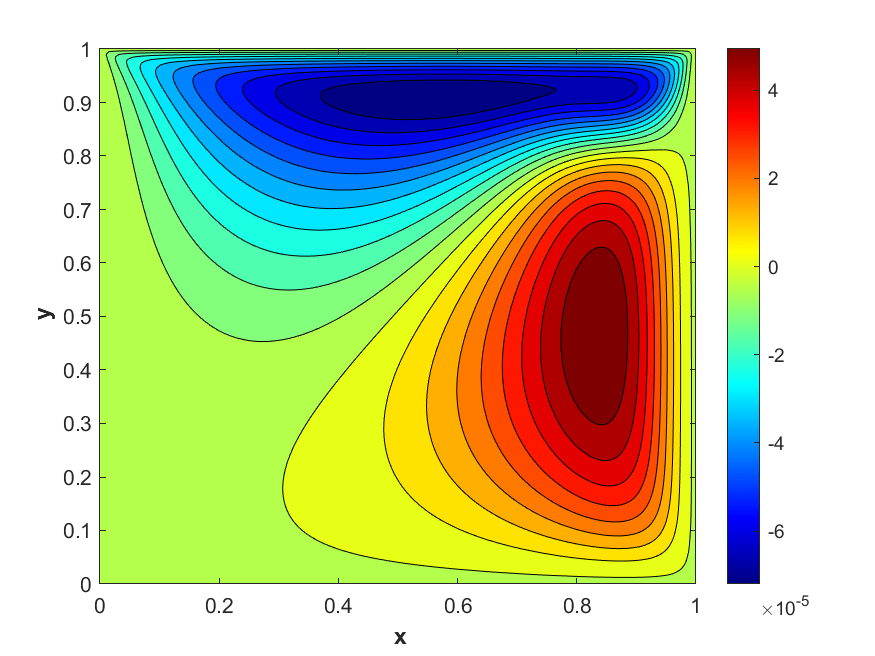}    
    \label{fig:sub-w-Ra=5000}
  }
  \subfigure[$Ra = 50000$]{
    \includegraphics[width=0.33\textwidth]{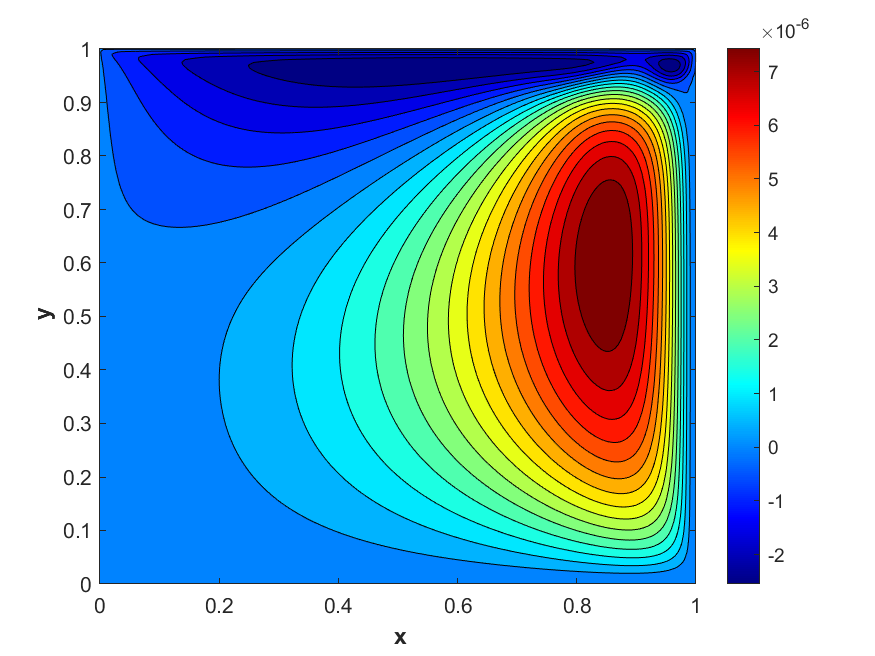}   
    \label{fig:sub-w-Ra=50000}
  }
  \caption{Visualization of the numerical solution distribution for angular velocity.}
  \label{fig:total8}
\end{figure}

\begin{figure}[!htbp]
  \centering
  \subfigure[$\upsilon=1$]{
    \includegraphics[width=0.33\textwidth]{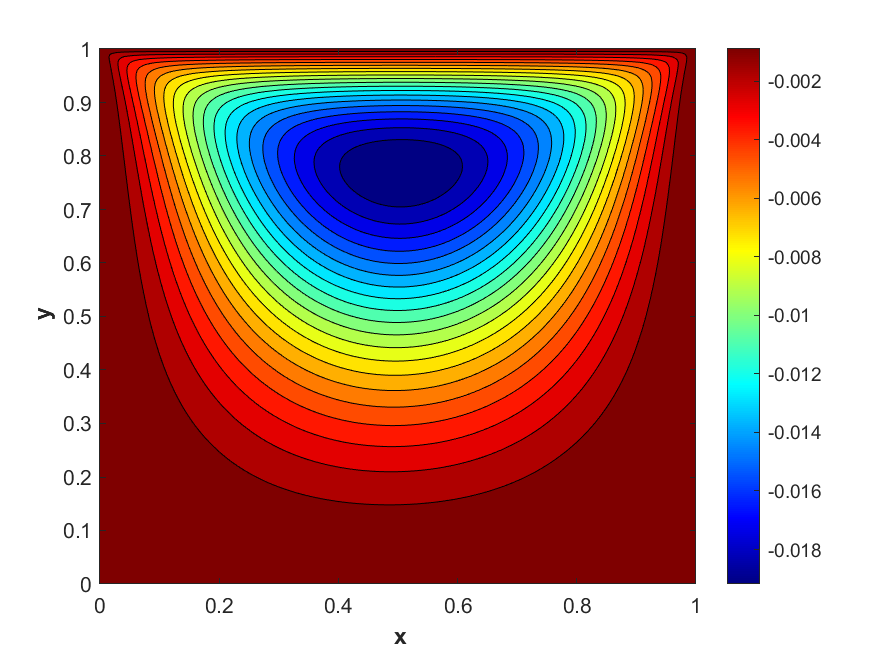}
    
    \label{fig:sub-w-nu=1}
  }
  \subfigure[$\upsilon=0.01$]{
    \includegraphics[width=0.33\textwidth]{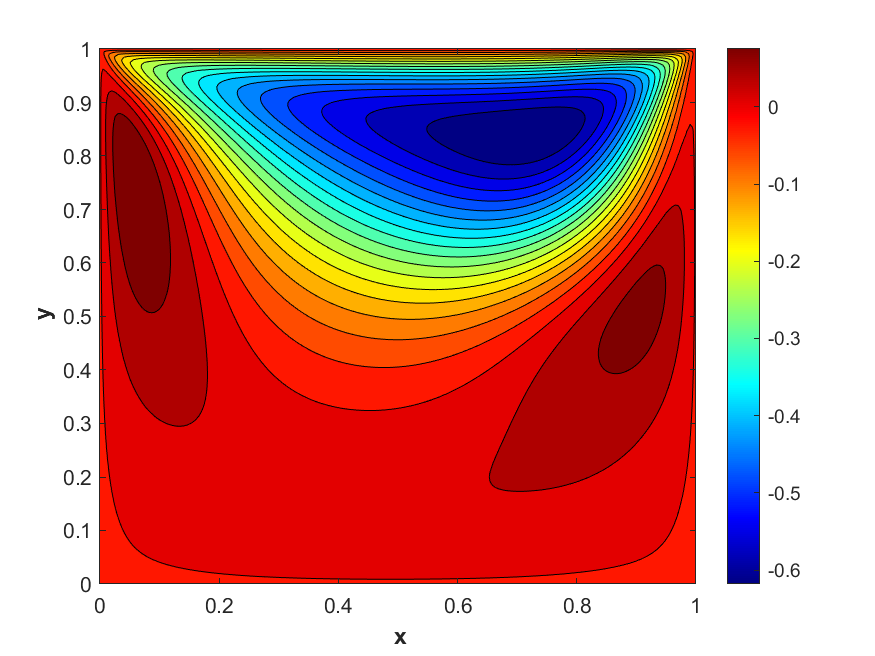}
    \label{fig:sub-w-nu=0.011}
  }\\
  \subfigure[$\upsilon=0.001$]{
    \includegraphics[width=0.33\textwidth]{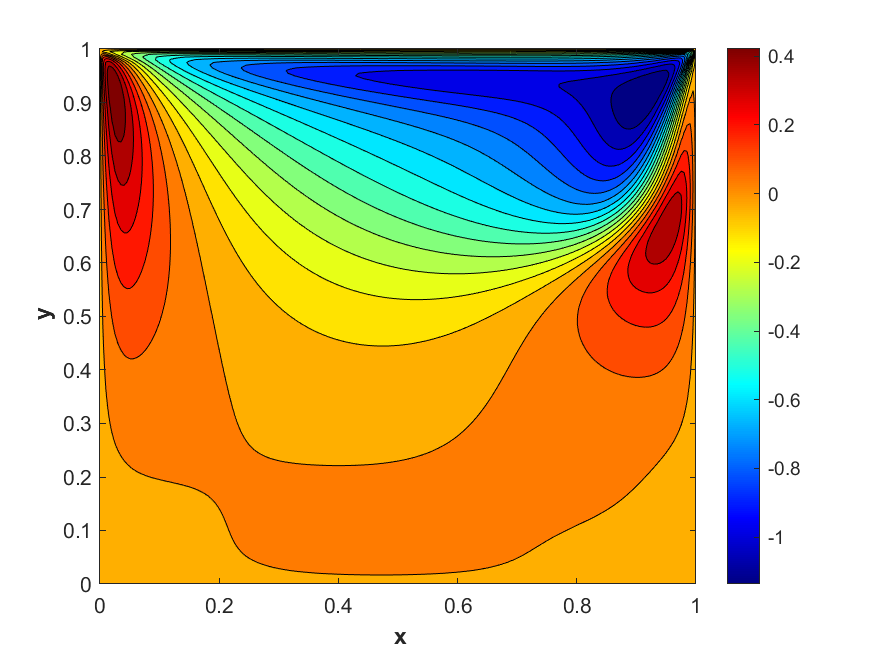}  
    \label{fig:sub-w-nu=0.001}
  }
  \subfigure[$\upsilon=0.0001$]{
    \includegraphics[width=0.33\textwidth]{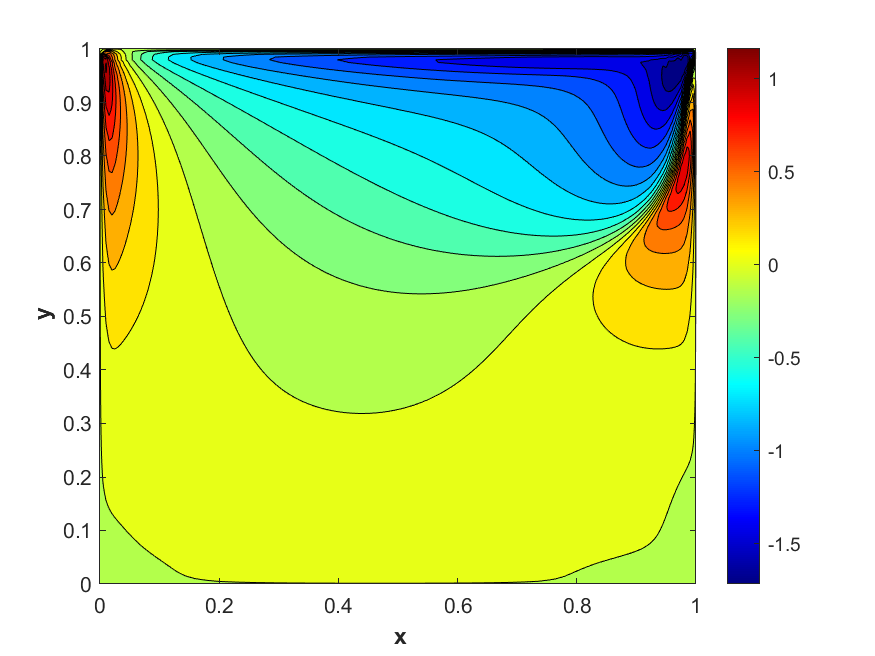}
    \label{fig:sub-w-nu=0.0001}
  }
  \caption{Visualization of the numerical solution distribution for angular velocity.}
  \label{fig:total12}
\end{figure}
\clearpage

%% file: Conclusion.tex
\section{Conclusions}

The micropolar Rayleigh-B{\'e}nard convection system is a strongly coupled system, resulting in expensive computations for traditional coupled algorithms. With spatial discretization of the finite element method, a numerical scheme possessing the favorable properties of linearity, full decoupling, and second-order time accuracy is constructed. Based on a second-order projection technique for the hydrodynamic component, the scheme exhibits significant computing efficiency with only a few independent linear elliptic subproblems with constant coefficients at each time step. This work not only conducts a rigorous stability analysis, but also employs the inverse Stokes operator and negative-norm estimates to effectively address the convergence-order reduction caused by the pressure term, thereby achieving optimal error estimates. The effectiveness and robustness of the scheme are presented with numerical experiments, including precision tests, temperature-driven cavity flow simulations, and the stirring of a passive scalar experiment. In future research, the model will be extended to the magneto-micropolar fluids, where the Maxwell equations are coupled to describe the influence of external magnetic and electric fields on micro-particle rotation. This coupling is expected to enable active control of microrotation dynamics through electromagnetic fields.